\documentclass[a4paper,11pt]{amsart}
\usepackage[utf8]{inputenc}
\usepackage[T1]{fontenc}
\usepackage[UKenglish]{babel}
\usepackage[margin=15mm]{geometry}

\usepackage{color}
\usepackage{marginnote}
\usepackage{graphicx}

\usepackage{amsmath,amssymb,amsfonts,amsthm}
\usepackage{mathrsfs,eucal,dsfont}
\usepackage{verbatim,enumitem}

\usepackage{hyperref,url}

\renewcommand{\le}{\leqslant}
\renewcommand{\ge}{\geqslant}

\definecolor{mno}{rgb}{0.5,0.1,0.5}

\newcommand{\R}{\mathds R}
\newcommand{\B}{\beta}
\newcommand{\T}{\theta}

\newcommand{\Pp}{\mathds P}
\newcommand{\Ee}{\mathds E}

\newcommand{\I}{\mathds 1}
\newcommand{\p}{\preceq}
\newcommand{\q}{\succeq}
\newcommand{\K}{\kappa}
\newcommand{\A}{\alpha}
\newcommand{\G}{\gamma}

\newcommand{\Bb}{\mathscr{B}}

\newcommand{\e}{ \varepsilon}

\def\LL{{\mathcal L}}
\def\wt{\widetilde}
\def\wh{\widehat}

\newtheorem{theorem}{Theorem}[section]
\newtheorem{lemma}[theorem]{Lemma}
\newtheorem{proposition}[theorem]{Proposition}
\newtheorem{corollary}[theorem]{Corollary}

\theoremstyle{definition}

\newtheorem{remark}[theorem]{Remark}

\makeatletter 

\@addtoreset{equation}{section}
\makeatother 

\begin{document}
\allowdisplaybreaks
\title[Heat kernel for non-local operators with variable order]{
Heat kernel for non-local operators with variable order}

\author{{Xin Chen}, \quad {Zhen-Qing Chen} \quad \hbox{and} \quad {Jian Wang}}
\thanks{\emph{X.\ Chen:}
   Department of Mathematics, Shanghai Jiao Tong University, 200240, Shanghai, P.R. China. \url{chenxin217@sjtu.edu.cn}}
  \thanks{\emph{Z.-Q.\ Chen:}
   Department of Mathematics, University of Washington, Seattle,
WA 98195, USA. \url{zqchen@uw.edu}}

  \thanks{\emph{J.\ Wang:}
 College of Mathematics and Informatics \& Fujian Key Laboratory of Mathematical Analysis and Applications (FJKLMAA), Fujian Normal University, 350007 Fuzhou, P.R. China. \url{jianwang@fjnu.edu.cn}}

\maketitle

\begin{abstract} Let $\alpha(x)$ be a measurable function taking values in  $ [\alpha_1,\alpha_2]$ for  $0<\A_1\le \A_2<2$, and
$\kappa(x,z)$ be a positive measurable function that is symmetric in $z$ and bounded
between two positive constants.  Under  a uniform H\"{o}lder continuous assumptions on
$\alpha(x)$ and $x\mapsto \kappa(x,z)$, we obtain  existence, upper and
lower bounds, and regularity properties of the heat kernel associated
with the following non-local operator of variable order
$$
\LL f(x)=\int_{\R^d}\big(f(x+z)-f(x)-\langle\nabla f(x), z\rangle \I_{\{|z|\le
1\}}\big) \frac{\kappa(x,z)}{|z|^{d+\alpha(x)}}\,dz.
$$
In particular, we  show
that the operator $\LL$   generates a conservative Feller process on $\R^d$ having the strong Feller property,
which is usually assumed a priori  in the literature
to study analytic properties of $\LL$ via probabilistic approaches.  Our near-diagonal
estimates and lower bound estimates of the heat kernel depend on the local behavior of index function $\alpha(x)$; when
$\alpha(x)\equiv \A\in(0,2)$, our results recover  some results
 by Chen and Kumagai (2003) and Chen and Zhang (2016).

\medskip


\medskip

\end{abstract}

\section{Introduction and Main Results}
\subsection{Setting}
Let $\LL$  be a non-local operator with variable order given by
\begin{equation}\label{e1-1}
\LL f(x)=\int_{\R^d}\big(f(x+z)-f(x)-\langle\nabla f(x), z\rangle \I_{\{|z|\le 1\}}\big)
\frac{\kappa(x,z)}{|z|^{d+\alpha(x)}}\,dz,\quad   f \in C_c^2(\R^d),
\end{equation}
where $\alpha:\R^d \rightarrow (0,2)$
is a measurable function such that
\begin{equation}\label{e1-3}
\begin{split}
& 0<\alpha_1\le \alpha(x) \le \alpha_2<2  \quad \hbox{for }   x \in \R^d,\\
& |\A(x)-\A(y)|\le c_1(|x-y|^{\beta_1}\wedge 1)   \quad \hbox{for }   x,y\in \R^d,
\end{split}
\end{equation}
for  some constants $c_1\in(0,\infty)$ and $\beta_1\in(0,1]$, and $\kappa:\R^d\times \R^d \rightarrow (0,\infty)$ is a measurable function satisfying
\begin{equation}\label{e1-2}
\begin{split}
& \K(x,z)=\K(x,-z) \quad \hbox{for }  x,z \in \R^d,\\
&  0<\K_1\le \K(x,z)\le \K_2<\infty \quad \hbox{for }   x,z\in \R^d,\\
&  |\K(x,z)-\K(y,z)|\le c_2(|x-y|^{\beta_2}\wedge 1)  \quad \hbox{for }   x,y,z\in \R^d,
\end{split}
\end{equation}
for  some constants $c_2\in(0,\infty)$ and $\beta_2 \in (0,1]$.
Set $\beta_0= \beta_1\wedge \beta_2$. Clearly \eqref{e1-3} and \eqref{e1-2} hold with $\beta_0$ in place of $\beta_1$ and $\beta_2$, respectively.

When
$\A(x)\equiv \A \in (0,2)$ and
$\K(x,z)\equiv \alpha2^{\alpha-1}\, \frac{\Gamma((\alpha+d)/2))}{\pi^{d/2}\,\Gamma(1-(\alpha/2))}$
for all $x,z\in\R^d$,
$\LL $ of \eqref{e1-1} is the fractional Laplacian operator $-(-\Delta)^{\A/2}$, which is the infinitesimal generator of the rotationally symmetric $\A$-stable L\'evy process
in $\R^d$.
Therefore $\LL$ could be viewed as a generalization of fractional Laplacian operator with
jump coefficient $\K(x,z)$ and  of variable order  $\A(x)$.

\subsection{Background}
Non-local operators  arise naturally in the study of stochastic processes with jumps. Various properties of the non-local operator $\LL$ have been intensively investigated both from the analytic and the probabilistic point of view.

We first recall some known results for the constant order case,
i.e.\ $\A(x) \equiv \A \in (0,2)$ for all $x\in\R^d$. The regularity
properties of the operator $\LL$ including the H\"{o}lder continuity
of harmonic functions, the Schauder estimates and the $L^p$
estimates were studied in \cite{Bass2,CS1,CS,CV,CRZ,S1,S2,Z}. When
the operator $\LL$ is symmetric with respect to the Lebesgue measure,
existence and two-sided estimates of the heat kernel associated with $\LL$
were obtained in \cite{CK,CK1}, and
the corresponding parabolic Harnack inequality was also established
there. For general $\K(x,z)$ satisfying \eqref{e1-2}, the associated
heat kernel was constructed by Levi's method in \cite{CZ}, where the
corresponding two-sided estimates and gradient estimates were given,
and the parabolic equation (see \eqref{t1-1-1} below) was also
verified. Recently the extension of \cite{CZ}
from the L\'evy kernel $1/|z|^{d+\alpha}$ to that of a  class of subordinate Brownian motion
was presented in  \cite{KSV}. See also
\cite{BSK, GS, Kim, KK1,KK2, Ku, KR17,Kul, SZ} for the construction of certain L\'evy
type processes via Levi's  (parametrix)
 method. Besides, existence and uniqueness
(i.e.\ well-posedness) of the martingale problem for $\LL$ were proved
in \cite{BT, CZ4, J} under some mild continuous conditions on $\K(x,z)$.

When the index function $\A(x)$ depends on $x $ and the jump
coefficient $\kappa(x,z)=\frac{\alpha(x)2^{\alpha(x)-1}\Gamma
((\alpha(x)+d)/2)}{\pi^{d/2}\,\Gamma(1-(\alpha(x)/2))}$ for all
$x,z\in \R^d$, $\LL$ can be written as a pseudo-differential operator
$-p(x,D)$ with the symbol $p(x, \xi)= -|\xi|^{\alpha(x)}$; that is,
$$
\LL f(x) = \frac{1}{(2\pi)^d} \int_{\R^d} e^{\langle ix, \xi\rangle} p(x, \xi) \wh  f(\xi)\, d\xi
= \frac{1}{(2\pi)^d} \int_{\R^d\times \R^d} e^{i \langle x-y, \xi\rangle} p(x, \xi) f(y) \,dy \,d\xi .
$$
Formally, we can write $\LL$ as $-(-\Delta)^{\alpha (x)/2}$.
For this reason, we call $\alpha(x)$ the order of the non-local operator at the point $x\in \R^d$.
Existence and well-posedness of
the martingale solution for $\LL= -(-\Delta)^{\alpha (x)/2}$ were established in Bass \cite{Bass1}, which
immediately yields that there is  a strong Markov process corresponding to $-(-\Delta)^{\alpha (x)/2}$.
(This process was called the stable-like process in \cite{Bass1}). A few
properties for such stable-like process (including existence of heat
kernel, sample path properties and the long-time behaviors) have
been investigated in \cite{BSW,Hoh,KN,Neg,NS, Sch} by using the theory of
pseudo-differential operators and Fourier analysis. When $\kappa (x,
z)=\kappa (x)$ is independent of $z$ and both $\alpha (x)$ and
$\kappa (x)$ have uniformly bounded continuous derivatives, the heat
kernel for $\LL$ was constructed formally in \cite[Section 5]{Kol} via
 Duhamel's formula, and some upper bound estimates for heat kernel were
also given there. Duhamel's formula has also been adopted in
\cite{CW,W} to study the foundational solution to fractional
Laplacian or Laplacian perturbed by non-local operators. We also
mention that there are some results for the case that $\K(x,z)$
depends on $z$. For instance, the regularity for harmonic functions
or the semigroups associated with $\LL$ were studied in
\cite{BK1,LW,S1}, and the elliptic Harnack inequality was obtained
in \cite{BK2}.

From all the results mentioned above, we can see that there are
already a lot of developments related to   non-local operators of
variable order. However, the following questions, which should be
fundamental and interesting, were still unknown.
\begin{itemize}
\item [(1)] When $\A(x)$ is not a constant, how can we construct a fundamental solution of $\LL$ that really satisfies
 the parabolic equation \eqref{t1-1-1} below?  What are upper
and lower bound estimates for this solution (if it exists)? We point out that it had not been established that,
even under additional smoothness assumptions on $\alpha (x)$ and $\kappa (x)$,
the  heat kernel constructed  in \cite{Kol}  by Duhamel formula is actually a solution to the equation \eqref{t1-1-1}.

\item[(2)] When $\A(x)$ is not a constant and $\K(x,z)$ depends on $z$, is there a strong Markov process
associated with $\LL$?  In literature, the existence of strong Markov process was
always assumed a priori  in the study of regularity of harmonic
functions and elliptic Harnack inequalities for such non-local operator $\LL$,
see \cite{BK1,BK2, LW} for examples.

\item[(3)] In existing literature,
  regularity for the solution to parabolic equation \eqref{t1-1-1} associated
with $\LL$ usually depends on
 the uniform bounds $\A_1$ and $\A_2$ in
\eqref{e1-3}.
Can one establish regularity of the fundamental solution in some neighborhood of  $x_0$ (such  as  gradient estimate) in terms of $\A(x_0)$ under some special settings?
\end{itemize}

\subsection{Main results} The aim of this paper is to address these questions.
We show that there is a Feller process associated with $\LL$ and it has the strong Feller property.
We give existence and uniqueness of fundamental solutions to $\LL$ as well as their regularity properties.

For $x_0\in \R^d$ and $r>0$, define  $B(x_0, r):=\{x\in \R^d: |x-x_0|<  r_0\}$.
Let $C_{b,u}(\R^d)$ denote the set of all bounded and uniformly continuous functions on $\R^d$.
The following are two  of the main results of our paper.

\begin{theorem}\label{t1-1}
Suppose that conditions \eqref{e1-3} and \eqref{e1-2}   hold. If $\K(x,z)= \kappa (x)$ is
independent of $z$,
then there exists a
  jointly continuous non-negative
function
$p:(0,1]\times \R^d
\times \R^d \rightarrow \R_+:=[0,\infty)$ such that for every $t \in (0,1]$ and $x,y \in \R^d$,
\begin{equation}\label{t1-1-1}
\frac{\partial p(t,x,y)}{\partial t}=\LL p(t,\cdot,y)(x),
\end{equation}
and has the following properties.
\begin{enumerate}
\item[(i)]{\bf (Upper bounds)}\,\,
For every $\gamma, c_0>0$, there exist positive constants $c_1 =c_1(\A,\K,c_0)$ and $c_2 =c_2(\A,\K,\G,c_0)$ such that
for all $t\in(0,1]$ and $x,y\in\R^d$,
\begin{equation}\label{t1-1-4}
p(t,x,y)\le
\begin{cases}
{c_1}{t^{-{d/\A(x)}}},&\quad |x-y|< c_0t^{{1/\A(x)}},\\
\frac{c_1t}{|x-y|^{d+\A_2}}\wedge \frac{c_2t^{1-\gamma}}{|x-y|^{d+\A(x)}},&\quad
c_0t^{{1/\A(x)}}\le |x-y|< 1,\\
\frac{c_1t}{|x-y|^{d+\A_1}},&\quad|x-y|\ge 1.
\end{cases}
\end{equation}
If, in addition,
there are some $x_0\in\R^d$ and $r_0\in(0,\infty]$ such that $\alpha(z)=\alpha(x_0)$ for all $z \in B(x_0,r_0)$,
then there is a positive constant
$c_3 =c_3(\A,\K,r_0)$ such that for every $t\in (0,1]$ and $y \in \R^d$,
\begin{equation}\label{t1-1-4a}
p(t,x_0,y)\le
\begin{cases}
\frac{c_3 t }{(t^{{1/\A(x_0)}}+|y-x_0|)^{d+\A(x_0)}},&
|y-x_0|< {r_0}/{2},\\
\frac{c_3 t }{|y-x_0|^{d+\A_1}},&|y-x_0|\ge {r_0}/{2};
\end{cases}
\end{equation}

\noindent {\bf (Lower bounds)}\,\, There exists a constant $c_4 =c_4(\A,\K)>0$ such that for all $t\in(0,1]$ and $x,y\in\R^d$,
\begin{equation}\label{t1-1-4b}
p(t,x,y)\ge \frac{c_4t}{(t^{1/\A(x)}+|x-y|)^{d+\A(x)}}.
\end{equation}

\item[(ii)]{\bf (H\"{o}lder regularity and gradient estimates)}\,\,
For any
$\G \in (0, \alpha_2]$,
 there exist positive constants
$c_5 =c_5(\A,\K,\G)$ and $R_1=R_1(\A,\K,\G)$ such that for all $t \in (0,1]$ and $x, x', y \in \R^d$ with $|x-x'|\le R_1$,
\begin{equation}\label{t1-1-7}
\begin{split}
 \big|p(t,x,y)-p(t,x', y)\big|
&\le c_5
|x-x'|^{(\A(x)-\G)_+\wedge 1}
\big(\rho^{y,0}_{\gamma_0}(t,x-y)+\rho^{y,0}_{\gamma_0}(t,x'-y)\big),
\end{split}
\end{equation}
where $\gamma_0= {\G/(2\A_2)}$ and
 \begin{equation}\label{e3-0}
\rho_{\G}^{y,\beta}(t,x)=t^{\G}(|x|^{\beta}\wedge 1)
\begin{cases}
\frac{1}{(t^{{1/\A(y)}}+|x|)^{d+\A(y)}},&  |x|\le 1,\\
\frac{1}{|x|^{d+\A_1}},& |x|> 1.
\end{cases}
\end{equation}
If moreover $\tilde \B_0(x_0):=\big({\A_1\B_0}/{\A(x_0)})\wedge \A_1>1-\A(x_0)$ for some $x_0 \in \R^d$, then for every fixed $t \in (0,1]$ and $y \in \R^d$,
$p(t,\cdot,y)$ is differentiable at $x=x_0$.
In this case,
for every $\G>0$, $t\in(0,1]$ and $ y\in \R^d$,
\begin{equation}\label{t1-1-7a}
|\nabla p(t,\cdot,y)(x_0)|\le c_6 \, \rho^{y,0}_{1-({1/\A(x_0)})+
(\B_0^*/\A_2)-(\B_0^*/\A_1)-\G}(t,x_0-y)
\end{equation}
  for some $c_6 =c_6(\A,\K,\G,x_0)>0$, where $\B_0^*:=\B_0 \wedge\A_2$.

\item[(iii)]{\bf (Chapman-Kolmogorov equation)} For every $s,t \in (0,1]$ with $s+t\le 1$,
\begin{equation}\label{t1-1-5}
\int_{\R^d}p(s,x,z)p(t,z,y)\,dz=p(s+t,x,y),\quad x,y\in \R^d.
\end{equation}

\item[(iv)]{\bf (Conservativeness)} For every $(t,x)\in (0,1]\times \R^d$,
\begin{equation}\label{t1-1-6}
\int_{\R^d}p(t,x,y)\,dy=1.
\end{equation}

\item[(v)] {\bf (Strong continuity and generator)}
For every $f \in C_{b,u}(\R^d)$, let
$$u_f(t,x):=\int_{\R^d}p(t,x,y)f(y)\,dy.$$
Then
\begin{equation}\label{t1-1-2a}
\LL u_f (t, \cdot) (x) \hbox{ exists pointwise}, \quad t \mapsto \LL u_f (t, ) (x)
\hbox{  is continuous},
\end{equation}
\begin{equation}\label{t1-1-2}
\frac{\partial u_f(t,x)}{\partial t}
= \LL u_f(t,\cdot)(x),
\quad  t\in (0,1],\ x \in \R^d,
\end{equation}
and
\begin{equation}\label{t1-1-3}
\lim_{t \downarrow 0}\sup_{x \in \R^d}|u_f(t,x)-f(x)|=0.
\end{equation}

For every $f \in C_{b,u}^2(\R^d):= C_{b,u}(\R^d)\cap C^2(\R^d)$,  we also have
\begin{equation}\label{t1-1-5a}
\frac{\partial u_f(t,x)}{\partial t}=\int_{\R^d}
p(t,x,y)\LL f(y)\,dy,\quad  t\in
(0,1],\ x \in \R^d
\end{equation} and
\begin{equation}\label{t1-1-5b}
\lim_{t \downarrow 0}\sup_{x \in \R^d}\Big|\frac{1}{t}\cdot\big(
u_f(t,x)-f(x)\big)-\LL f(x)\Big|=0.
\end{equation}

\item[(vi)] {\bf (Uniqueness)} Jointly continuous function $p(t, x, y)$ on $(0, 1]\times \R^d\times \R^d$
that is bounded for each $t>0$ and  satisfies
\eqref{t1-1-7} and \eqref{t1-1-2a}-\eqref{t1-1-3} is unique.
\end{enumerate}
\end{theorem}

\begin{remark}\label{remark1}
\begin{itemize}
\item [(1)] From \eqref{t1-1-4} and \eqref{t1-1-4b}, we know that both upper and lower bounds of
near-diagonal
estimates for $p(t,x,y)$ enjoy the same order $t^{-{d/\A(x)}}$, which typically depends on $x$.
In particular, along with \eqref{t1-1-4b}, when $r_0=\infty$ in \eqref{t1-1-4a} (i.e.\ $\A(x)\equiv \A \in (0,2)$ for all $x\in \R^d$), we arrive at the two-sided heat kernel estimates obtained in
\cite{CZ}.
Furthermore, by \eqref{t1-1-7} and \eqref{t1-1-7a}, the regularity of $p(t,\cdot,y)$ at $x=x_0$ only depends on
$x_0$.

\item [(2)] Compared with these results yielded by  Duhamel's formula in \cite{Kol},  Theorem \ref{t1-1}(v) shows that the function $u_f(t,x):=\int_{\R^d}p(t,x,y)f(y)\,dy$ solves  the Cauchy problem for  $\LL$. In particular, by \eqref{t1-1-1} and \eqref{t1-1-3}, we have
\begin{equation*}
\begin{cases}
\frac{\partial }{\partial t} p(t,x,y) = \LL p(t,\cdot,y)(x)\\
\lim_{t \downarrow 0} p(t,x,y)=\delta_y(x),
\end{cases}
\end{equation*}
which means that $p(t,x,y)$ is the fundamental solution associated
with $\LL$. It is easy to deduce from \eqref{t1-1-7},  \eqref{t1-1-5}, \eqref{t1-1-6},
\eqref{t1-1-2} and \eqref{t1-1-5a} that there is  a conservative Feller process  having the strong Feller property
associated with the operator $\LL$;  cf.  Proposition \ref{p4-4} below.
Recall that a Markov process on $\R^d$ is said to be a Feller process if its transition semigroup is a strongly continuous semigroup in the Banach space of continuous functions that vanish at infinity equipped
with the uniform norm.
A Markov process is said to have the strong Feller property if its transition semigroup maps bounded
measurable functions to bounded continuous functions.
\end{itemize}
\end{remark}

\begin{theorem}\label{t1-2} Let $\B_0^{**}\in(0,\B_0]\cap(0,{\alpha_2}/{2}).$
For general $\K (x, z)$  satisfying  \eqref{e1-3} and \eqref{e1-2}, if
$
 ({\A_2}/{\A_1})-1<\B_0^{**}/{\A_2},
$
 there exists a
jointly continuous non-negative function
$p:(0,1]\times \R^d \times \R^d \rightarrow \R_+$ such that
all the conclusions in Theorem $\ref{t1-1}$ hold
except that the  upper bounds \eqref{t1-1-4}-\eqref{t1-1-4a} in {\rm (i)} and
H\"{o}lder regularity and gradient estimates \eqref{t1-1-7} and \eqref{t1-1-7a}
in {\rm (ii)} are to be replaced
by the following estimates.
\begin{enumerate}
\item[(i)]{\bf (Upper bounds)}\,\, For any $\gamma, c_0>0$, there exist constants $c_1 =c_1(\A,\K,c_0)>0$ and $c_2 =c_1(\A,\K,\G,c_0)>0$
such that for all $t\in (0,1]$ and $x,y\in\R^d$,
\begin{equation}\label{t1-1-9}
p(t,x,y)\le
\begin{cases}
{c_1}{t^{-{d/\A(x)}}},&\quad |x-y|< c_0t^{{1/\A(x)}}, \smallskip \\
\frac{c_1t^{2-({\A_2}/{\A_1})}}{|x-y|^{d+\A_2}} \wedge \frac{c_2t^{2-({\A_2}/{\A_1})-\gamma}}{|x-y|^{d+\A(x)}},&\quad
c_0t^{{1/\A(x)}}\le |x-y|< 1,  \smallskip \\
\frac{c_1t^{2-({\A_2}/{\A_1})}}{|x-y|^{d+\A_1}},&\quad |x-y|\ge 1.
\end{cases}
\end{equation}
If, in addition,  there are some $x_0\in\R^d$ and $r_0\in(0,\infty]$ such that $\alpha(z)=\alpha(x_0)$ for all $z \in B(x_0,r_0)$, then for any $c_0>0$ there is a positive
constant $c_3:=c_3(\A,\K,r_0,c_0)$ so that for every $t\in (0,1]$ and $y \in \R^d$,
\begin{equation}\label{t1-1-9a}
p(t,x_0,y)\le
\begin{cases}
{c_3}{t^{-{d/\A (x_0)}}},&\quad |y-x_0|< c_0t^{{1/\A(x_0)}}, \smallskip \\
\frac{c_3 t^{2-({\A_2}/{\A_1})}
}{|y-x_0|^{d+\A(x_0)}},& \quad
c_0t^{{1/\A(x_0)}}\le |y-x_0|< {r_0}/{2}, \smallskip \\
\frac{c_3 t^{2-({\A_2}/{\A_1})} }{|x-y_0|^{d+\A_1}},&\quad |y-x_0|\ge {r_0}/{2}.
\end{cases}
\end{equation}

\item[(ii)] {\bf (H\"{o}lder regularity and gradient estimates)}\,\,
For every
$\G \in (0, \A_1]$,
there exist constants
$c_5 =c_5(\A,\K,\G)$ and $R_1 =R_1(\A,\K,\G)>0$ such that for all $t \in (0,1]$ and $x,x',y \in \R^d$ with $|x-x'|\le R_1$,
\begin{equation}\label{t1-1-10}
\begin{split}
 \big|p(t,x,y)-p(t,x',y)\big|
&\le c_5|x-x'|^{(\A_1-\G)_+\wedge 1}
\big(\rho^{y,0}_{\gamma_1}(t,x-y)+ \rho^{y,0}_{\gamma_1}(t,x'-y)\big),
\end{split}
\end{equation}
where $\gamma_1:=1-({\A_2}/{\A_1})+({\G}/({2\A_2}))$.
Let $\B_0^*$ and $\tilde \B_0(x_0)$ be the same constants in
Theorem $\ref{t1-1}$ $(ii)$. If moreover $\tilde \B_0(x_0)-({\A_1\A_2}/{\A(x_0)})
\big(({\A_2}/{\A_1})-1\big)>1-\A_1$ for some
$x_0\in \R^d$, then for every fixed $t \in (0,1]$ and $y \in \R^d$, $p(t,\cdot,y)$ is differentiable
at $x=x_0$. In this case, for every $\G>0$, there exists a constant
$c_6:=c_6(\A,\K,\G, x_0)$ such that for all $ t\in (0,1]$ and $y\in \R^d$,
\begin{equation}\label{t5-1-2}
|\nabla p(t,\cdot,y)(x_0)|\le c_6
\rho^{x,0}_{1-(1/\A_1)+(\B_0^*/\A_2)-(\B_0^*/\A_1)-\G_2}(t,x_0-y),
\end{equation} where $\G_2=(\A_2/\A_1)-1+\G.$
\end{enumerate}
\end{theorem}

\begin{remark}
\begin{enumerate}
\item[(1)] According to Theorem \ref{t1-2} above, when $\A(x)$ is not a constant function and $\K(x,z)$
depends on $z$,
we can construct heat kernel associated with the operator $\LL$ when
$({\A_2}/{\A_1})-1<\B_0^{**}/{\A_2}$, which indicates that
the oscillation of the index function $\A(x)$ could not be too large.
We note that in \cite{BK2}, the elliptic Harnack inequality for $\LL$ was established under similar assumptions on the index function $\A(x)$.
Note that by \eqref{t1-1-9a} and \eqref{t1-1-4b}, when $\A(x)\equiv \A \in (0,2)$ for all $x\in \R^d$, two-sided estimates of heat kernel in Theorem \ref{t1-2} are reduced
to those in \cite{CZ}.

\item[(2)] As mentioned in Remark \ref{remark1}(2), Theorem \ref{t1-2} indeed provides us some sufficient conditions
for the existence of a strong Markov process (in fact a Feller process having the strong Feller property)
associated with $\LL$
when $\K(x,z)$ depends on $z$. Note that, the existence of such
strong Markov process is required and assumed a prior in the proofs of
the H\"{o}lder continuity of harmonic functions in \cite{BK1} and
the elliptic Harnack inequality for $\LL$ in \cite{BK2}.

 \end{enumerate}
\end{remark}

We conclude this part with the following remarks on some possible extensions of Theorems \ref{t1-1} and \ref{t1-2}.

\begin{remark}\label{rrr-000}
\begin{itemize}
 \item[(1)] By carefully checking the arguments of Theorem \ref{t1-1}, one can see that Theorem \ref{t1-1} still holds true if the condition that $\K(x,z)$ is independent of $z$ is replaced  by the following two assumptions on $\kappa(x,z)$:
\begin{itemize}
\item [(i)] For every fixed $x \in \R^d$, $\K(x,\cdot):\R^d \rightarrow \R_+$
is a radial function, i.e., $\K(x,z)$ only depends on $|z|$.

\item [(ii)] Let $j(x,r)=\wt \K(x,r)r^{-d-\A(x)}$ for $x \in \R^d$ and $r>0$, where
$\wt \K(x,r):=\K(x,z)$ with $|z|=r$. For every $x \in \R^d$, $j(x,\cdot)$ is non-increasing and differentiable on $(0,\infty)$ such that the function $r\mapsto -\frac{1}{r}\frac{\partial j(x,r)}{\partial r}$ is non-increasing.
\end{itemize}

\item[(2)] Let $\wt \LL f(x)=\LL f(x)+\langle b(x), \nabla f(x)\rangle$, where $b:\R^d\rightarrow \R^d$ is bounded and uniformly
H\"older continuous.
It is possible to extend Theorems \ref{t1-1} and \ref{t1-2}  to $\wt \LL$ under the assumption that $\A_1>1$,
When $\A_1>1$, the gradient perturbation $\langle b(x), \nabla \rangle$ is of lower order than $\LL$.
We  refer the readers to
\cite{BJ, CD, CH, JS, MM, MM1,XZ}
on heat kernel estimates associated
with non-local operators under gradient perturbation.
For heat kernel estimates associated
with non-local operators under non-local perturbations, see \cite{CW, W, CHXZ}.

\item[(3)] By combining with the approach from \cite{CZ4},
it is possible to remove the symmetry assumption of $\kappa (x, z)$ in $z$.
But we will not pursue this extension  in this paper.

\end{itemize}
\end{remark}

\subsection{Idea of the proofs: Levi's method}
In this paper, we will apply the Levi's method to construct
the fundamental solution $p(t,x,y)$ for $\LL$. Some ideas of our approach are inspired by
those in \cite{CZ, Kol}.

We first introduce some notation which will be
frequently used in this article, and then briefly mention the ideas of our approach.
Due to the symmetry assumption  $\K(x,z)=\K(x,-z)$, it is easy to see that
\begin{equation}\label{e2-1}
\begin{split}
\LL f(x)&=\lim_{\e \downarrow 0}\frac{1}{2}\int_{\{|z|>\e\}}
\big(f(x+z)+f(x-z)-2f(x)\big)\frac{\K(x,z)}{|z|^{d+\alpha(x)}}\,dz.
\end{split}
\end{equation}
For simplicity, throughout this paper we write $\LL f(x)$ as
\begin{equation*}
\LL f(x)=\frac{1}{2}\int_{\R^d}
\big(f(x+z)+f(x-z)-2f(x)\big)\frac{\K(x,z)}{|z|^{d+\alpha(x)}}\,dz.
\end{equation*}
Note that the above integral is absolutely convergent for bounded $C^2$ functions.

For fixed $y \in \R^d$, define
\begin{equation}\label{e2-2}
\LL^y f(x):=\int_{\R^d}\big(f(x+z)-f(x)-\langle\nabla f(x), z \rangle\I_{\{|z|\le 1\}}\big)\frac{\K(y,z)}{|z|^{d+\alpha(y)}}\,dz.
\end{equation}
Then $ \LL^y$ is the generator of a pure jump symmetric  L\'{e}vy process
$X^y:=(X_t^y)_{t\ge 0}$ with  jump measure
$\nu^y(dz)=\frac{\K(y,z)}{|z|^{d+\alpha(y)}}\,dz$.
We denote the
fundamental solution for $ \LL^y$ by $p^y:(0,\infty)\times \R^d
\rightarrow \R_+$, which is just the transition density of the
process $X^y$. The fundamental solution $p^y$ satisfies
that
\begin{equation*}
\frac{\partial p^y(t,x)}{\partial t}= \LL^yp^y(t,\cdot)(x)
\quad \hbox{for every } (t,x)\in (0,1]\times \R^d.
\end{equation*}
We remark here that although the operator $ \LL^y$
is clearly well defined on $C_c^2(\R^d)$.
It is also pointwisely well defined for the function
$x\mapsto p^y(t, x)$;  see  the
estimates in \cite[Theorem 2.4]{CZ}.

Throughout the paper,  we define for $ f:\R^d\rightarrow \R$,
\begin{equation*}
\delta_f(x;z):=f(x+z)+f(x-z)-2f(x), \quad \ x,z\in \R^d.
\end{equation*}
\begin{equation*}
\delta_{p^y}(t,x;z):=p^y(t,x+z)+p^y(t,x-z)-2p^y(t,x),\quad t\in
(0,\infty),\ x,y,z\in \R^d.
\end{equation*}
Then, for every $x,y,w\in \R^d$,
\begin{equation}\label{e2-3}
 \LL^w p^y(t,\cdot)(x)=\frac{1}{2}\int_{\R^d}\delta_{p^y}(t,x;z)\frac{\K(w,z)}{|z|^{d+\alpha(w)}}\,dz.
\end{equation}

According to the Levi's method (cf. \cite[pp.\ 310--311]{Fi}), we
look for the fundamental solution to \eqref{t1-1-1} of the following
form:
\begin{equation}\label{eq2-1}
\begin{split}
& p(t,x,y)=p^y(t,x-y)+\int_0^t\int_{\R^d}p^z(t-s,x-z)q(s,z,y)\,dz\,ds,
\end{split}
\end{equation}
where $q(t,x,y)$ solves
\begin{equation}\label{eq2-2}
q(t,x,y)=q_0(t,x,y)+\int_0^t\int_{\R^d}q_0(t-s,x,z)q(s,z,y)\,dz\,ds
\end{equation}
with
$
q_0(t,x,y)=( \LL^x- \LL^y)p^y(t,\cdot)(x-y).
$
So the main task of the remainder of this paper is to solve equation \eqref{eq2-1},
show it is  indeed  the unique fundamental solution of $\LL$, and derive its various properties.
The next three sections are devoted to the estimates for $q_0(t,x,y)$, $q(t,x,y)$ and $p(t,x,y)$ respectively. The proofs of main results are presented in Section \ref{section5}.

We will mainly follow the approach of \cite{CZ}, where the jumping
kernel is of type $\kappa(x,z)/|z|^{d+\alpha}$ for some $\alpha\in (0,2)$.
However, due to the variable order nature of the operator $\LL$ given
by \eqref{e1-1}, there are many
new challenges and difficulties.
In order to obtain
good estimates
and regularity of heat
kernel for $\LL$
in terms of the local behavior of the index function $\alpha (x)$,
we need to  introduce the key function $\rho_{\G}^{y,\B}$
(see \eqref{e3-0}), which involves the variable order $\A (x)$.
This brings us a lot of
difficulties  from the beginning of applying the Levi's method.
In comparison to \cite[Section 2.1]{CZ}, we need take into account
the variable index function $\alpha(x)$ in
some key
convolution
inequalities, see Section \ref{section3.1} for more details.
In the derivation of explicit upper bounds for $q_0(t,x,y)$ in  Proposition
\ref{l2-1}, we need consider the variation of both $\kappa(x,z)$
and $\alpha(x)$ in dealing with the difference between
$ \LL^xp^y(t,\cdot)$ and $ \LL^yp^y(t,\cdot)$. These causes a lot of complications
in proofs for  the crucial Proposition \ref{p3-1}, Proposition \ref{t4-2}
and Proposition \ref{t5-1}.

\medskip

\noindent {\bf Notation}\quad For any $a,b\in \R_+$, $a\wedge b:=\min\{a,b\}$ and $a\vee b:=\max \{a,b\}.$ For every measurable function $f,g:(0,1]\times\R^d\times \R^d \rightarrow \R_+$, the notation $f \asymp g$ means that there exists a
constant $1\le c_0<\infty$ such that $c_0^{-1} g(t,x,y)\le f(t,x,y)\le c_0 g(t,x,y)$ for every $(t,x,y) \in (0,1]\times\R^d\times \R^d $,  and the notation $f \p g$ (resp.\ $f \q g$)  means that there exists a
constant $0<c_1<\infty$ such that $f(t,x,y)\le c_1 g(t,x,y)$ (resp.\ $f(t,x,y)\ge c_1g(t,x,y)$) for every $(t,x,y) \in (0,1]\times\R^d\times \R^d $.

\medskip

\section{Estimates for $q_0(t,x,y)$}\label{section2}

\subsection{Preliminary estimates}
For each fixed $y \in \R^d$, $\B \in [0,\infty)$, $\G\in \R$ and $R \in (0,\infty)$, we
define a function
$\rho_{\G,R}^{y,\beta}:(0,1]\times \R^d \rightarrow (0,\infty)$ as follows
\begin{equation*}
\rho_{\G,R}^{y,\beta}(t,x):=t^{\G}\big(|x|^{\beta}\wedge 1\big)
\begin{cases}
 \big(t^{{1/\A(y)}}+|x|\big)^{-(d+\A(y))},&  |x|\le R,\\
 |x|^{-(d+\A_1)},& |x|> R,
\end{cases}
\end{equation*} where $\alpha(x)$ is the index function in \eqref{e1-1} and $\alpha_1$
is the lower bound of $\alpha (x)$  in \eqref{e1-3}.
It is easy to see that for any $0<R_1\le R_2<\infty$, there is a constant $c>1$ such that
$c^{-1}\rho_{\G,R_1}^{y,\beta}\le \rho^{y,\beta}_{\G,R_2}\le
c\rho_{\G,R_1}^{y,\beta}$ for all $y\in \R^d$ and $\beta\in[0,\infty)$.
Therefore, without loss of generality,
we may and do assume that $R=1$ in the definition of
$\rho_{\G,R}^{y,\beta},$ and we write
$\rho_{\G,1}^{y,\beta}$ as $\rho_{\G}^{y,\beta}$, which is exactly the one defined by
\eqref{e3-0}.

For every $y \in \R^d$ and $s>0$, define
$
\overline{\A}(y;s):=\sup\{\A(z): |z-y|\le s\}$ and $  \underline{\A}(y;s):=\inf\{\A(z): |z-y|\le s\}.
$ For simplicity, we write $\overline
\A(y;|x-y|)$ and $\underline{\A}(y;|x-y|)$ as $\overline{\A}(y;x)$ and $\underline{\A}(y;x)$, respectively.

We begin with the following simple lemma.
\begin{lemma}\label{l3-1}
For every $\beta \in [0,\A_2)$, $t \in (0,1]$ and $x \in \R^d$,
\begin{equation}\label{l3-1-1}
\int_{\R^d}\rho_{0}^{z,\beta}(t,x-z)\,dz\p \big(t^{(\beta/\A(x))-1}\vee 1\big)\big(1+|\log t|\I_{\{\B=\A(x)\}}\big)
\p t^{(\B/\A_2)-1}\vee 1.
\end{equation}\end{lemma}

\begin{proof} Since $\beta\in (0,\A_2)$, it suffices to prove the first inequality in \eqref{l3-1-1}.

For every $x,z\in \R^d$ with $|x-z|\le 1$,
\begin{equation}\label{l3-1-1a}
\begin{split}
|x-z|^{-\A(z)}&=|x-z|^{-\A(x)}|x-z|^{\A(x)-\A(z)}\p |x-z|^{-\A(x)}\cdot \exp\left(\left|\log|x-z|\right|\left|\A(x)-\A(z)\right|\right)\\
&\p |x-z|^{-\A(x)}\cdot \exp\left(C\left|\log|x-z|\right|\cdot |x-z|^{\B_0}\right)
\p |x-z|^{-\A(x)},
\end{split}
\end{equation}
where in the last inequality we have used
$$\sup_{z\in \R^d: |z|\le 1}\exp\big(C|\log
|z||\cdot|z|^{\beta_0}\big)<\infty.$$
Then, for all $t\in (0,1]$ and $x,z\in \R^d$,
\begin{align*}
\frac{1}{(t^{{1}/{\A(z)}}+|x-z|)^{d+\A(z)}}& \p
\begin{cases}
|x-z|^{-d-\A(z)},& |x-z|\le 1,\\
t^{-1-({d/\underline{\A}(x; t^{1/\A_2})})},&  |x-z|\le t^{{1/\A(x)}},
\end{cases}
\\
&\p
\begin{cases}
|x-z|^{-d-\A(x)},& |x-z|\le 1,\\
t^{-1-({d/\underline{\A}(x; t^{1/\A_2})})},&  |x-z|\le t^{{1/\A(x)}}.
\end{cases}
\end{align*}
Therefore, by the definition of $\rho_{0}^{z,\beta}$ (see \eqref{e3-0}), for all $t\in (0,1]$ and $x\in \R^d$, we have
\begin{align*}
\int_{\R^d}\rho_{0}^{z,\beta}(t,x-z)\,dz \p&
\int_{\{|z-x|> 1\}}\frac{1}{|x-z|^{d+\A_1}}\,dz +
\int_{\{t^{{1/\A(x)}}<|z-x|\le 1\}}\frac{|x-z|^{\beta}}{|x-z|^{d+\A(x)}}\,dz\\
&+t^{-1-({d/\underline{\A}(x; t^{1/\A_2})})}\int_{\{|z-x|\le t^{1/\A(x)}\}}
|x-z|^{\beta}\,dz\\
=&:J_1+J_2+J_3.
\end{align*}

It is easy to verify that $J_1\p 1$  and $J_2\p \big(t^{(\beta-\A(x))/{\A(x)}}\vee 1\big)
\big(1+|\log t|\I_{\{\B=\A(x)\}}\big)$, due to the assumption that $\beta\in [0,\A_2)$.
At the same time,
\begin{align*}
J_3&\p t^{-1-({d/\underline{\A}(x; t^{1/\A_2}}))}\cdot t^{(\beta+d)/\A(x)} \p t^{(\beta-\A(x))/\A(x)}\exp\left(d|\log t| \frac{\A(x)-\underline{\A}(x;t^{{1/\A_2}}) }{\A(x)\underline{\A}(x;t^{{1/\A_2}})}\right) \\
&\p t^{(\beta-\A(x))/\A(x)}\exp\big(c|\log t| t^{{\B_0}/{\A_2}}\big)\p t^{{(\beta-\A(x))}/{\A(x)}},
\end{align*}
where in the forth inequality we have used
$$
\frac{\A(x)-\underline{\A}(x;t^{{1/\A_2}}) }{\A(x)\underline{\A}(x;t^{{1/\A_2}})}\p t^{{\B_0}/{\A_2}},
$$
 thanks to \eqref{e1-3}.
Combining all the estimates above, we get the first inequality in \eqref{l3-1-1}. The proof is complete.
\end{proof}

\begin{remark} By the proof of \cite[Lemma 2.1(i)]{CZ}, we can obtain that for every $\beta \in [0,\A_2)$, $t \in (0,1]$ and $x \in \R^d$,
\begin{equation}\label{l3-1-1-00}
\int_{\R^d}\rho_{0}^{x,\beta}(t,x-z)\,dz\p \big(t^{({\beta}/{\A(x)})-1}\vee 1\big)\big(1+|\log t|\I_{\{\B=\A(x)\}}\big)\p t^{({\beta}/{\A_2})-1}\vee 1.
\end{equation}
Note that although the definition of $\rho^{x,\B}_{0}(t,x)$ here is a little
different from that in \cite{CZ} when $|x|>1$, the proof of
\cite[Lemma 2.1]{CZ}  still works  for the first inequality in \eqref{l3-1-1-00}. We emphasize that in the present setting we
need  estimate \eqref{l3-1-1}, where in the integrand
$\rho^{z,\B}_0$ the index $\A(z)$ depends on $z$.
\end{remark}

\begin{lemma}\label{l1-1}
For every $t \in (0,1]$ and $x,y,w\in \R^d$, define
\begin{equation}\label{l1-1-0}
\begin{split}
I_1(t,x,y,w):=&\int_{\{|z|\le 1\}}\big[\big(t^{-{2}/{\A(y)}}|z|^2\big)\wedge 1\big]
\rho^{y,0}_1(t,w
+
z)\frac{1+|\log|z||}{|z|^{d+\overline{\A}(y;x)}}\,dz,\\
I_2(t,x,y,w):=&\int_{\{|z|>1\}}\rho^{y,0}_1(t,w
+
z)\frac{1+|\log|z||}{|z|^{d+\underline{\A}(y;x)}}\,dz,\\
I_3(t,x,y):=&\int_{\R^d}\big[\big(t^{-{2}/{\A(y)}}|z|^2\big)\wedge 1\big]  \left(
\frac{1+|\log|z||}{|z|^{d+\overline{\A}(y;x)}}\I_{\{|z|\le 1\}}
+\frac{1+|\log|z||}{|z|^{d+\underline{\A}(y;x)}}\I_{\{|z|> 1\}}\right)\,dz.
\end{split}
\end{equation}
Then there exists a constant $c_1:=c_1(\A,\K)$ such that for every
$x,y,w \in \R^d$ and $t \in (0,1]$,
\begin{equation}\label{l1-1-1}
\begin{split}
 I_1(t,x,y,w) \le
\begin{cases}
c_1(1+|\log t|)t^{1-({\overline{\A}(y;x)}/{\A(y)})}\rho^{y,0}_0(t,w),\ & |w|\le {1}/{2},\\
c_1\big[1+|\log|w||+(1+|\log t|)t^{1-({\overline{\A}(y;x)}/{\A(y)})}\big]\rho^{y,0}_0(t,w),\ &
|w|> {1}/{2};
\end{cases}\end{split}
\end{equation}
\begin{equation}\label{l1-1-2}
I_2(t,x,y,w)\le
\begin{cases}
c_1\rho^{y,0}_1(t,w),\ & |w|\le {1}/{2},\\
c_1\big(1+|\log|w||\big)\rho^{y,0}_0(t,w),\ & |w|>{1}/{2};
\end{cases}
\end{equation} and
\begin{equation}\label{l1-1-3}
I_3(t,x,y)\le
c_1(1+|\log t|)t^{-({\overline{\A}(y;x)}/{\A(y)})}.
\end{equation}
\end{lemma}
\begin{proof}
We only prove \eqref{l1-1-1} and \eqref{l1-1-2}, since
the proof of \eqref{l1-1-3} is similar and more direct. We denote $I_1(t,x,y,w)$ and $I_2(t,x,y,w)$ by $I_1$ and $I_2$, respectively.
Write
\begin{align*}
I_1&=\int_{\{|z|\le t^{{1/\A(y)}}\}}
\big(t^{-{2}/{\A(y)}}|z|^2\big)\rho^{y,0}_1(t,w{+} z)\frac{1+|\log|z||}{|z|^{d+\overline{\A}(y;x)}}\,dz  +\int_{\{t^{{1/\A(y)}}<|z|\le 1\}}\rho^{y,0}_1(t,w{+} z)\frac{1+|\log|z||}{|z|^{d+\overline{\A}(y;x)}}\,dz\\
&=:I_{11}+I_{12}.
\end{align*}
We will divide the proof into the following three subcases.

{\bf Case (a): $0\le |w| \le 2t^{{1/\A(y)}}$.}

Since, by \eqref{e3-0}, \begin{equation}\label{e:ond} \rho^{y,0}_1(t,w{+} z)\p t^{-{d/\A(y)}}\quad \textrm{ for all } t\in(0,1]\textrm{ and }  y,z,w\in\R^d,\end{equation} we have
\begin{align*}
I_{11}&\p t^{-{(d+2)}/{\A(y)}}\int_{\{|z|\le t^{{1/\A(y)}}\}}
\frac{(1+|\log|z|)|z|^2}{|z|^{d+\overline{\A}(y;x)}}\,dz\\
&\p(1+|\log t|)t^{-{(d+\overline\A(y;x))}/{\A(y)}} \p(1+|\log t|)t^{1-({\overline{\A}(y;x)}/{\A(y)})}\rho^{y,0}_0(t,w),
\end{align*}
where in the last inequality we have used the fact that \begin{equation}\label{e:one1}t^{-{d/\A(y)}}\p \rho^{y,0}_1(t,w)
\quad \textrm{ when } |w|\le 2t^{{1/\A(y)}}.\end{equation}

Applying \eqref{e:ond} and \eqref{e:one1} to
$I_{12}$ and $I_2$ again, we get
\begin{align*}
I_{12}&\p t^{-{d/\A(y)}}\int_{\{t^{{1/\A(y)}}<|z|\le 1\}}
\frac{1+|\log |z||}{|z|^{d+\overline{\A}(y;x)}}\,dz \p(1+|\log t|)t^{-({d+\overline\A(y;x)})/{\A(y)}}\\
& \p(1+|\log t|)t^{1-({\overline
\A(y;x)}/{\A(y)})}\rho^{y,0}_0(t,w)
\end{align*} and
\begin{align*}
I_2&\p  t^{-{d/\A(y)}}\int_{\{|z|> 1\}} \frac{1+\log
|z|}{|z|^{d+\A_1}}\,dz\p \rho^{y,0}_1(t,w).
\end{align*}

{\bf Case (b): $2t^{{1/\A(y)}}<|w|\le {1}/{2}$.}

When $|w|>2t^{{1/\A(y)}}$ and $|z|\le t^{{1/\A(y)}}$,
$
|w{+} z|\ge |w|-{|w|}/{2}\ge {|w|}/{2},
$
which along with \eqref{e3-0} implies that
\begin{equation}\label{l2-1-7}
\begin{split}
\rho^{y,0}_0(t,w{+} z)&\p
\frac{1}{(t^{{1/\A(y)}}+|w{+} z|)^{d+\A(y)}}\p \frac{1}{(t^{{1/\A(y)}}+|w|)^{d+\A(y)}}.
\end{split}
\end{equation}
Therefore,
\begin{align*}
I_{11}&\p \frac{t^{1-{2}/{\A(y)}}}{(t^{{1/\A(y)}}+|w|)^{d+\A(y)}}
\int_{\{|z|\le t^{{1/\A(y)}}\}}\frac{(1+|\log |z||)|z|^2}{|z|^{d+\overline{\A}(y;x)}}\,dz\\
& \p (1+|\log t|)t^{1-({\overline\A(y;x)}/{\A(y)})}\frac{1}{(t^{{1/\A(y)}}+|w|)^{d+\A(y)}} \p (1+|\log t|)t^{1-({\overline
\A(y;x)}/{\A(y)})}\rho^{y,0}_0(t,w),
\end{align*}
where the last step is due to the fact that \begin{equation}\label{e:one-2}\frac{1}{(t^{{1/\A(y)}}+|w|)^{d+\A(y)}} \p \rho^{y,0}_0(t,w)\quad
\textrm{ when }|w|>2t^{{1/\A(y)}}.\end{equation}

On the other hand, we have the following decomposition for $I_{12}$:
\begin{align*}
I_{12}&=\int_{\{
t^{{1/\A(y)}}<|z|\le{|w|}/{2}\}}\rho^{y,0}_1(t,w{+} z)
\frac{1+|\log |z||}{|z|^{d+\overline{\A}(y;x)}}\,dz  + \int_{\{{|w|}/{2}\le |z|\le 1\}}\rho^{y,0}_1(t,w{+} z)
\frac{1+|\log |z||}{|z|^{d+\overline{\A}(y;x)}}\,dz\\
&=:I_{121}+I_{122}.
\end{align*}
If $|z|\le {|w|}/{2}$, then $|w{+} z|\ge {|w|}/{2}$,
and so \eqref{l2-1-7} still holds. Hence, we have
\begin{align*}
I_{121}&\p \frac{t}{(t^{{1/\A(y)}}+|w|)^{d+\A(y)}}
\int_{\{|z|>t^{{1/\A(y)}}\}}\frac{1+|\log |z||}{|z|^{d+\overline{\A}(y;x)}}\,dz\\
&\p (1+|\log t|)t^{-{\overline
\A(y;x)}/{\A(y)}}\frac{t}{(t^{{1/\A(y)}}+|w|)^{d+\A(y)}} \p (1+|\log t|)t^{1-({\overline
\A(y;x)}/{\A(y)})}\rho^{y,0}_0(t,w),
\end{align*} where the last step follows from \eqref{e:one-2}.
Meanwhile,
\begin{align*}
I_{122}&\p \frac{1+|\log|w||}{|w|^{d+\overline{\A}(y;x)}}\int_{\R^d}\rho^{y,0}_1(t,w{+} z)\,dz\p \frac{1+|\log|w||}{|w|^{d+\overline{\A}(y;x)}}\\
&\p \frac{(1+|\log|w||)|w|^{\A(y)-\overline{\A}(y;x)}}{(t^{{1/\A(y)}}+|w|)^{d+\A(y)}}\p (1+|\log t|)t^{1-({\overline{\A}(y;x)}/{\A(y)})}\rho^{y,0}_0(t,w),
\end{align*}
where the second inequality is due to \eqref{l3-1-1-00},
in the third inequality we have used the fact that $$|w|^{-d-\A(y)}\p
(t^{{1/\A(y)}}+|w|)^{-d-\A(y)}\quad \textrm{ when } 2t^{{1/\A(y)}}<|w|\le {1}/{2},$$ and the fourth
inequality follows from
$$ (1+|\log|w||)|w|^{\A(y)-\overline{\A}(y;x)}\p (1+|\log t|)t^{1-({\overline{\A}(y;x)}/{\A(y)})}
\quad\textrm{ when }2t^{{1/\A(y)}}<|w|\le {1}/{2}.$$
Combining with all the estimates above, we obtain
$$I_1\p (1+|\log t|)t^{1-({\overline{\A}(y;x)}/{\A(y)})}\rho^{y,0}_0(t,w).$$

Furthermore, note that if $|w|\le {1}/{2}$ and $|z|>1$, then
$
|w{+} z|\ge |z|-|w|\ge 2|w|-|w|=|w|.
$
Therefore, for every $y,z,w\in\R^d$ and $t\in(0,1]$ with $|z|>1$ and $2t^{{1/\A(y)}}\le |w|\le{1}/{2}$,
\begin{align*}
\rho^{y,0}_0(t,w{+} z)& \p
\begin{cases}
|w|^{-d-\A(y)}, \ &\ |w{+} z|\le 1,\\
|w|^{-d-\A_1},\ &\  |w{+} z|> 1
\end{cases}\\
&\p |w|^{-d-\A(y)},
\end{align*}
which implies that
\begin{align*}
I_2&\p \frac{t}{|w|^{d+\A(y)}}
\int_{\{|z|\ge 1\}} \frac{1+\log |z|}{|z|^{d+\A_1}}\,dz\p
\frac{t}{|w|^{d+\A(y)}}\p \rho^{y,0}_1(t,w).
\end{align*}

{\bf{Case (c): $|w|>{1}/{2}$}.}

By adjusting the constants properly, it is easy to verify that
\begin{align*}
I_1+I_2&\p \int_{\{|z|\le {1}/{4}\}}\big[\big(t^{-{2}/{\A(y)}}|z|^2\big)\wedge 1\big]
\rho^{y,0}_1(t,w{+} z)\frac{1+|\log|z||}{|z|^{d+\overline{\A}(y;x)}}\,dz +\int_{\{|z|>{1}/{4}\}}\rho^{y,0}_1(t,w{+} z)\frac{1+|\log|z||}{|z|^{d+\underline{\A}(y;x)}}\,dz\\
&=:J_1+J_2.
\end{align*}
Note that when $|w|>{1}/{2}$ and $|z|\le {1}/{4}$,
$|w{+} z|\ge |w|-|z|\ge {|w|}/{2}>{1}/{4}$, which implies
that
$$\rho^{y,0}_1(t,w{+} z)\p \frac{t}{|w|^{d+\A_1}}.$$ On the other hand,
\begin{align*}
&\int_{{\{|z|\le 1\}}}\big[\big(t^{-{2}/{\A(y)}}|z|^2\big)\wedge 1\big]
\frac{1+|\log|z||}{|z|^{d+\overline{\A}(y;x)}}\,dz\\
&\p \int_{{\{|z|\le t^{{1/\A(y)}}\}}}\big(t^{-{2}/{\A(y)}}|z|^2\big)
\frac{1+|\log|z||}{|z|^{d+\overline{\A}(y;x)}}\,dz+\int_{\{t^{{1/\A(y)}}<|z|\le 1\}}\frac{1+|\log|z||}{|z|^{d+\overline{\A}(y;x)}}\,dz\\
&\p (1+|\log t|)t^{-{\overline{\A}(y;x)}/{\A(y)}}.
\end{align*}
Therefore, we obtain
\begin{align*}
J_1&\p \frac{t}{|w|^{d+\A_1}} \int_{{\{|z|\le {1}/{4}\}}}\big[\big(t^{-{2}/{\A(y)}}|z|^2\big)\wedge 1\big]
\frac{1+|\log|z||}{|z|^{d+\overline{\A}(y;x)}}\,dz\\
&\p \frac{\big(1+|\log t|\big)t^{1-({\overline{\A}(y;x)}/{\A(y)})}}{|w|^{d+\A_1}}\p \big(1+|\log t|\big)t^{1-({\overline{\A}(y;x)}/{\A(y)})}\rho^{y,0}_0(t,w).
\end{align*}
Meanwhile,
\begin{align*}
J_2&\p \int_{\{|z|>{|w|}/{2}\}}\!\!\rho^{y,0}_1(t,w{+} z)\frac{1+|\log|z||}{|z|^{d+\underline{\A}(y;x)}}\,dz +\!
\int_{\{{1}/{4}<|z|<{|w|}/{2}\}}\!\!\rho^{y,0}_1(t,w{+} z)\frac{1+|\log|z||}{|z|^{d+\underline{\A}(y;x)}}\,dz\\
&\p
\frac{1+|\log|w||}{|w|^{d+\underline{\A}(y;x)}} \int_{\{|z|>{|w|}/{2}\}}\!\!
\rho^{y,0}_1(t,w{+} z)\,dz +\!
\frac{t}{|w|^{d+\A_1}}\int_{\{|z|>{1}/{4}\}}\!\!\frac{1+|\log|z||}{|z|^{d+\A_1}}\,dz\\
&\p \frac{1+|\log |w||}{|w|^{d+\A_1}}\p (1+|\log |w||)\rho^{y,0}_0(t,w),
\end{align*}
where the second step above follows from  the following property
\begin{equation*}
\rho^{y,0}_0(t,w{+} z)\p |w|^{-d-\A_1}\quad \textrm{ when }
|z|<{|w|}/{2}\textrm{ and }\ |w|>{1}/{2},
\end{equation*}
and the third step is due to \eqref{l3-1-1-00}.
Hence, for all $w\in \R^d$ with $|w|>{1}/{2}$,
$$ I_1\p J_1+J_2\p \big[1+|\log |w||+(1+|\log t|)t^{-{\overline{\A}(y;x)}/{\A(y)}}\big]\rho^{y,0}_0(t,w)$$ and
$$ I_2\p J_2\p  (1+|\log |w||)\rho^{y,0}_0(t,w).
$$
Combining all the three cases together,  we finish the proof.
\end{proof}

\begin{remark}\label{r-1}
Following the arguments of \eqref{l1-1-1} and \eqref{l1-1-2} for case (c) above, we can get the following
 for every $c_0>0$,
$t \in (0,1]$ and $x,y,w \in \R^d$ with $|w|>2c_0$:
\begin{itemize}
\item[(i)] \begin{equation}\label{r1-1-1a}
\int_{\{|z|\le c_0\}}\big[\big(t^{-{2}/{\A(y)}} |z|^2\big)\wedge 1\big]
\big(
\rho_1^{y,0}(t,w{+} z)
+\rho_1^{y,0}(t,w)\big)\frac{1}{|z|^{d+\A_2}}\,dz\p \frac{t^{1-\A_2/\A_1}}{|w|^{d+\A_1}}
\end{equation} and
\begin{equation}\label{r1-1-2a}
\begin{split}
&\int_{\{|z|\le c_0\}}\big[\big(t^{-{2}/{\A(y)}} |z|^2\big)\wedge 1\big]
\big(
\rho_1^{y,0}(t,w + z)
+\rho_1^{y,0}(t,w)\big)\frac{1+|\log|z||}{|z|^{d+\A_2}}\,dz\\
&\p
\big(1+|\log t|+|\log|w||\big)\frac{t^{1-\A_2/\A_1}}{|w|^{d+\A_1}}.
\end{split}
\end{equation}

\item[(ii)]
\begin{equation}\label{r1-1-1}
\int_{\{|z|>c_0\}}\big(\rho^{y,0}_1(t,w\pm z)+\rho^{y,0}_1(t,w)\big)\frac{1}{|z|^{d+\A_1}}\,dz\p \frac 1{|w|^{d+\A_1}}
\end{equation}
 and
\begin{equation}\label{r1-1-2}
\int_{\{|z|>c_0\}}\big(\rho^{y,0}_1(t,w{+} z)+\rho^{y,0}_1(t,w)\big)\frac{1+|\log|z||}{|z|^{d+\A_1}}\,dz\p
\log(1+|w|)\frac 1{|w|^{d+\A_1}}.
\end{equation}
\end{itemize}
\end{remark}

\medskip

\subsection{Upper bounds for $q_0(t,x,y)$}

\begin{proposition}\label{l2-1}
\begin{itemize}
\item[(1)] Suppose $\K(x,z) =\K (x)$ is independent of $z$. Then for any $\gamma>0$, there exist constants
$R_0:=R_0(\A,\K, \G)\in(0,1)$ and $c_0:=c_0(\A,\K, \G)>0$ such
that for all $t\in(0,1]$ and $x,y\in \R^d$,
\begin{equation}\label{l2-1-2}\begin{split}
|q_0(t,x,y)|\le &c_0\big(|x-y|^{\beta_0}\wedge 1\big)\cdot
\begin{cases}
\frac{t^{-\gamma}}{(t^{{1/\A(y)}}+|x-y|)^{d+\A(y)}},& |x-y|< R_0,\\
\frac{1}{|x-y|^{d+\A_1}},&|x-y|\ge R_0
\end{cases}\\
\p& \rho_{-\gamma}^{y,\beta_0}(t,x-y).\end{split}
\end{equation}
Suppose, in addition, that there are some $y_0\in \R^d$ and $r_0\in(0,\infty]$ such that $\alpha(z)=\alpha(y_0)$ for all $z \in B(y_0,r_0)$. Then
there exists a constant
$c_1:=c_1(\A,\K, \G, r_0)>0$ such that for every $t\in(0,1]$ and $x \in \R^d$,
\begin{equation}\label{l2-1-33a}
|q_0(t,x,y_0)|\le c_1\big(|x-y_0|^{\beta_0}\wedge 1\big)\cdot
\begin{cases}
\frac{1}{(t^{{1/\A(y_0)}}+|x-y_0|)^{d+\A(y_0)}},& |x-y_0|< r_0,\\
\frac{1}{|x-y_0|^{d+\A_1}},&|x-y_0|\ge r_0.
\end{cases}
\end{equation}

\item[(2)]
For general $\K (x, z)$ and for any $\gamma>0$, there are constants
$R_0:=R_0(\A,\K, \G)\in(0,1)$ and $c_0:=c_0(\A,\K, \G)>0$ such
that for all $t\in(0,1]$ and $x,y\in \R^d$,
\begin{equation}\label{l2-1-1} \begin{split}
|q_0(t,x,y)|\le& c_0\big(|x-y|^{\beta_0}\wedge 1\big)\cdot
\begin{cases}
\frac{t^{-\gamma}}{(t^{{1/\A(y)}}+|x-y|)^{d+\A(y)}},& |x-y|< R_0,\\
\frac{t^{1-({\A_2}/{\A_1})}}{|x-y|^{d+\A_1}}, &|x-y|\ge R_0
\end{cases}\\
\p&\rho_{1-({\A_2}/{\A_1})-\gamma}^{y,\beta_0}(t,x-y).\end{split}
\end{equation}
Suppose, in addition, that there are some $y_0\in \R^d$ and $r_0\in(0,\infty]$ such that $\alpha(z)=\alpha(y_0)$ for all $z \in B(y_0,r_0)$. Then for every $t\in (0,1]$ and $x\in \R^d$,
\begin{equation}\label{l2-1-33}
|q_0(t,x,y_0)|\le c_1\big(|x-y_0|^{\beta_0}\wedge 1\big)\cdot
\begin{cases}
\frac{1}{(t^{{1/\A(y_0)}}+|x-y_0|)^{d+\A(y_0)}},& |x-y_0|< r_0,\\
\frac{t^{1-({\A_2}/{\A_1})}}{|x-y_0|^{d+\A_1}},&|x-y_0|\ge r_0.
\end{cases}
\end{equation}
\end{itemize}
\end{proposition}

\medskip

To prove Proposition \ref{l2-1}, we need some regularity estimates for $p^y(t,\cdot)$ taken from \cite{CZ}.
For simplicity, we use the following abbreviation for a function $f$:
$$
f(x\pm y):= f(x+y) + f(x-y).
$$

\begin{lemma}$($\cite[Lemmas 2.2 and 2.3]{CZ}$)$\label{p1-1} The following statements hold.
\begin{itemize}
\item[(1)] If $\K(x,z) =\K (x)$ is independent of $z$, then for any $j \in \mathbb{Z}_+:=\{0,1,\ldots\}$, there exists a constant $c_1:=c_1(\A,\K, j)>0$ such that for every $t \in (0,1]$ and
$x,y\in \R^d$,
\begin{equation}\label{p1-1-1}
|\nabla^j p^y(t,\cdot)(x)|\le c_1t\big(t^{{1/\A(y)}}+|x|\big)^{-d-\A(y)-j}.
\end{equation}

\item[(2)] For general  $\K(x,z)$,
there exists a constant
$c_2:=c_2(\A,\K)>0$ such that for every $t \in (0,1]$ and $x,x',y,z \in \R^d$,
\begin{equation}\label{p1-1-1-1-0}
p^y(t,x)\le c_2 \rho_1^{y,0}(t,x),
\end{equation}
\begin{equation}\label{p1-1-2}
|\delta_{p^y}(t,x;z)|\le
c_2\big[\big(t^{-{2}/{\A(y)}}|z|^2\big)\wedge 1\big] \cdot
\big(\rho_1^{y,0}(t,x\pm z)+\rho_1^{y,0}(t,x)\big),
\end{equation}
\begin{equation}\label{p1-1-3}
\begin{split}
|p^y(t,x)-p^y(t,x')|&\le c_2\big[\big(t^{-{1/\A(y)}}|x-x'|\big)\wedge 1\big]
\cdot\big(\rho_1^{y,0}(t,x')+\rho_1^{y,0}(t,x)\big),
\end{split}
\end{equation} and
\begin{equation}\label{p1-1-4}
\begin{split}
|\delta_{p^y}(t,x;z)-\delta_{p^y}(t,x';z)|&
\le  c_2 \big[\big(t^{-{1/\A(y)}}|x-x'|\big)\wedge 1\big]
\cdot \big[\big(t^{-{2}/{\A(y)}}|z|^2\big)\wedge 1\big]\\
&\quad\times\big(\rho_1^{y,0}(t,x\pm z)+\rho_1^{y,0}(t,x)+\rho_1^{y,0}(t,x'\pm z)+\rho_1^{y,0}(t,x')\big).
\end{split}
\end{equation}
\end{itemize}
\end{lemma}

\begin{lemma}\label{l1-2}
For every $\G>0$, there exist constants $R_1:=R_1(\A,\K,\G)\in (0,{1}/{2})$ and $c_1:=c_1(\A,\K,\G)>0$ such that
for every $t \in (0,1]$ and $x,y,w\in \R^d$ with $|x-y|\le R_1$
\begin{equation}\label{l1-2-1}\begin{split}
&\big|\big( \LL^x - \LL^y\big)p^y(t,w)\big| \le c_1\big(|x-y|^{\B_0}\wedge 1\big)
\cdot
\begin{cases}
t^{-\G}\rho_0^{y,0}(t,w), &\ |w|\le {1}/{2},\\
(t^{-\G}+|\log |w||)\rho_0^{y,0}(t,w),&\ |w|>{1}/{2}.
\end{cases}\end{split}
\end{equation} In particular, for every $t\in(0,1]$ and $x,y\in\R^d$ with $|x-y|\le R_1$,
\begin{equation}\label{l1-2-1a}
\big|q_0(t,x,y)\big|\p \rho_{-\G}^{y,\B_0}(t,x-y).
\end{equation}

If, in addition, there are some $y_0\in \R^d$ and $r_0\in(0,\infty]$ such that $\alpha(z)=\alpha(y_0)$ for all $z \in B(y_0,r_0)$,
then
there exists a constant
$c_2:=c_2(\A,\K, \G, r_0)>0$ such that for every $t\in(0,1]$ and $x \in \R^d$ with $|x-y_0|<r_0$,
\begin{equation}\label{l1-2-1b}
|q_0(t,x,y_0)|\le c_2\big(|x-y_0|^{\beta_0}\wedge 1\big)\cdot
\frac{1}{(t^{{1/\A(y_0)}}+|x-y_0|)^{d+\A(y_0)}}.
\end{equation}
\end{lemma}
\begin{proof} (i)
Applying the mean value theorem to the function $r\mapsto a^r$, we find that for every $x,y,z \in \R^d$ with
$|z|>0$,
\begin{equation}\label{e:pp}
\begin{split}
&\bigg|\frac{\K(x,z)}{|z|^{d+\A(x)}}-\frac{\K(y,z)}{|z|^{d+\A(y)}}\bigg|\\
&\p |\A(x)-\A(y)|\bigg(\frac{|\log|z||}{|z|^{d+\overline{\A}(y;x)}}\I_{\{|z|\le 1\}}+
\frac{\log|z|}{|z|^{d+\underline{\A}(y;x)}}\I_{\{|z|> 1\}}\bigg)  +\frac{|\K(x,z)-\K(y,z)|}{|z|^{d+\A(y)}}\\
&\p \big(|x-y|^{\B_0}\wedge
1\big)\bigg(\frac{1+|\log|z||}{|z|^{d+\overline
\A(y;x)}}\I_{\{|z|\le 1\}}+ \frac{1+\log|z|}{|z|^{d+\underline
\A(y;x)}}\I_{\{|z|> 1\}}\bigg),
\end{split}
\end{equation}
where in the last inequality we have used \eqref{e1-2}, \eqref{e1-3} and the
fact that
\begin{equation*}
\frac{1}{|z|^{d+\A(y)}}\p \frac{1}{|z|^{d+\overline
\A(y;x)}}\I_{\{|z|\le 1\}}+ \frac{1}{|z|^{d+\underline
\A(y;x)}}\I_{\{|z|> 1\}}.
\end{equation*}
Combining \eqref{e:pp} with \eqref{e2-3}, we can obtain that
\begin{align*}
 \big|\big( \LL^x- \LL^y\big)p^y(t,w)\big|
& \p
\big(|x-y|^{\B_0}\wedge 1\big) \int_{\R^d}|\delta_{p^y}(t,w;z)|\bigg(
\frac{1+|\log|z||}{|z|^{d+\overline{\A}(y;x)}}\I_{\{|z|\le 1\}}+\frac{1+|\log|z||}{|z|^{d+\underline{\A}(y;x)}}\I_{\{|z|> 1\}}\bigg)\,dz\\
&\p\big(|x-y|^{\B_0}\wedge 1\big)\cdot \big(I_1(t,x,y,w)+I_2(t,x,y,w)+\rho^{y,0}_1(t,w)I_3(t,x,y)\big),
\end{align*}
where in the second inequality we have used \eqref{p1-1-2}, and $I_1$, $I_2$ and $I_3$ are defined in \eqref{l1-1-0}.
Hence, Lemma \ref{l1-1} yields that
\begin{equation}\label{l1-2-2}
\begin{split}
&\big|\big( \LL^x- \LL^y\big)p^y(t,w)\big|\\
&\p  \!\big(|x-y|^{\B_0}\wedge 1\big) \!\cdot\!
\begin{cases}
(1+|\log t|)t^{1-({\overline{\A}(y;x)}/{\A(y)})}\rho^{y,0}_0(t,w),\ \ &\ |w|\le {1}/{2},\\
\big(1+|\log |w||+(1+|\log t|)t^{1-({\overline{\A}(y;x)}/{\A(y)})}\big)\rho^{y,0}_0(t,w),\ \
&\ |w|>{1}/{2}.
\end{cases}
\end{split}
\end{equation}
Note that due to \eqref{e1-3},
\begin{equation*}
\overline{\A}(y;x)\le \A(y)+c_0|x-y|^{\B_0}\le
\A(y)+c_0R_1^{\B_0}\quad\textrm{ for all } x,y\in\R^d\textrm{ with }|x-y|\le R_1.
 \end{equation*}
Therefore, choosing $R_1\in(0,{1}/{2})$ small enough such that
$c_0R_1^{\B_0} /\alpha_1 < \gamma$, we have
\begin{equation*}
(1+|\log t|)t^{1-({\overline{\A}(y;x)}/{\A(y)})}\p t^{-\G}\quad\textrm{ when }|x-y|\le R_1.
\end{equation*}
This, along with \eqref{l1-2-2}, immediately implies \eqref{l1-2-1}.

Estimate  \eqref{l1-2-1a} follows from \eqref{l1-2-1} by taking $w=x-y$.

\smallskip

(ii) Suppose, in addition, that  there are some $y_0\in\R^d$ and $r_0\in(0,\infty]$ such that
 $\alpha(x)=\alpha(y_0)$ for all $x \in B(y_0,r_0)$. Then \eqref{e:pp} is reduced to
\begin{equation*}
\begin{split}
&\bigg|\frac{\K(x,z)}{|z|^{d+\A(x)}}-\frac{\K(y_0,z)}{|z|^{d+\A(y_0)}}\bigg|\p \big(|x-y|^{\beta_0}\wedge
1\big)\cdot\bigg(\frac{1}{|z|^{d+\A(y_0)}}\I_{\{|z|\le 1\}}+ \frac{1}{|z|^{d+
\A(y_0)}}\I_{\{|z|> 1\}}\bigg)
\end{split}
\end{equation*} for all $x\in B(y_0,r_0)$.
Having this at hand and repeating the same argument above, we can
easily see that the second assertion holds in this case.
\end{proof}

Now, we are in the position to present the

\begin{proof}[Proof of Proposition $\ref{l2-1}$]
Let $R_1=R_1(\A,\K,\G)$, where $R_1(\A,\K,\G)$ is the constant in
Lemma \ref{l1-2}. According to Lemma \ref{l1-2},  we only need to
treat the case that $|x-y|> R_1$.

(1)  We first assume that $\K(x,z)$ is independent of $z$.

By \eqref{e2-3}, for every $t\in (0,1]$ and $x,y \in \R^d$ with $|x-y|>R_1$, we have
\begin{equation}\label{l2-1-3a}
\begin{split}
 | \LL^x p^y(t,\cdot)(x-y)|
& \p \int_{\{|z|\le {R_1}/{2}\}}|\delta_{p^y}(t,x-y;z)|\frac{1}{|z|^{d+\A(x)}}\,dz\\
&\quad +\int_{\{|z|>{R_1}/{2}\}} \big(p^y(t,x-y\pm z)+p^y(t,x-y)\big)\frac{1}{|z|^{d+\A(x)}}\,dz\\
&=:J_1+J_2.
\end{split}
\end{equation}
Applying \eqref{p1-1-1} and using the mean value theorem, we can get
\begin{align*}
J_1&\p \int_{\{|z|\le {R_1}/{2}\}}|\nabla^2 p^y(t,\cdot)(x-y+\theta_{x,y,z}z)||z|^{2-d-\A(x)}\, dz\\
&\p \int_{\{|z|\le {R_1}/{2}\}}\frac{t}{(t^{{1/\A(y)}}+|x-y+\theta_{x,y, z}z|)^{d+\A(y)+2}}
|z|^{2-d-\A(x)}\, dz\\
&\p \frac{t}{|x-y|^{d+\A(y)+2}}\int_{\{|z|\le
{R_1}/{2}\}}|z|^{2-d-\A(x)}\, dz\p
\frac{t}{|x-y|^{d+\A(y)+2}}\p \frac{1}{|x-y|^{d+\A_1}}=\rho_0^{y,0}(t,x-y),
\end{align*}
where in the first inequality $\theta_{x,y,z}$ is a constant depending on
$x,y,z$ such that $|\theta_{x,y,z}|\le 1$, and the third inequality follows
from the fact that $$|x-y+\theta_{x,y,z}z|\ge {|x-y|}/{2}\ge
{R_1}/{2} \ge C_0 t^{{1/\A(y)}}\quad\textrm{ for any } z\in \R^d\textrm{ with } |z|\le
{R_1}/{2}.$$
On the other hand, according to \eqref{p1-1-1-1-0} and \eqref{r1-1-1}, we have
\begin{equation}\label{e:ref1}
\begin{split}
J_2&\p \int_{\{|z|>{R_1}/{2}\}}\rho^{y,0}_1(t,x-y\pm z)\frac{1}{|z|^{d+\A(x)}}\,dz +
\rho^{y,0}_1(t,x-y) \int_{\{|z|>{R_1}/{2}\}}\frac{1}{|z|^{d+\A(x)}}\,dz \p  \rho^{y,0}_0(t,x-y).
\end{split}
\end{equation}
Hence, for every $t\in (0,1]$ and $x,y\in\R^d$ with $|x-y|>R_1$,
$$| \LL^x p^y(t,\cdot)(x-y)|\p\rho^{y,0}_0(t,x-y).$$

Following the same argument above, we have for every $t\in (0,1]$ and $x,y\in\R^d$ with $|x-y|>R_1$,
\begin{equation}\label{l2-1-4}
| \LL^y p^y(t,\cdot)(x-y)|\p \frac{1}{|x-y|^{d+\A_1}}\p \rho^{y,0}_0(t,x-y).\end{equation}
Therefore, for every $t\in (0,1]$ and $x,y\in\R^d$ with $|x-y|>R_1$,
\begin{align*}
|q_0(t,x,y)|&\le | \LL^x p^y(t,\cdot)(x-y)|+| \LL^y p^y(t,\cdot)(x-y)|\p
\rho^{y,0}_0(t,x-y).
\end{align*}
By now we have shown the second case in the estimate \eqref{l2-1-2}.
From \eqref{l1-2-1b} and \eqref{l2-1-2}, we arrive at
\eqref{l2-1-33a} immediately.

(2) Now we suppose the general case that $\K(x,z)$ may depend on $z$. We still define $J_1$ and $J_2$ by those in \eqref{l2-1-3a}. It is not
difficult to verify that \eqref{e:ref1} also holds for this case.
On the other hand, according to \eqref{p1-1-2}, it holds that
\begin{equation}\label{l2-1-5}
\begin{split}
J_1&\p \int_{\{|z|\le {R_1}/{2}\}}\big[(t^{-{2}/{\A(y)}}|z|^2)\wedge 1\big]\cdot\rho^{y,0}_1(t,x-y\pm z)\frac{1}{|z|^{d+\A(x)}}\,dz\\
&\quad +\rho_1^{y,0}(t,x-y)\int_{\{|z|\le {R_1}/{2}\}}\big[(t^{-{2}/{\A(y)}}|z|^2)\wedge 1\big]\cdot\frac{1}{|z|^{d+\A(x)}}\,dz\\
&\p\frac{t^{1-{2}/{\A(y)}}}{|x-y|^{d+\A_1}}\int_{\{|z|\le
t^{{1/\A(y)}}\}}|z|^{2-d-\A(x)}\,dz  +
\frac{t}{|x-y|^{d+\A_1}}\int_{\{t^{{1/\A(y)}}\le |z|\le {R_1}/{2}\}}|z|^{-d-\A(x)}\,dz\\
&\quad + \frac{t}{|x-y|^{d+\A_1}}
\int_{\{|z|\le {R_1}/{2}\}}\big[(t^{-{2}/{\A(y)}}|z|^2)\wedge 1\big]|z|^{-d-\A(x)}\,dz\\
&\p \frac{t^{1-({\overline{\alpha}(y;x)}/{\alpha(y)})}}{|x-y|^{d+\alpha_1}}\le  \frac{t^{1-(\A_2/\A_1)}}{|x-y|^{d+\alpha_1}}.
\end{split}
\end{equation}
Here the second step follows from
\begin{equation*}
\rho^{y,0}_0(t,x-y+ z)\p \frac{1}{|x-y|^{d+\A_1}}\quad\textrm{
when } |x-y|> R_1\textrm{ and }|z|\le {R_1}/{2},
\end{equation*}
which is due to the fact that for $ |x-y|> R_1$ and $ |z|\le {R_1}/{2}$
$$|x-y+ z| \ge |x-y|-|z|\ge  {|x-y|}/{2}\ge
{R_1}/{2}\ge C_0t^{{1/\A(y)}}.$$

Note also that \eqref{l2-1-4} is still true for this case. Combining all the estimates above, we obtain
\begin{align*}
|q_0(t,x,y)|&\p
| \LL^x p^y(t,x-y)|+| \LL^y p^y(t,x-y)| \p J_1+J_2+| \LL^yp^y(t,x-y)|\p t^{1-({\A_2}/{\A_1})}\rho_0^{y,0}(t,x-y),
\end{align*}
where in the last step we used the fact that $t^{-{\overline{\A}(y;x)}/{\A(y)}}\le t^{-{\A_2}/{\A_1}}$ for all $t \in (0,1].$
Thus, the second case of \eqref{l2-1-1} is proved.
\end{proof}

According to the proof of Proposition \ref{l2-1}, we also have the following estimates for $q_0(t,x,y)$.
\begin{proposition}\label{p2-1} \begin{itemize}
\item [(1)] Suppose $\K(x,z)$ is independent of $z$. Then there exists a constant $c_1:=c_1(\A,\K)>0$ such that
for every $t \in (0,1]$ and $x,y\in \R^d$,
\begin{equation}\label{p2-1-1}\begin{split}
|q_0(t,x,y)|&\le
\begin{cases}
c_1t^{-1-({d}/{\A(y)})},\quad & |x-y|\le t^{{1}/{\A(y)}},\\
\frac{c_1}{|x-y|^{d+\A(x)}}+ \frac{c_1}{|x-y|^{d+\A(y)}},\quad & |x-y|> t^{{1}/{\A(y)}}
\end{cases}\\
&\le
\begin{cases}
c_1t^{-1-({d}/{\A(y)})},\quad & |x-y|\le t^{{1}/{\A(y)}},\\
\frac{c_1}{|x-y|^{d+\A_2}\wedge |x-y|^{d+\A_1}},\quad & |x-y|> t^{{1}/{\A(y)}}.
\end{cases}\end{split}
\end{equation}

\item[(2)] Suppose $\K(x,z)$ depends on $z$ and ${\B_0}/{\A_2}>({\A_2}/{\A_1})-1$. Then there exists a constant $c_2:=c_2(\A,\K)>0$ such that
for every $t \in (0,1]$ and $x,y\in \R^d$,
\begin{equation}\label{p2-1-1a}
|q_0(t,x,y)|\le
\begin{cases}
c_2t^{-1-({d}/{\A(y)})},\quad &|x-y|\le t^{{1}/{\A(y)}},\\
\frac{c_1 t^{1-(\A_2/\A_1)}}{|x-y|^{d+\A_2}\wedge |x-y|^{d+\A_1}},\quad & |x-y| > t^{{1}/{\A(y)}}.
\end{cases}
\end{equation}
\end{itemize}
\end{proposition}
\begin{proof}
(1) We first suppose that $\K(x,z)$ is independent of $z$.

{\bf Case (a): $|x-y|\le t^{{1}/{\A(y)}}.$}

According to \eqref{l1-2-1a}, we can choose $0<\G<{\B_0}/{\A_2}$ and adjust the constants properly to get that
\begin{equation}\label{p2-1-2a}
\begin{split}
|q_0(t,x-y)|&\p \rho^{y,\B_0}_{-\G}(t,x-y)\p t^{-\G}|x-y|^{\B_0}\cdot\rho^{y,0}_{0}(t,x-y)\\
&\p t^{({\B_0}/{\A_2})-\G}\rho^{y,0}_0(t,x-y)\p \rho^{y,0}_0(t,x-y)\p t^{-({d}/{\A(y)})-1},
\end{split}
\end{equation}
where in the third inequality we have used the fact that $|x-y|\le t^{{1}/{\A(y)}}\le t^{{1}/{\A_2}}$, and the
fourth inequality follows from $\G<{\B_0}/{\A_2}$.

{\bf Case (b): $|x-y|>t^{{1}/{\A(y)}}.$}

By \eqref{e2-3}, it holds that
\begin{equation}\label{p2-1-2}
\begin{split}
| \LL^x p^y(t,\cdot)(x-y)|&
\p \int_{\{|z|\le {|x-y|}/{2}\}}\left|
\delta_{p^y}\left(t,x-y;z\right)\right|\frac{1}{|z|^{d+\A(x)}}\,dz\\
&\quad +\int_{\{|z|> {|x-y|}/{2}\}}\left|
\delta_{p^y}\left(t,x-y;z\right)\right|\frac{1}{|z|^{d+\A(x)}}\,dz\\
&=:J_1+J_2.
\end{split}
\end{equation}
Then, using the mean value theorem, we have
\begin{align*}
J_1&\p \int_{\{|z|\le {|x-y|}/{2}\}}
\left|\nabla^2 p^y(t,x-y+\theta_{x,y,z} z)\right|\cdot |z|^2\frac{1}{|z|^{d+\A(x)}}\,dz\\
&\p \int_{\{|z|\le {|x-y|}/{2}\}}
\frac{t}{\left(t^{{1}/{\A(y)}}+|x-y+\theta_{x,y,z} z|\right)^{d+\A(y)+2}}
\cdot |z|^{2-d-\A(x)}\,dz\\
&\p \frac{t}{\left(t^{{1}/{\A(y)}}+|x-y|\right)^{d+\A(y)+2}}
\int_{\{|z|\le {|x-y|}/{2}\}} |z|^{2-d-\A(x)}\,dz\\
&\p \frac{|x-y|^{\A(y)}}{|x-y|^{d+\A(y)+2}}\cdot |x-y|^{2-\A(x)}\p \frac{1}{|x-y|^{d+\A(x)}},
\end{align*}
where in the first inequality  $\theta_{x,y,z}$ is a constant depending on $x,y,z$ such that
$|\theta_{x,y,z}|\le 1$, in the second inequality we have used \eqref{p1-1-1}, the third inequality follows
from the fact that $|x-y+\theta_{x,y,z} z|\ge |x-y|-|z|\ge {|x-y|}/{2}$, and the fourth inequality is due to
the fact that $t\le |x-y|^{\A(y)}$.

On the other hand,
by \eqref{p1-1-2} and the fact that $|z|>{|x-y|}/{2}\ge t^{{1}/{\A(y)}}/2$, we have
$ \left|\delta_{p^y}(t,x-y;z)\right|\p \rho^{y,0}_1(t,x-y\pm z)+\rho^{y,0}_1(t,x-y),$ and so \begin{align*}
J_2&\p  \int_{\{|z|>{|x-y|}/{2}\}}\!\!
\rho^{y,0}_1(t,x-y\pm z)\frac{1}{|z|^{d+\A(x)}}\,dz +
\rho^{y,0}_1(t,x-y)\int_{\{|z|>{|x-y|}/{2}\}}\!\!
\frac{1}{|z|^{d+\A(x)}}\,dz=:J_{21}+J_{22}.
\end{align*}
It holds that
\begin{align*}
J_{21}&\p \frac{1}{|x-y|^{d+\A(x)}} \int_{\{|z|>{|x-y|}/{2}\}}\rho^{y,0}_1(t,x-y\pm z)\,dz\p \frac{1}{|x-y|^{d+\A(x)}}\int \rho^{y,0}_1(t,x-y\pm z)\,dz\p \frac{1}{|x-y|^{d+\A(x)}},
\end{align*}
where in the first inequality we have used following fact
$$ \frac{1}{|z|^{d+\A(x)}}\p \frac{1}{|x-y|^{d+\A(x)}}\quad\textrm{ for } |z|>{|x-y|}/{2},$$
and the last inequality follows from \eqref{l3-1-1-00}.
Furthermore, it is easy to verify that
\begin{align*}
J_{22}&\p \rho^{y,0}_1(t,x-y)
\int_{\{|z|>{|x-y|}/{2}\}}\frac{1}{|z|^{d+\A(x)}}\, dz\p \frac{t}{|x-y|^{d+\A(x)+\A(y)}}\p \frac{1}{|x-y|^{d+\A(x)}},
\end{align*} where in the last inequality we used $|x-y|\le t^{{1}/{\A(y)}}$.

Combining all these estimates above, we arrive at
$$ \left| \LL^xp^y(t,\cdot)(x-y)\right|\p \frac{1}{|x-y|^{d+\A(x)}}.$$
Following the same arguments as above, we also have
$$ \left| \LL^yp^y(t,\cdot)(x-y)\right|\p \frac{1}{|x-y|^{d+\A(y)}}.$$
Hence,
\begin{align*}
|q_0(t,x,y)|&\p
\left| \LL^x p^y(t,\cdot)(x-y)\right|+\left| \LL^yp^y(t,\cdot)(x-y)\right|\p \frac{1}{|x-y|^{d+\A(x)}}+\frac{1}{|x-y|^{d+\A(y)}},
\end{align*}
from which we can get \eqref{p2-1-1} immediately.

(2) Next we consider the case that $\K(x,z)$ depends on $z$.
With the estimate \eqref{l2-1-1} and the condition
${\B_0}/{\A_2}>({\A_2}/{\A_1})-1$ at hand, we can follow the same argument as in \eqref{p2-1-2a} to verify the
upper bound in \eqref{p2-1-1a} for the case that $|x-y|\le t^{{1}/{\A(y)}}$.

When $|x-y|>t^{{1}/{\A(y)}}$, we still define
$J_1$ and $J_2$ via \eqref{p2-1-2}. By carefully tracking the proof above, we find that the argument for the
estimates of $J_2$ in part (1) still works. So it remains to consider upper bound for
$J_1$.
According to \eqref{p1-1-2}, we have
\begin{align*}
J_1&\p \int_{\{|z|\le {|x-y|}/{2}\}}
\big[\big(t^{-{2}/{\A(y)}}|z|^2\big)\wedge 1\big] \cdot
\big(\rho_1^{y,0}(t,x-y\pm z)+\rho_1^{y,0}(t,x-y)\big)\frac{1}{|z|^{d+\A(x)}}\,dz\\
&\p \frac{t}{|x-y|^{d+\A(y)}}\int
\big[\big(t^{-{2}/{\A(y)}}|z|^2\big)\wedge 1\big] \cdot
\frac{1}{|z|^{d+\A(x)}}\,dz\\
&\p \frac{t}{|x-y|^{d+\A(y)}}\cdot \bigg(\int_{\{|z|\le t^{{1}/{\A(y)}}\}}
t^{-{2}/{\A(y)}}|z|^{2-d-\A(x)}\,dz+\int_{\{|z|> t^{{1}/{\A(y)}}\}}|z|^{-d-\A(x)}\,dz\bigg)\\
&\p t^{1-({\A(x)}/{\A(y)})}\cdot \frac{1}{|x-y|^{d+\A(y)}}\p t^{1-({\A_2}/{\A_1})}\cdot \frac{1}{|x-y|^{d+\A(y)}},
\end{align*}
where the second inequality above follows from
$$ \rho^{y,0}_1(t,x-y\pm z)\p \frac{t}{|x-y|^{d+\A(y)}}\quad\textrm{ for any }|z|\le {|x-y|}/{2}.$$
Combining this with the estimate of $J_2$ in part (1) yields
\begin{align*}
|q_0(t,x,y)|&\p
\left| \LL^x p^y(t,\cdot)(x-y)\right|+\left| \LL^xp^y(t,\cdot)(x-y)\right|\p t^{1-({\A_2}/{\A_1})}\cdot \bigg(\frac{1}{|x-y|^{d+\A(x)}}+\frac{1}{|x-y|^{d+\A(y)}}\bigg),
\end{align*}
which implies that \eqref{p2-1-1a} holds for every $x,y \in \R^d$ with $|x-y|>t^{{1}/{\A(y)}}$. The proof is complete.
\end{proof}

\subsection{Continuity of $q_0(t,x,y)$}
\begin{proposition}\label{p3-1-1a-p}
 \begin{itemize}
 \item[(1)] If $\kappa(x,z)=\K (x) $ is independent of $z$, then for any $\gamma,\theta>0$ with $\gamma<\theta<{\beta_0}/{\A_2}$ and any
$\varepsilon>0$,
there exists a positive constant $c_1:=c_1(\A,\K, \G, \theta,\varepsilon)$ such that
for every $(t,x,x',y)\in (0,1]\times \R^d\times \R^d \times \R^d$,
\begin{equation}\label{p3-1-1a-1}
\begin{split}
&|q_0(t,x,y)-q_0(t,x',y)|\\
&\le c_1\left(|x-x'|^{\A_1(\B_0-\A_2\theta)/{\A(x)}}\wedge 1\right)\\
&\quad \times\Big[\left(\rho^{y,0}_{\theta-\gamma}+\rho^{y,\B_0}_{\theta-\gamma-({\B_0}/{\A(x)})}\right)(t,x-y)
+\left(\rho^{y,0}_{\theta-\gamma}+\rho^{y,\B_0}_{\theta-\gamma-({\B_0}/{\A(x')})}\right)(t,x'-y)\\
&\qquad\quad  +|x-y|^{\varepsilon}\I_{\{|x-y|>R_0\}}\rho^{y,0}_{\theta-\gamma}(t,x-y) +|x'-y|^{\varepsilon}\I_{\{|x'-y|>R_0\}}\rho^{y,0}_{\theta-\gamma}(t,x'-y)\Big].
\end{split}
\end{equation}

\item[(2)]
For general $\K (x, z)$, if $({\A_2}/{\A_1})-1<{\B_0}/{\A_2}$,
then for any $\gamma,\theta>0$ such that $ ({\A_2}/{\A_1})-1+\gamma =:\G_2 <\theta<{\B_0}/{\A_2}$ and
$\varepsilon>0$, there exists a positive constant $c_2:=c_2(\A,\K,\G,\theta,\varepsilon)$ such that for every $(t,x,x',y)\in (0,1]\times \R^d\times \R^d \times \R^d$,
\begin{equation}\label{p3-1-2a-1}
\begin{split}
& |q_0(t,x,y)-q_0(t,x',y)|\\
&\le c_2\left(|x-x'|^{\A_1(\B_0-\A_2\theta)/{\A(x)}}\wedge 1\right)
\Big[\left(\rho^{y,0}_{\theta-\G_2 }+\rho^{y,\B_0}_{\theta-\G_2-({\B_0}/{\A(x)})}\right)(t,x-y)\\
&\qquad\quad +\left(\rho^{y,0}_{\theta-\G_2}+\rho^{y,\B_0}_{\theta-\G_2 -({\B_0}/{\A(x')})}\right)(t,x'-y)
 +|x-y|^{\varepsilon}\I_{\{|x-y|>R_0\}}\rho^{y,0}_{\theta-\G_2 }(t,x-y)\\
 & \qquad \quad
+|x'-y|^{\varepsilon}\I_{\{|x'-y|>R_0\}}\rho^{y,0}_{\theta-\G_2 }(t,x'-y)\Big].
\end{split}
\end{equation}
\end{itemize}

\end{proposition}

\begin{proof} In the remainder of this paper, we denote $\nabla p^y(t,\cdot)(x)$ by $\nabla p^y(t,x)$ for simplicity.

(1) We first consider the case that $\K(x,z)$ is independent of $z$. Let $R_1:=R_1(\A,\K,\G)$ be a positive constant to be determined later. The proof is split into the following five different cases.

{\bf Case (a): $|x-x'|>R_1$.}

By \eqref{l2-1-2} and the condition $\B_0>\A_2\theta$, it holds that
\begin{align*} |q_0(t,x,y)-q_0(t,x',y)|
&\le |q_0(t,x,y)|+|q_0(t,x',y)|\p \rho_{-\gamma}^{y,\B_0}(t,x-y)+\rho_{-\gamma}^{y,\B_0}(t,x'-y)\\
&\p \big(|x-x'|^{{\A_1(\B_0-\A_2\theta)/{\A(x)}}}\wedge1\big)\big(\rho_{-\gamma}^{y,\B_0}(t,x-y)+\rho_{-\gamma}^{y,\B_0}(t,x'-y)\big).\end{align*}

{\bf Case (b): $C_0t^{1/\A(y)}\le |x-x'|\le R_1$ for some small
constant $C_0>0$.}

According to \eqref{l2-1-2} again, we have
\begin{align*} |q_0(t,x,y)|\p& \rho_{-\gamma}^{y,\B_0}(t,x-y)\p t^{{(\B_0-\A(x)\theta)}/{\A(x)}} \rho_{-\gamma+\theta-({\B_0}/{\A(x)})}^{y,\B_0}(t,x-y) \\
\p&|x-x'|^{{{\A(y)(\B_0-\A(x)\theta)}/{\A(x)}}} \rho_{-\gamma+\theta-({\B_0}/{\A(x)})}^{y,\B_0}(t,x-y) \\
\p&|x-x'|^{{\A_1(\B_0-\A_2\theta)/{\A(x)}}} \rho_{-\gamma+\theta-({\B_0}/{\A(x)})}^{y,\B_0}(t,x-y),  \end{align*}
where the third inequality is due to $t\p |x-x'|^{\A(y)}$.
Then
\begin{align*}&|q_0(t,x,y)-q_0(t,x',y)|\\
&\le |q_0(t,x,y)|+|q_0(t,x',y)|\\
&\p |x-\!x'|^{{\A_1(\B_0-\A_2\theta)/{\A(x)}}} \! \rho_{-\gamma+\theta-({\B_0}/{\A(x)})}^{y,\B_0}(t,x-y)  +
|x-\!x'|^{{\A_1(\B_0-\A_2\theta)/{\A(x')}}}\!
\rho_{-\gamma+\theta-({\B_0}/{\A(x')})}^{y,\B_0}(t,x'\!-y)\\
&\p |x-\!x'|^{{\A_1(\B_0-\A_2\theta)/{\A(x)}}}\Big(
\rho_{-\gamma+\theta-({\B_0}/{\A(x)})}^{y,\B_0}(t,x\!-y)+
\rho_{-\gamma+\theta-({\B_0}/{\A(x')})}^{y,\B_0}(t,x'\!-y)\Big).\end{align*}
Here in the last inequality we have used the fact that for $ x,x'\in
\R^d$ with $|x-x'|\le R_1,$
\begin{equation}\label{p3-1-3}
|x-x'|^{{\A_1(\B_0-\A_2\theta)/{\A(x')}}}\p
|x-x'|^{{\A_1(\B_0-\A_2\theta)/{\A(x)}}},
\end{equation}
which can be verified by following the argument of \eqref{l3-1-1a}.

Next, we mainly treat the case that $|x-x'|\le C_0t^{1/\A(y)}$,
which we divide into three cases.

{\bf Case (c): $|x-x'|\le C_0t^{1/\A(y)}$ and $|x-y|\ge R_1$.}

By the definition of $q_0(t,x,y)$,
\begin{equation}\label{p3-1-2aa}
\begin{split}
 |q_0(t,x,y)-q_0(t,x',y)|
&=\frac{1}{2}\bigg|\int_{\R^d} \delta_{p^y}(t,x-y;z)\bigg(\frac{\kappa(x,z)}{|z|^{d+\alpha(x)}}-\frac{\kappa(y,z)}{|z|^{d+\alpha(y)}} \bigg)\,dz\\
&\qquad\quad-\int_{\R^d} \delta_{p^y}(t,x'-y;z)\bigg(\frac{\kappa(x',z)}{|z|^{d+\alpha(x')}}-\frac{\kappa(y,z)}{|z|^{d+\alpha(y)}} \bigg)\,dz\bigg|\\
&\le \frac{1}{2}\int_{\R^d} |\delta_{p^y}(t,x-y;z)-\delta_{p^y}(t,x'-y;z)| \bigg|\frac{\kappa(x,z)}{|z|^{d+\alpha(x)}}-\frac{\kappa(y,z)}{|z|^{d+\alpha(y)}} \bigg|\,dz\\
&\quad+ \frac{1}{2}\int_{\R^d} |\delta_{p^y}(t,x'-y;z)| \bigg|\frac{\kappa(x',z)}{|z|^{d+\alpha(x')}}-\frac{\kappa(x,z)}{|z|^{d+\alpha(x)}} \bigg|\,dz\\
&=:J_1+J_2.
\end{split}
\end{equation}

Note that \begin{equation}\label{e:conpr0}
\begin{split} \delta_{p^y}(t,x;z)=&p^y(t,x+z)+p^y(t,x-z)-2p^y(t,x)\\
=&\int_0^1 \frac{d}{d\theta}\Big(p^y(t,x+\theta z)+p^y(t,x-\theta z)\Big)\,d\theta\\
=&\int_0^1 \langle \nabla p^y(t,x+\theta z)-\nabla p^y(t,x-\theta z),z\rangle \,d\theta\\
=&\int_0^1 \Big\langle \int_{-1}^1 \frac{d}{d \theta'}\left(\nabla p^y(t,x+\theta'\theta z)\right)d\theta',z \Big\rangle \,d\theta\\
=&\int_0^1\int_{-1}^1 \nabla^2 p^y(t,x+\theta'\theta z)(\theta z, z)d\theta'\,d\theta,\end{split}\end{equation}
and so
\begin{align*}  \delta_{p^y}(t,x;z)-\delta_{p^y}(t,x';z)
&=\int_0^1\int_{-1}^1 \big(\nabla^2 p^y(t,x+\theta'\theta z)- \nabla^2 p^y(t,x'+\theta'\theta z)\big)(\theta z, z)\,d\theta'\,d\theta\\
&=\int_0^1\int_{-1}^1\int_0^1  \frac{d}{d\theta''}\Big(\nabla^2 p^y(t,x'+\theta''(x-x')+\theta'\theta z)\Big)
(\theta z, z)\,d\theta''d\theta'\,d\theta\\
&=\int_0^1\int_{-1}^1\int_0^1  \nabla^3 p^y(t,x'+\theta''(x-x')+\theta'\theta z)
(x-x',\theta z, z)\,d\theta''d\theta'\,d\theta,\end{align*} which implies that
\begin{align*}
 |\delta_{p^y}(t,x;z)-\delta_{p^y}(t,x';z) |
&\le \Big(\int_0^1\int_{-1}^1\int_0^1 |\nabla^3 p^y(t,x'+\theta''(x-x')+\theta'\theta z)|\,\,d\theta''d\theta'\,d\theta\Big)
\cdot |z|^2|x-x'|.
\end{align*}
For every $x,x',z\in\R^d$ with $|x-x'|\le C_0 t^{{1/\A(y)}}$, $|x|\ge cR_1$, $|x'|\ge cR_1$ and $|z|\le {cR_1}/{4}$
for $C_0$ small enough, there exists a constant $c'>0$ such that for every $\theta\in(0,1)$, $\theta'\in(-1,1)$ and $\theta''\in(0,1)$, we have
$|x'+\theta''(x-x')+\theta'\theta z|\ge c'|x'|\ge c''R_1,$ which along with \eqref{p1-1-1} yields that
$$|\nabla^3 p^y(t,x'+\theta''(x-x')+\theta'\theta z)| \p \frac{t}{|x'|^{d+\A(y)+3}}.$$ Combining with all the estimates above, we arrive at that
for $|x-x'|\le C_0t^{1/\A(y)}$, $|x|\ge cR_1$, $|x'|\ge cR_1$ and $|z|\le {cR_1}/{4}$,
\begin{equation}\label{e:conpr1} |\delta_{p^y}(t,x;z)-\delta_{p^y}(t,x';z) |\p \frac{t|z|^2|x-x'|}{|x'|^{d+\A(y)+3}}.
   \end{equation}

Since $|x-x'|\le C_0t^{{1/\A(y)}}\le C_0$, $|x'-y|\ge cR_1$ holds for some constant $c>0$.
By taking $x=x-y$ and $x'=x'-y$ in \eqref{e:conpr1}, we obtain that for all $z\in\R^d$ with $|z|\le {cR_1}/{4}$,
\begin{equation}\label{e:conpr2} |\delta_{p^y}(t,x-y;z)-\delta_{p^y}(t,x'-y;z) |\p \frac{t|z|^2|x-x'|}{|x'-y|^{d+\A(y)+3}}.  \end{equation}

Let \begin{equation}\label{p3-1-3a}
\begin{split}
J_1=&\frac{1}{2}\int_{\{|z|\le {cR_1}/{4}\}} |\delta_{p^y}(t,x-y;z)-\delta_{p^y}(t,x'-y;z)| \bigg|\frac{\kappa(x,z)}{|z|^{d+\alpha(x)}}-\frac{\kappa(y,z)}{|z|^{d+\alpha(y)}} \bigg|\,dz\\
&+ \frac{1}{2}\int_{\{|z|> {cR_1}/{4}\}} |\delta_{p^y}(t,x-y;z)-\delta_{p^y}(t,x'-y;z)| \bigg|\frac{\kappa(x,z)}{|z|^{d+\alpha(x)}}-\frac{\kappa(y,z)}{|z|^{d+\alpha(y)}} \bigg|\,dz\\
=&:J_{11}+J_{12}.
 \end{split}
\end{equation}
 By \eqref{e:conpr2},
 \begin{align*}J_{11}\p &|x-x'|\frac{t}{|x'-y|^{d+\A(y)+3}} \int_{\{|z|\le {cR_1}/{4}\}}\bigg( \frac{|z|^2}{|z|^{d+\A(x)}}+\frac{|z|^2}{|z|^{d+\A(y)}} \bigg)\,dz\\
 \p& |x-x'|\frac{t}{|x'-y|^{d+\A(y)+3}} \p |x-x'|\frac{1}{|x'-y|^{d+\A_1}},
 \end{align*} where the last inequality is due to $|x'-y|>cR_1$ and $t\le 1$.
On the other hand, according to \eqref{r1-1-1}, we have for any $x,y\in\R^d$ with $|x-y|>R_1$,
\begin{equation}\label{l3-1-3b}
\begin{split}
&\int_{\{|z|>{cR_1}/{4}\}} \rho_1^{y,0}(t,x-y\pm z) \frac{1}{|z|^{d+\A (x)}}\,dz \p \frac{1}{|x-y|^{d+\A_1}},\\
&\int_{\{|z|>{cR_1}/{4}\}} \rho_1^{y,0}(t,x-y) \frac{1}{|z|^{d+\A (x)}}\,dz \p \frac{1}{|x-y|^{d+\A_1}}.
\end{split}
\end{equation}
Combining \eqref{p1-1-4} with \eqref{l3-1-3b} and the fact that $|z|>{cR_1}/{4}$ implies $|z|>C_1t^{{1/\A(y)}}$, we arrive at
\begin{align*}
J_{12}&\p t^{-{1/\A(y)}}|x-x'| \int_{\{|z|>{cR_1}/{4}\}}\Big(\rho_1^{y,0}(t,x-y\pm z)+\rho_1^{y,0}(t,x-y)\\
&\qquad \qquad\qquad\qquad +\rho_1^{y,0}(t,x'-y\pm z)+\rho_1^{y,0}(t,x'-y)\Big)
\bigg(\frac{1}{|z|^{d+\alpha(x)}}+\frac{1}{|z|^{d+\alpha(y)}} \bigg)\,dz\\
&\p t^{-{1/\A(y)}}|x-x'| \left( \frac{1}{|x-y|^{d+\A_1}}+
\frac{1}{|x'-y|^{d+\A_1}} \right). \end{align*} By both of the
estimates above, we obtain that
 \begin{align*}J_1&\p t^{-{1/\A(y)}}|x-x'| \big( \rho_0^{y,0}(t,x-y)+\rho_0^{y,0}(t,x'-y)\big)\p  |x-x'| \big( \rho_{-{1/\A(y)}}^{y,\B_0}(t,x-y)+\rho_{-{1/\A(y)}}^{y,\B_0}(t,x'-y)\big), \end{align*}
where the last inequality follows from the facts that $|x-y|\ge cR_1$ and $|x'-y|\ge c R_1$.

Let \begin{equation}\label{p3-1-4a}
\begin{split}
J_2=&\int_{\{|z|\le {cR_1}/{4}\}} |\delta_{p^y}(t,x'-y;z)| \bigg|\frac{\kappa(x',z)}{|z|^{d+\alpha(x')}}-\frac{\kappa(x,z)}{|z|^{d+\alpha(x)}} \bigg|\,dz\\
&+ \int_{\{|z|> {cR_1}/{4}\}} |\delta_{p^y}(t,x'-y;z)| \bigg|\frac{\kappa(x',z)}{|z|^{d+\alpha(x')}}-\frac{\kappa(x,z)}{|z|^{d+\alpha(x)}} \bigg|\,dz\\
=&:J_{21}+J_{22}.
\end{split}
\end{equation}
Since $|x'-y|>cR_1$, by \eqref{e:conpr0} and \eqref{p1-1-1}, we find that for any $|z|<{cR_1}/{4} $,
\begin{align*}|\delta_{p^y}(t,x'-y;z)|&\le |z|^2\int_0^1\int_{-1}^1 |\nabla^2 p^y(t,x'-y+\theta'\theta z)|\,d\theta'\,d\theta \p |z|^2 \frac{ t}{ |x'-y|^{d+\A(y)+2}}.\end{align*}
 Thus,
 due to \eqref{e:pp},
 \begin{align*}J_{21} \p& |x-x'|^{\B_0} \frac{t}{|x'-y|^{d+\A(y)+2}}\int_{\{|z|\le {cR_1}/{4}\}} \!\!\frac{1+|\log |z||}{|z|^{d+\overline\A(x;x')}}|z|^2\,dz
 \p  |x-x'|^{\B_0} \frac{1}{|x'-y|^{d+\A_1}},\end{align*} where we have used the facts that $|x'-y|>cR_1$ and $t<1$ in the second inequality.
On the other hand, when $|z|> {cR_1}/{4}$,
$$|\delta_{p^y}(t,x'-y;z)|\p p^y(t,x'-y\pm z)+p^y(t,x'-y).$$
Having this at hand and using \eqref{r1-1-2} and \eqref{p1-1-1-1-0}, we know that
$$\int_{\{|z|>{cR_1}/{4}\}} |\delta_{p^y}(t,x'-y;z)|\frac{1+|\log|z||}{|z|^{d+\underline{\A}(x;x')}}\,dz
\p \frac{\log(1+|x'-y|)}{|x'-y|^{d+\A_1}}.$$
Then, by \eqref{e:pp}, we find that
\begin{align*}
J_{22}&\p |x-x'|^{\B_0}\cdot \int_{\{|z|>{cR_1}/{4}\}}|\delta_{p^y}(t,x'-y;z)|\frac{1+|\log|z||}{|z|^{d+\underline{\A}(x;x')}}\,dz  \p   |x-x'|^{\B_0}\frac{\log(1+|x'-y|)}{|x'-y|^{d+\A_1}}.
\end{align*}
By the estimates for $J_{21}$ and $J_{22}$, we have
\begin{equation*}
J_2\p |x-x'|^{\B_0}\log(1+|x'-y|)\rho^{y,0}_0(t,x'-y).
\end{equation*}

Putting the estimates of $J_1$ and $J_2$ together, we finally arrive at
\begin{align*}
|q_0(t,x,y)-q_0(t,x',y)|\p & |x-x'|\big( \rho_{-{1/\A(y)}}^{y,\B_0}(t,x-y)+\rho_{-{1/\A(y)}}^{y,\B_0}(t,x'-y)\big)\\
&+|x-x'|^{\B_0}\log(1+|x'-y|)\rho^{y,0}_0(t,x'-y).
\end{align*}

{\bf Case (d): $|x-x'|\le C_0t^{{1/\A(y)}}$ and $t^{{1/\A(y)}}\le |x-y|< R_1$.}

We still define $J_1$ and $J_2$ by those in \eqref{p3-1-2aa}. Combining \eqref{p1-1-4} with \eqref{e:pp}, we arrive at
\begin{align*}
J_1 \p & t^{-{1/\A(y)}}|x-x'||x-y|^{\B_0}\Bigg[\int_{\R^d} \left[\left(t^{-{2}/{\A(y)}}|z|^2\right)\wedge 1\right]\\
&\qquad\qquad\qquad\times \big(
\rho_1^{y,0}(t,x-y)+\rho_1^{y,0}(t,x'-y)+\rho_1^{y,0}(t,x-y\pm z)+\rho_1^{y,0}(t,x'-y\pm z)\big)\\
&\qquad\qquad\qquad \times\bigg(\frac{1+|\log|z||}{|z|^{d+\overline{\A}(y;x)}}\I_{\{|z|\le 1\}}+\frac{1+|\log|z||}{|z|^{d+\underline{\A}(y;x)}}\I_{\{|z|> 1\}}\bigg)\,dz\Bigg].
\end{align*}
Noting that $t^{{1/\A(y)}}\le |x-y|< R_1$ and using \eqref{l1-1-1}--\eqref{l1-1-3}, we obtain that
\begin{equation}\label{p3-1-4}
\begin{split}
& \int_{\R^d} \left[\left(t^{-{2}/{\A(y)}}|z|^2\right)\wedge 1\right]\big
(\rho_1^{y,0}(t,x-y\pm z)+\rho_1^{y,0}(t,x-y)\big)\\
&\qquad\qquad \times \bigg(\frac{1+|\log|z||}{|z|^{d+\overline{\A}(y;x)}}\I_{\{|z|\le 1\}}
+\frac{1+|\log|z||}{|z|^{d+\underline{\A}(y;x)}}\I_{\{|z|> 1\}}\bigg)\,dz\\
&\p \left(1+|\log t|\right)t^{-{(\overline{\A}(y;x)-\A(y))}/{\A(y)}}
\rho_0^{y,0}(t,x-y)\p \rho_{-\gamma}^{y,0}(t,x-y),
\end{split}
\end{equation}
where the last step is due to the fact that we can choose the constant $R_1$ small enough such that $$\frac{\overline{\A}(y;x)-\A(y)}{\A(y)}\le
\frac{C|x-y|^{\B_0}}{\A_1}<\gamma\quad \textrm{ for every }|x-y|\le R_1.$$
Since $|x-x'|\le C_0 t^{{1/\A(y)}}$ for some $C_0$ small enough,
$C_2t^{{1/\A(y)}}\le |x'-y|\le C_3R_1$ holds for some positive constants $C_2$ and $C_3$. Hence, \eqref{p3-1-4} still holds with $x$ replaced by $x'$.
Therefore, combining both the estimates above together, we arrive at
\begin{align*}
J_1 &\p |x-x'|\cdot |x-y|^{\B_0}\cdot
\left(\rho^{y,0}_{-\gamma-({1/\A(y)})}(t,x-y)+\rho^{y,0}_{-\gamma-({1/\A(y)})}(t,x'-y)\right)\\
&\p |x-x'|\left(\rho^{y,\B_0}_{-\gamma-({1/\A(y)})}(t,x-y)+\rho^{y,\B_0}_{-\gamma-({1/\A(y)})}(t,x'-y)\right),
\end{align*}
where in the last step we have used the fact that
\begin{equation*}
|x'-y|\ge |x-y|-|x-x'| \ge |x-y|-C_0t^{{1/\A(y)}}\ge (1-C_0)|x-y|.
\end{equation*}

Now we are going to estimate $J_2$. \eqref{p1-1-2}
along with \eqref{e:pp} yields that
\begin{equation}\label{p3-1-5a}
\begin{split}
J_2 & \p |x-x'|^{\B_0}\int_{\R^d}\left[\left(t^{-{2}/{\A(y)}}|z|^2\right)\wedge 1\right]\cdot\Big(
\rho_1^{y,0}(t,x'\!-y)\!+\rho_1^{y,0}(t,x'\!-y\pm z)\Big)\\
&  \qquad\times \bigg(\frac{|\log|z||+1}{|z|^{d+\overline{\A}(x;x')}}\I_{\{|z|\le 1\}}
\!+\frac{|\log|z||+1}{|z|^{d+\underline{\A}(x;x')}}\I_{\{|z|> 1\}}\!\!\bigg)\,dz.
\end{split}
\end{equation}
As noted above, it holds that $C_2t^{{1/\A(y)}}\le |x'-y|\le C_3R_1$. Although the indexes $\overline{\A}(y;x)$ and
$\underline{\A}(y;x)$ are replaced by $\overline{\A}(x;x')$ and $\underline{\A}(x;x')$ respectively, we can
still follow the proof of Lemma \ref{l1-1} to obtain
that
\begin{equation}\label{p3-1-5}
\begin{split}
& \int_{\R^d}\left[\left(t^{-{2}/{\A(y)}}|z|^2\right)\wedge 1\right]\cdot\big(
\rho_1^{y,0}(t,x'-y)+\rho_1^{y,0}(t,x'-y\pm z)\big)\\
&\qquad \qquad\qquad \qquad\qquad  \times  \bigg(\frac{1+|\log|z||}{|z|^{d+\overline{\A}(x;x')}}\I_{\{|z|\le 1\}}
+\frac{1+|\log|z||}{|z|^{d+\underline{\A}(x;x')}}\I_{\{|z|> 1\}}\bigg)\,dz\\
& \p \left(1+|\log t|\right)t^{1-({\overline{\A}(x;x')}/{\A(y)})}\rho_0^{y,0}(t,x'-y).
\end{split}
\end{equation}
Note that when  $|x-y|\le R_1$ and $|x-x'|\le C_0t^{{1/\A(y)}}$,
\begin{align*}
|\overline{\A}(x;x')-\A(y)|&
\le |\overline{\A}(x;x')-\A(x)|+|\A(x)-\A(y)|\le C_4(t^{{\B_0}/{\A(y)}}+R_1^{\B_0}).
\end{align*}
So, by changing $R_1$ properly such that
$$t^{-{(\overline{\A}(x;x')-\A(y))}/{\A(y)}}\p \exp(C_5t|\log t|)\cdot t^{-C_4R_1^{\B_0}}
\p t^{-\G/2}$$ and putting all these estimates into \eqref{p3-1-5a},  we arrive at
$$J_2 \p |x-x'|^{\B_0}\rho_{-\gamma}^{y,0}(t,x'-y).$$
Therefore, according to the estimates for $J_1$ and $J_2$, we have
\begin{align*}
|q_0(t,x,y)-q_0(t,x',y)| \p & |x-x'|\left(\rho^{y,\B_0}_{-\gamma-({1/\A(y)})}(t,x-y)+\rho^{y,\B_0}_{-\gamma-({1/\A(y)})}(t,x'-y)\right) +|x-x'|^{\B_0}\rho_{-\gamma}^{y,0}(t,x'-y).
\end{align*}

{\bf Case (e): $|x-x'|\le C_0t^{{1/\A(y)}}$ and $|x-y|\le t^{{1/\A(y)}}$.}

We still define $J_1$ and $J_2$ by those in \eqref{p3-1-2aa}. It is
easy to see that \eqref{p3-1-4} and \eqref{p3-1-5} still hold for
such case. Note that $|x'-y|\le |x-y|+|x-x'|\le
(1+C_0)t^{{1/\A(y)}}$. Using \eqref{p3-1-4} and
\eqref{p3-1-5}, and repeating the argument in {\bf Case (d)}, we
have
\begin{align*}
& J_1 \p |x-x'|\cdot |x-y|^{\B_0}\cdot
\left(\rho^{y,0}_{-\gamma-({1/\A(y)})}(t,x-y)+\rho^{y,0}_{-\gamma-({1/\A(y)})}(t,x'-y)\right),\\
& J_2  \p |x-x'|^{\B_0}\rho_{-\gamma}^{y,0}(t,x'-y).
\end{align*}
Then, we arrive at
\begin{align*}
&|q_0(t,x,y)-q_0(t,x',y)|\\
&\p|x-x'|\rho^{y,\B_0}_{-\gamma-({1/\A(y)})}(t,x-y)+|x-x'||x-y|^{\B_0}\rho^{y,0}_{-\gamma-({1/\A(y)})}(t,x'-y) +|x-x'|^{\B_0}\rho_{-\gamma}^{y,0}(t,x'-y)\\
&\p|x-x'|\rho^{y,\B_0}_{-\gamma-({1/\A(y)})}(t,x-y)+|x-x'|^{\B_0}\rho_{-\gamma}^{y,0}(t,x'-y),
\end{align*}
where the last step follows from the property that
\begin{equation*}
|x-x'||x-y|^{\B_0}=|x-x'|^{\B_0}|x-x'|^{1-\B_0}|x-y|^{\B}\p t^{{1/\A(y)}}|x-x'|^{\B_0}.
\end{equation*}

Combining all {\bf Cases (c)--(e)} together, we can find a constant $C_0>0$ small enough such that for every $|x-x'|\le C_0t^{{1/\A(y)}}$
\begin{align*}
 |q_0(t,x,y)-q_0(t,x',y)|
 &\p |x-x'|\big(
\rho^{y,\B_0}_{-\gamma-({1/\A(y)})}(t,x-y)
+\rho^{y,\B_0}_{-\gamma-({1/\A(y)})}(t,x'-y)\big)\\
&\,\,+|x-x'|^{\B_0}\big(1+\log(1+|x'-y|)\I_{\{|x'-y|>R_0\}}\big)\rho^{y,0}_{-\gamma}(t,x'-y).
\end{align*}

Furthermore, for $0<\gamma<\theta<{\B_0}/{\A_2}$, we have
\begin{align*}
|x-x'|^{\B_0}\rho^{y,0}_{-\gamma}(t,x'-y)&=
|x-x'|^{\B_0-\A(y)\theta}|x-x'|^{\A(y)\theta}
\rho^{y,0}_{-\gamma}(t,x'-y)\\
& \p |x-x'|^{\B_0-\A(y)\theta}\rho^{y,0}_{\theta-\gamma}(t,x'\!-y) \! \p |x-x'|^{\B_0-\A_2\theta}\rho^{y,0}_{\theta-\gamma}(t,x'\!-y),
\end{align*}
where the second step above follows from $|x-x'|\le C_0t^{{1/\A(y)}}$.
On the other hand, it holds that
\begin{align*}
 |x-x'|\rho^{y,\B_0}_{-\gamma-({1/\A(y)})}(t,x'-y)
&=
|x-x'|^{\A_1(\B_0-\A(x')\theta)/{\A(x')}}|x-x'|^{1-(\A_1(\B_0-\A(x')\theta)/{\A(x')})}
\rho^{y,\B_0}_{-\gamma-({1/\A(y)})}(t,x'-y)\\
& \p |x-x'|^{\A_1(\B_0-\A_2\theta)/{\A(x')}}
t^{{(1/\A(y))}-({\A_1(\B_0-\A(x')\theta)}/({\A(x')\A(y)}))}
\rho^{y,\B_0}_{-\gamma-({1/\A(y)})}(t,x'-y)\\
&\p|x-x'|^{\A_1(\B_0-\A_2\theta)/{\A(x)}}
\rho^{y,\B_0}_{-\gamma-(({\A_1}/{\A(y)})(({\B_0}/{\A(x')})-\theta))}(t,x'-y)\\
&\p |x-x'|^{\A_1(\B_0-\A_2\theta)/{\A(x)}}
\rho^{y,\B_0}_{-\gamma+\theta-({\B_0}/{\A(x')})}(t,x'-y),
\end{align*}
where in the first inequality we have used again that $|x-x'|\le
C_0t^{{1/\A(y)}}$, and the second inequality follows from
\eqref{p3-1-3}. Similarly, it holds that
$$|x-x'|\rho^{y,\B_0}_{-\gamma-({1/\A(y)})}(t,x-y) \p |x-x'|^{\A_1(\B_0-\A_2\theta)/{\A(x)}}
\rho^{y,\B_0}_{-\gamma+\theta-({\B_0}/{\A(x)})}(t,x-y).$$
Hence, we arrive at that for every $|x-x'|\le C_0t^{{1/\A(y)}}$ and $\varepsilon>0$
\begin{equation}\label{p3-1-6}
\begin{split}
& |q_0(t,x,y)-q_0(t,x',y)| \\
& \p |x-x'|^{\A_1(\B_0-\A_2\theta)/{\A(x)}}
\Big(
\rho^{y,\B_0}_{-\gamma+\theta-({\B_0}/{\A(x)})}(t,x-y)
+\rho^{y,\B_0}_{-\gamma+\theta-({\B_0}/{\A(x')})}(t,x'-y)\\
&\qquad   +\rho^{y,0}_{\theta-\gamma}(t,x'-y)+|x'-y|^{\varepsilon}\I_{\{|x'-y|>R_0\}}\rho^{y,0}_{\theta-\gamma}(t,x'-y)\Big).
\end{split}
\end{equation}

Therefore, according to all the five different cases above, \eqref{p3-1-1a-1} holds for $q_0(t,x,y)$.

\medskip

(2) Now we study the estimate for $|q_0(t,x,y)-q_0(t,x',y)|$ under the assumption that
$\K(x,z)$ depends on $z$.
Recall that, by  \eqref{l2-1-1}, we have
\begin{equation}\label{p3-1-8}
|q_0(t,x,y)|\p \rho^{y,\B_0}_{-\G_2 }(t,x-y),\quad (t,x,y)\in (0,1]\times \R^d \times \R^d,
\end{equation}
where $\G_2 =:\gamma+({\A_2}/{\A_1})-1$.

Using \eqref{p3-1-8} and following the arguments in {\bf Case (a)} and {\bf Case (b)} of  part (1) above, we derive that for every $(t,x,y)\in (0,1]\times \R^d \times \R^d$ such that
$|x-x'|\ge C_0t^{{1/\A(y)}}$ (which includes both $|x-x'|>R_1$
and $C_0t^{{1/\A(y)}}\le |x-x'|\le R_1$, and where as in (1) the constant $C_0$ is chosen to be small enough) and $\G_2 <\theta<{\B_0}/{\A_2}$,
\begin{align*}
|q_0(t,x,y)-q_0(t,x',y)|& \p
\big(|x-x'|^{\A_1(\B_0-\A_2\theta)/{\A(x)}}\wedge 1\big)\\
&\quad\times\big(\rho^{y,\B_0}_{-\G_2 +\theta-({\B_0}/{\A(x)})}(t,x-y)
+\rho^{y,\B_0}_{-\G_2 +\theta-({\B_0}/{\A(x')})}(t,x'-y)\big).
\end{align*}

Note that
in the proofs of {\bf Case (d)} and {\bf Case (e)} above, we do not need the assumption that
$\K(x,z)$ is independent of $z$, so the conclusions there are still true, which in particular means that
\eqref{p3-1-2a-1} holds for {\bf Case (d)} and {\bf Case (e)}.
Now we turn to the case that $|x-x'|\le C_0t^{{1/\A(y)}}$ and $|x-y|>R_1$, i.e., the {\bf Case (c)} above.
We still define
$J_{11}$, $J_{12}$, $J_{21}$ and $J_{22}$ as those in \eqref{p3-1-2aa}, \eqref{p3-1-3a} and \eqref{p3-1-4a}, respectively.

For $J_{11}$, we apply \eqref{p1-1-4} instead of \eqref{e:conpr2} in part (1), and derive that
\begin{align*}
J_{11}\p \,&t^{-{1/\A(y)}}|x-x'|\int_{\{|z|\le {cR_1}/{4}\}}
[(t^{-{2}/{\A(y)}} |z|^2)\wedge 1]\\
&\times\Big(
\rho_1^{y,0}(t,x-y)+\rho_1^{y,0}(t,x'-y)+\rho_1^{y,0}(t,x-y\pm z)+\rho_1^{y,0}(t,x'-y\pm z)\Big)\cdot\bigg(\frac{1}{|z|^{d+\A(x)}}
+\frac{1}{|z|^{d+\A(y)}}\bigg)\,dz.
\end{align*}
Note that $|x-y|>R_1$. By \eqref{r1-1-1a}, we arrive at
\begin{align*}
&\int_{\{|z|\le {cR_1}/{4}\}}[(t^{-{2}/{\A(y)}} |z|^2)\wedge 1]
\big(
\rho_1^{y,0}(t,x-y)
+\rho_1^{y,0}(t,x-y\pm z)\big) \bigg(\frac{1}{|z|^{d+\A(x)}}
+\frac{1}{|z|^{d+\A(y)}}\bigg)\,dz\p \frac{t^{1-({\A_2}/{\A_1})}}{|x-y|^{d+\A_1}},
\end{align*}
which yields that
\begin{equation*}
J_{11}\p\, |x-x'|\big(\rho_{-\G_2-({1/\A(y)})}^{y,\B_0}(t,x-y)
+\rho_{-\G_2-({1/\A(y)})}^{y,\B_0}(t,x'-y)\big).
\end{equation*}
Similarly, using \eqref{p1-1-2} instead of \eqref{p1-1-1} in the estimate of $J_{21}$ and applying
\eqref{r1-1-2a}, we obtain
\begin{align*}
J_{21} &\p |x-x'|^{\B_0}\cdot \bigg(
\int_{\{|z|\le {cR_1}/{4}\}}[(t^{-{2}/{\A(y)}} |z|^2)\wedge 1]
\big[\rho^{y,0}_1(t,x'-y\pm z)+\rho^{y,0}_1(t,x'-y)\big]
\frac{1+|\log |z||}{|z|^{d+\overline{\A}(x;x')}}\,dz\bigg)\\
&\p |x-x'|^{\B_0}(1+|\log t|+|\log|x'-y||)\rho_{1-({\A_2}/{\A_1})}^{y,0}(t,x'-y).
\end{align*}
As explained before, the proofs for estimates of $J_{12}$ and
$J_{22}$ in (1) do not require the condition that $\K(x,z)$ is
independent of $z$, so those estimates still hold here.

By the estimates for $J_{11}$, $J_{12}$, $J_{21}$ and $J_{22}$, we know that for
every $|x-x'|\le C_0t^{{1/\A(y)}}$ and $|x-y|>R_1$,
\begin{align*}
|q_0(t,x,y)-q_0(t,x',y)|
 \p& |x-x'|\cdot\big[
\rho^{y,\B_0}_{-\gamma_2-({1/\A(y)})}(t,x-y)
+\rho^{y,\B_0}_{-\gamma_2-({1/\A(y)})}(t,x'-y)\big]\\
&+|x-x'|^{\B_0}\big(1+\log(1+|x'-y|)\big)\rho^{y,0}_{-\gamma_2}(t,x'-y).
\end{align*}
Combining the estimates above for all the cases, we can find a constant $C_0>0$ small enough such that for
every $|x-x'|\le C_0 t^{{1/\A(y)}}$ and
$\gamma,\varepsilon>0$,
\begin{align*}
|q_0(t,x,y)-q_0(t,x',y)|\p& |x-x'|\cdot\big[
\rho^{y,\B_0}_{-\G_2 -({1/\A(y)})}(t,x-y)
+\rho^{y,\B_0}_{-\G_2 -({1/\A(y)})}(t,x'-y)\big]\\
&+|x-x'|^{\B_0}|x-y|^{\varepsilon}\I_{\{|x-y|>R_0\}}\rho^{y,0}_{-\gamma_2}(t,x-y)\\
&+|x-x'|^{\B_0}|x'-y|^{\varepsilon}\I_{\{|x'-y|>R_0\}}\rho^{y,0}_{-\gamma_2}(t,x'-y).
\end{align*}
Then, by the same argument for \eqref{p3-1-6} and all the conclusions above, we can immediately show that \eqref{p3-1-2a-1} holds.
\end{proof}

\medskip

\section{Estimates for $q(t,x,y)$ }\label{section3}

\subsection{Convolution inequalities}\label{section3.1}
In this subsection, we will establish a convolution inequality
involving $\rho_\G^{y,\B}$, which will be frequently used in the
remainder of our paper. Though part idea of the arguments below is
inspired by these of \cite{CZ,Kol}, some essential and non-trivial
modifications are required in our variable order setting.

\begin{lemma}\label{l3-2} \begin{itemize}
\item[(1)]
For every $\theta_1,\theta_2\in \R$ and  $\beta_1,\beta_2 \in (0,1)$, it holds for all $0<s<t\le 1$ and
$x,y \in \R^d$ that
\begin{equation}\label{l3-2-0}
\begin{split}
&\int_{\R^d} \rho_{\T_1}^{z,\beta_1}(t-s,x-z)\rho_{\T_2}^{y,\beta_2}(s,z-y)\, dz\\
 &\p \Big[(t-s)^{\T_1+[((\beta_1+\beta_2)/{\A(x)})\wedge 1]-1}s^{\T_2}
 \big(1+|\log (t-s)|\I_{\{\B_1+\B_2=\A(x)\}}\big)\\
 &\qquad+(t-s)^{\T_1}
s^{\T_2 +[((\beta_1+\beta_2)/{\A(y)})\wedge 1]-1}\big(1+|\log s|\I_{\{\B_1+\B_2=\A(y)\}}\big)\Big]\rho_{0}^{y,0}(t,x-y)\\
&\qquad +(t-s)^{\T_1+[({\beta_1}/{\A(x)})\wedge 1]-1}s^{\T_2}
\big(1+|\log (t-s)|\I_{\{\B_1=\A(x)\}}\big)\rho_{0}^{y,\beta_2}(t,x-y)\\
&\qquad+(t-s)^{\T_1}s^{\T_2+({\beta_2}/{\A(y)})\wedge 1 -1}
\big(1+|\log s|\I_{\{\B_2=\A(y)\}}\big)\rho_0^{y,\beta_1}(t,x-y).
\end{split}
\end{equation}

\item[(2)] For every $\theta_1,\theta_2\in \R$ and  $\beta_1,\beta_2 \in (0,1)$, it holds
for every $x,y \in \R^d$ and $0<s<t\le 1$ that
\begin{equation}\label{l3-2-1}
\begin{split}
&\int_{\R^d} \rho_{\T_1}^{z,\beta_1}(t-s,x-z)\rho_{\T_2}^{y,\beta_2}(s,z-y)\, dz\\
 &\p \Big[(t-s)^{\T_1+[((\beta_1+\beta_2)/{\A_2})\wedge 1]-1}s^{\T_2}\big(1+|\log (t-s)|\I_{\{\B_1+\B_2=\A_2\}}\big)\\
 &\quad+(t-s)^{\T_1}
s^{\T_2 +[((\beta_1+\beta_2)/{\A_2})\wedge 1]-1}\big(1+|\log s|\I_{\{\B_1+\B_2=\A_2\}}\big)\Big]\rho_{0}^{y,0}(t,x-y)\\
&\quad +(t-s)^{\T_1+[({\beta_1}/{\A_2})\wedge 1]-1}s^{\T_2}\big(1+|\log (t-s)|\I_{\{\B_1=\A_2\}}\big)\rho_{0}^{y,\beta_2}(t,x-y)\\
&\quad+(t-s)^{\T_1}s^{\T_2+[({\beta_2}/{\A_2})\wedge 1] -1}\big(1+|\log s|\I_{\{\B_2=\A_2\}}\big)\rho_0^{y,\beta_1}(t,x-y).
\end{split}
\end{equation}\end{itemize}
\end{lemma}
\begin{proof} (1) The proof is split into three cases.

{\bf Case (a): $|x-y|>2.$}

It holds that
\begin{equation}\label{l3-2-1a}
\begin{split}
 \int_{\R^d} \rho_{\T_1}^{z,\beta_1}(t-s,x-z)\rho_{\T_2}^{y,\beta_2}(s,z-y)\,dz
&\le  \int_{\{|z-y|\ge {|x-y|}/{2}\}} \rho_{\T_1}^{z,\beta_1}(t-s,x-z)\rho_{\T_2}^{y,\beta_2}(s,z-y)\,dz\\
 &\quad+\int_{\{|z-y|\le {|x-y|}/{2}\}} \rho_{\T_1}^{z,\beta_1}(t-s,x-z)\rho_{\T_2}^{y,\beta_2}(s,z-y)\, dz\\
&=:J_1+J_2.
\end{split}
\end{equation}

If $|z-y|\ge {|x-y|}/{2}\ge 1\ge s^{{1/\A(y)}}$, then
\begin{equation}\label{l3-2-2}
\begin{split}
\rho_0^{y,0}(s,z-y)&\p \frac{1}{|z-y|^{d+\A_1}}\p
\frac{1}{|x-y|^{d+\A_1}}\p \rho_0^{y,0}(t,x-y).
\end{split}
\end{equation}
Observe that
(see e.g.\ the proof of \cite[Lemma 2.2]{CZ}) for any $x,y, z\in
\R^d$,
\begin{equation}\label{l3-2-3}
\begin{split}
\big(|x-z|^{\B_1}\wedge 1\big)
\big(|z-y|^{\B_2}\wedge 1\big)\le &\big(|x-z|^{\B_1+\B_2}\wedge 1\big) +\big(|x-z|^{\B_1}\wedge 1\big)
\big(|x-y|^{\B_2}\wedge 1\big).
\end{split}
\end{equation}
Combining this with \eqref{l3-2-2} and the first inequality in \eqref{l3-1-1},
we get that for $\beta_1,\beta_2\in (0,1)$,
 \begin{equation}\label{l3-2-4}
\begin{split}
J_1&\p (t-s)^{\T_1}s^{\T_2}\rho_0^{y,0}(t,x-y)\int_{\R^d}\rho^{z,\B_1+\B_2}_0(t-s,x-z)\,dz\\
&\quad +(t-s)^{\T_1}s^{\T_2}\rho_0^{y,\B_2}(t,x-y)\int_{\R^d}\rho^{z,\B_1}_0(t-s,x-z)\,dz\\
&\p (t-s)^{\T_1+[(({\B_1+\B_2})/{\A(x)
})\wedge 1]-1}s^{\T_2}\big(1\!+\!|\log (t-s)|\I_{\{\B_1+\B_2=\A(x)\}}\big)\rho_0^{y,0}(t,x-y)\\
&\quad+
(t-s)^{[\T_1+({\B_1}/{\A(x)})\wedge 1]-1}s^{\T_2}\big(1+|\log (t-s)|\I_{\{\B_1=\A(x)\}}\big)\rho_0^{y,\B_2}(t,x-y).
\end{split}
\end{equation}

If $|z-y|\le {|x-y|}/{2}$, then $|x-z|\ge
{|x-y|}/{2}\ge 1 \ge (t-s)^{{1}/{\A(z)}}$, and so
\begin{align*}
& \rho_0^{z,0}(t-s,x-z)\p \frac{1}{|x-z|^{d+\A_1}}\p \frac{1}{|x-y|^{d+\A_1}}\p \rho_0^{y,0}(t,x-y).
\end{align*}
Combining this with
\begin{equation}\label{l3-2-5}
\begin{split}
\big(|x-z|^{\B_1}\wedge 1\big)
\big(|z-y|^{\B_2}\wedge 1\big)\le &\big(|z-y|^{\B_1+\B_2}\wedge 1\big) +\big(|z-y|^{\B_2}\wedge 1\big)
\big(|x-y|^{\B_1}\wedge 1\big),
\end{split}
\end{equation}
we arrive at
\begin{equation}\label{l3-2-5a}
\begin{split}
J_2&\p (t-s)^{\T_1}s^{\T_2}\rho_0^{y,0}(t,x-y)\int_{\R^d}\rho^{y,\B_1+\B_2}_0(s,z-y)\, dz\\
&\quad +(t-s)^{\T_1}s^{\T_2}\rho_0^{y,\B_1}(t,x-y)\int_{\R^d}\rho^{y,\B_2}_0(s,z-y)\, dz\\
&\p (t-s)^{\T_1}s^{\T_2+[(({\B_1+\B_2})/{\A(y)})\wedge 1]-1}
\big(1+|\log s|\I_{\{\B_1+\B_2=\A(y)\}}\big)\rho_0^{y,0}(t,x-y)\\
&\quad+
(t-s)^{\T_1}s^{\T_2+[({\B_2}/{\A(y)})\wedge 1]-1}\big(1+|\log s|\I_{\{\B_2=\A(y)\}}\big)\rho_0^{y,\B_1}(t,x-y),
\end{split}
\end{equation}
where the second inequality follows from the first inequality in \eqref{l3-1-1-00}.

According to the above estimates for $J_1$ and $J_2$, \eqref{l3-2-1} holds in this case.

{\bf Case (b): $2t^{{1}/{\overline{\A}(y;t^{{1/\A_2}})}}\le |x-y|\le 2$.}

Let $J_1$ and $J_2$ be defined by \eqref{l3-2-1a}. If $|z-y|\ge
{|x-y|}/{2}$ and $0<s<t<1$, then $|z-y|\ge {|x-y|}/{2}$,
and so
\begin{equation}\label{l3-2-6}
\begin{split}
\rho_0^{y,0}(s,z-y)&\p
\begin{cases}
|z-y|^{-d-\A(y)},& |z-y|\le 1,\\
 |z-y|^{-d-\A_1},& |z-y|>1
\end{cases}
\\
&\p \begin{cases}
|x-y|^{-d-\A(y)},&|z-y|\le 1,\\
|x-y|^{-d-\A_1},& |z-y|>1
\end{cases}
\\
&\p |x-y|^{-d-\A(y)}\p \rho^{y,0}_0(t,x-y),
\end{split}
\end{equation}
where in
the third inequality we have used the fact that
$$|x-y|^{-d-\A_1}\p |x-y|^{-d-\A(y)}\quad \textrm{ for all }|x-y|\le 2 ,$$ and the last
inequality follows from
$$|x-y|^{-d-\A(y)} \p  \rho^{y,0}_0(t,x-y)\quad\textrm{ for all } |x-y|\ge 2t^{{1}/{\overline{\A}(y;t^{{1/\A_2}})}}\ge 2t^{{1/\A(y)}}.$$
Using \eqref{l3-2-6} and following the same argument of \eqref{l3-2-4}, we know that \eqref{l3-2-4} is satisfied too.

If $|z-y|<{|x-y|}/{2}$, then $|x-z|\ge {|x-y|}/{2}$,
and so we obtain
\begin{equation}\label{l3-2-7}
\begin{split}
\rho_0^{z,0}(t-s,x-z)&\p
\begin{cases}
|x-z|^{-d-\A(z)},& |x-z|\le 1,\\
|x-z|^{-d-\A_1}, & |x-z|>1
\end{cases}
\\
&\p |x-y|^{-d-\A(z)} \p |x-y|^{\A(y)-\A(z)}\cdot|x-y|^{-d-\A(y)}\\
&\p |x-y|^{-|\A(y)-\A(z)|}\cdot|x-y|^{-d-\A(y)}\\
&\p |x-y|^{-C|x-y|^{\beta_0}}\cdot|x-y|^{-d-\A(y)}\\
&\p \exp\big(C|\log|x-y||\cdot|x-y|^{\B_0}\big)\cdot|x-y|^{-d-\A(y)}\\
&\p |x-y|^{-d-\A(y)}\p \rho_0^{y,0}(t,x-y),
\end{split}
\end{equation}
where in the second inequality we used the fact that if $|x-z|\ge {|x-y|}/{2}$ and $|x-y|\le 2$, then
$|x-z|^{-d-\A_1}\p |x-y|^{-d-\A_1}\p |x-y|^{-d-\A(z)},$
the fifth inequality follows from the fact that
$$|\A(y)-\A(z)|\p |z-y|^{\beta_0}\le C |x-y|^{\B_0}\quad\textrm{ for all }|z-y|\le {|x-y|}/{2},$$
in the seventh inequality we used $$\sup_{z\in \R^d: |z|\le 2}\exp\big(C|\log
|z||\cdot|z|^{\beta_0}\big)<\infty,$$ and the last inequality is due to
the same argument of the last one in \eqref{l3-2-6}.

Using the above inequality for $\rho_0^{z,0}(t-s,x-z)$ and following the same procedure of \eqref{l3-2-5a},
we know that \eqref{l3-2-5a} still holds.

Therefore, according to the estimates for $J_1$ and $J_2$, \eqref{l3-2-1} also holds in this case.

{\bf Case (c): $|x-y|\le
2t^{{1}/{\overline{\A}(y;t^{{1/\A_2}})}}$.}

Note that for any $t \in (0,1]$ and $y \in \R^d$,
\begin{equation}\label{l3-2-8}
\begin{split}
1\le&
\frac{t^{{1}/{\overline{\A}(y;t^{{1/\A_2}})}}}{t^{{1/\A(y)}}}\le
t^{-\frac{\overline{\A}(y;t^{{1/\A_2}})-\A(y)}{\A(y)\overline{\A}(y;t^{{1/\A_2}})}}
\le
t^{-\frac{Ct^{{\B_0}/{\A_2}}}{\A_1^2}}\le \exp\big(C|\log
t|\cdot t^{{\B_0}/{\A_2}} \big)\le C_1 ,
\end{split}
\end{equation}
hence $|x-y|\le 2t^{{1}/{\overline{\A}(y;t^{{1/\A_2}})}}$
implies that $|x-y|\le C_2 t^{{1/\A(y)}}$ for some constant
$C_2>0$.

On the other hand,
\begin{align*}
 \int_{\R^d}\rho_{\T_1}^{z,\B_1}(t-s,x-z)\rho_{\T_2}^{y,\B_2}(s,z-y)\,dz
&\le \int_{\{|z-y|> t^{{1/\A(y)}}\}}\rho_{\T_1}^{z,\B_1}(t-s,x-z)\rho_{\T_2}^{y,\B_2}(s,z-y)\,dz\\
&\quad+\int_{\{|z-y|\le t^{{1/\A(y)}}\}}\rho_{\T_1}^{z,\B_1}(t-s,x-z)\rho_{\T_2}^{y,\B_2}(s,z-y)\, dz\\
&=:J_1+J_2.
\end{align*}

If $|z-y|\ge t^{{1/\A(y)}}>s^{{1/\A(y)}}$, then
\begin{align*}
\rho_0^{y,0}(s,z-y)\p |z-y|^{-d-\A(y)}\p t^{-({d+\A(y)})/{\A(y)}}\p \rho_0^{y,0}(t,x-y),
\end{align*}
where the second inequality follows from $|z-y|\ge t^{{1/\A(y)}}$, and the last inequality
is due to $t^{-({d+\A(y)})/{\A(y)}}\p \rho_0^{y,0}(t,x-y)$ for all $|x-y|\le C_2t^{{1/\A(y)}}$. Applying the above estimate for $\rho_0^{y,0}(s,z-y)$ and following the same argument of \eqref{l3-2-4}, we can find that \eqref{l3-2-4} still holds
for $J_1$ here.

Now we turn to estimate for $J_2$, which is divided into the following two subcases:

{\bf Subcase (c1): $s>{t}/{2}$.}

If $|z-y|\le t^{{1/\A(y)}}$, then
\begin{align*}
\rho_0^{y,0}(s,z-y)&\p s^{-({d/\A(y)})-1}\p t^{-({d/\A(y)})-1}\p \rho_0^{y,0}(t,x-y),
\end{align*}
where the second inequality follows
from $s>{t}/{2}$, and in the last inequality we have used the fact that
$t^{-({d/\A(y)})-1}\p \rho_0^{y,0}(t,x-y)$ for all $|x-y|\le C_2t^{{1/\A(y)}}$.

Having the above estimate for $\rho_0^{y,0}(s,z-y)$ at hand, we also can follow the same procedure of \eqref{l3-2-4} to get that \begin{align*}
J_2&\p (t-s)^{\T_1+(({\B_1+\B_2})/{\A(x)})-1}s^{\T_2}\rho_0^{y,0}(t,x-y)+
(t-s)^{\T_1+({\B_1}/{\A(x)})-1}s^{\T_2}\rho_0^{y,\B_2}(t,x-y).
\end{align*}

{\bf Subcase (c2): $s\le {t}/{2}$.}

If $|z-y|\le t^{{1/\A(y)}}$,
then we have
\begin{align*}
\rho_0^{z,0}(t-s,x-z)&\p (t-s)^{-({d}/{\A(z)})-1}\p t^{-({d}/{\A(z)})-1}\p t^{-|({d}/{\A(z)})-({d/\A(y)})|}\cdot t^{-({d/\A(y)})-1}\\
&\p t^{-Ct^{{\B_0}/{\A_2}}}\cdot t^{-({d/\A(y)})-1}\p t^{-({d/\A(y)})-1}\p \rho_0^{y,0}(t,x-y),
\end{align*}
where the second inequality is due to $(t-s)\ge {t}/{2}$, in the
fourth inequality we used the fact that
$|\A(y)-\A(z)|\p |z-y|^{\beta_0}\p t^{{\B_0}/{\A(y)}}\p
t^{{\B_0}/{\A_2}},$ the fifth inequality follows from
$$\sup_{t \in (0,1]}t^{-Ct^{{\B_0}/{\A_2}}}\p \sup_{t \in (0,1]}
\exp\big(C|\log t|t^{{\B_0}/{\A_2}}\big)<\infty,$$ and the last
inequality is again due to the fact that $t^{-({d/\A(y)})-1}\p
\rho_0^{y,0}(t,x-y)$ for all $|x-y|\le C_2t^{{1/\A(y)}}$.

Using the above estimate for $\rho_0^{z,0}(t-s,x-z)$ and following the same procedure of \eqref{l3-2-5a},  we
can also get \eqref{l3-2-5a}.

Combining both estimates for $J_1$ and $J_2$, we know \eqref{l3-2-1} holds for {\bf Case (c)}.
Therefore, the proof of \eqref{l3-2-0} is complete.

(2) We can apply the second inequality in \eqref{l3-1-1} in the proofs of
\eqref{l3-2-4} and \eqref{l3-2-5a}, and follow the same arguments as above to prove \eqref{l3-2-1}.
\end{proof}

According to Lemma \ref{l3-2}, we can immediately derive the following statement.
\begin{corollary}\label{l3-3} Let $\Bb(\gamma,\beta)$ denote the Beta function with respect to $\gamma$ and $\beta$. Then, the following two statements hold.
\begin{itemize}
\item[(1)] For any $\theta_1,\theta_2\in \R$ and $\beta_1,\beta_2 \in (0,1)$, there is a constant $c>0$ such that for all $t\in (0,1]$ and $x,y \in \R^d$  with $\B_1+\B_2<\A_*(x,y):=\A(x)\wedge\A(y)$,
$\T_1+({\B_1}/{\A(x)})>0$ and $\T_2+({\B_2}/{\A(y)})>0$,
\begin{equation}\label{l3-3-0}
\begin{split}
\int_0^t&\int_{\R^d} \rho_{\T_1}^{z,\beta_1}(t-s,x-z)\rho_{\T_2}^{y,\beta_2}(s,z-y)\,dz\,ds\\
\le &c\Bb\big(({\B_1}/{\A_2})+\T_1,({\B_2}/{\A_2})+\T_2\big)\\
&\times \Big(
\rho_{\T_1+\T_2+(({\B_1+\B_2})/{\A^*(x,y)})}^{y,0}+
\rho_{\T_1+\T_2+({\B_1}/{\A(x)})}^{y,\B_2}
+\rho_{\T_1+\T_2+({\B_2}/{\A(y)})}^{y,\B_1}\Big)(t,x-y),
\end{split}
\end{equation}
where $\A^*(x,y):= \A(x)\vee\A(y)$.

\item[(2)] For any $\theta_1,\theta_2\in \R$ and $\beta_1,\beta_2 \in (0,1)$ such that $\B_1+\B_2<\A_2$,  $\T_1+({\B_1}/{\A_2})>0$ and $\T_2+({\B_2}/{\A_2})>0$,
it holds  for any $t\in (0,1]$ and $x,y \in \R^d$ that
\begin{equation}\label{l3-3-1}
\begin{split}
\int_0^t&\int_{\R^d} \rho_{\T_1}^{z,\beta_1}(t-s,x-z)\rho_{\T_2}^{y,\beta_2}(s,z-y)\,dz\,ds\\
\p&\Bb\big(({\B_1}/{\A_2})+\T_1,({\B_2}/{\A_2})+\T_2\big)
\\
&\times \big(
\rho_{\T_1+\T_2+(({\B_1+\B_2})/{\A_2})}^{y,0}+\rho_{\T_1+\T_2+({\B_1}/{\A_2})}^{y,\B_2}
+\rho_{\T_1+\T_2+({\B_2}/{\A_2})}^{y,\B_1}\big)(t,x-y).
\end{split}
\end{equation}
\end{itemize}
\end{corollary}
\begin{proof}
Note that
\begin{equation*}
\int_0^t
(t-s)^{\gamma-1}s^{\B-1}\,ds=t^{\gamma+\B-1}\Bb(\gamma,\B),\quad
\gamma>0,\ \B>0.
\end{equation*}
This, along with  \eqref{l3-2-0}, \eqref{l3-2-1} and  the decreasing property of $\Bb(\gamma,\B)$ with respect to $\gamma$ and $\beta$, yields the desired assertions  \eqref{l3-3-0} and \eqref{l3-3-1}.
\end{proof}

At the end of this subsection, we make some remarks.

\begin{remark}\label{r3-3}
(1)  According to the proofs of \cite[Lemma 2.1 (ii) and (iii)]{CZ}, we know that for $0\le s\le t\le 1$, $x,y\in \R^d$ and $\B_1,\B_2
\in (0,1)$, $\T_1,\T_2\in \R$, it holds that
\begin{equation}\label{l3-2-1-}
\begin{split}
&\int_{\R^d} \rho_{\T_1}^{y,\beta_1}(t-s,x-z)\rho_{\T_2}^{y,\beta_2}(s,z-y)\, dz\\
 &\p \Big[(t-s)^{\T_1+[((\beta_1+\beta_2)/{\A(y)})\wedge 1]-1}s^{\T_2}
 \big(1+|\log (t-s)|\I_{\{\B_1+\B_2=\A(y)\}}\big)\\
 &\qquad+(t-s)^{\T_1}
s^{\T_2 +[((\beta_1+\beta_2)/{\A(y)})\wedge 1]-1} \big(1+|\log s|\I_{\{\B_1+\B_2=\A(y)\}}\big)\Big]\rho_{0}^{y,0}(t,x-y)\\
&\qquad +(t-s)^{\T_1+[({\beta_1}/{\A(y)})\wedge 1]-1}s^{\T_2} \big(1+|\log (t-s)|\I_{\{\B_1=\A(y)\}}\big)\rho_{0}^{y,\beta_2}(t,x-y)\\
&\qquad
+(t-s)^{\T_1}s^{\T_2+[({\beta_2}/{\A(y)})\wedge 1] -1} \big(1+|\log s|\I_{\{\B_2=\A(y)\}}\big)\rho_0^{y,\beta_1}(t,x-y).
\end{split}
\end{equation} If, in addition, $\B_1+\B_2<\A(x)$, $({\B_1}/{\A(y)})+\T_1>0$ and $({\B_2}/{\A(y)})+\T_2>0$, then it also holds
\begin{equation}\label{r3-2-1}
\begin{split}
&\int_0^t\int_{\R^d}\rho^{y,\B_1}_{\T_1}(t-s,x-z)
\rho^{y,\B_2}_{\T_2}(s,z-y)\,dz\,ds\\
&\p
\mathscr{B}\big(({\B_1}/{\A(y)})+\T_1,({\B_2}/{\A(y)})+\T_2\big)\\
&\quad \times \big(\rho^{y,0}_{\T_1+\T_2+(({\B_1+\B_2})/{\A(y)})} +\rho^{y,\B_1}_{\T_1+\T_2+({\B_2}/{\A(y)})}(t,x-y)
+\rho^{y,\B_2}_{\T_1+\T_2+({\B_1}/{\A(y)})}\big)(t,x-y).
\end{split}
\end{equation}
Note that in \eqref{l3-2-1-} and \eqref{r3-2-1} the index $\A(y)$ is independent of the
integrand variable $z$. However, in the present setting we also need \eqref{l3-2-0} and \eqref{l3-2-1} as well as the
convolution inequalities \eqref{l3-3-0} and \eqref{l3-3-1}, where the index $\A(z)$ will
depend on the integrand variable $z$.

(2) For every fixed $\varepsilon>0$ small enough, $\G\in \R$,
$\theta \in \R_+$ and $x, y \in \R^d$, define
\begin{equation}\label{p3-1-7}
\tilde \rho_{\G,\e}^{y,\theta }(t,x)=t^{\G}(|x|^{\theta}\wedge
1)
\begin{cases}
\frac{1}{(t^{{1/\A(y)}}+|x|)^{d+\A(y)}},&  |x|\le 1,\\
\frac{1}{|x|^{d+\A_1-\varepsilon}},& |x|>1.
\end{cases}
\end{equation}
By carefully tracking the proofs of Lemma \ref{l3-1}, Lemma \ref{l3-2}, Corollary
\ref{l3-3} and  \cite[Lemma 2.1 (ii) and (iii)]{CZ}, we know that the inequalities \eqref{l3-1-1},
\eqref{l3-1-1-00},  \eqref{l3-2-0}, \eqref{l3-2-1}, \eqref{l3-3-0}, \eqref{l3-3-1},  \eqref{l3-2-1-} and \eqref{r3-2-1}
are valid with
$\rho$ replaced by $\tilde \rho$. (In particular, \eqref{l3-2-6} holds true.) For simplicity, in the remainder of this paper we often omit
the parameter $\e$ in $\tilde \rho$.
\end{remark}

\subsection{Existence, upper bounds and continuity of $q(t,x,y)$.}
We will prove the existence and some estimates for the solution $q(t,x,y)$ to the equation
\eqref{eq2-2}. For this, we first define $q_n(t,x,y)$ inductively by
\begin{equation}\label{eq3-1}
q_n(t,x,y):=\int_0^t\int_{\R^d}q_0(t-s,x,z)q_{n-1}(s,z,y)\,dz\,ds,\ n\ge1,\ t\in(0,1].
\end{equation}
Then we can construct $q(t,x,y)$ as follows.
\begin{proposition}\label{p3-1} Let $\B_0^*\in(0,\B_0]\cap(0,{\alpha_2})$ and $\B_0^{**} \in (0,\B_0]\cap(0,{\alpha_2}/{2})$. Then, the following two statements hold.
\begin{itemize}
\item[(1)] If $\kappa(x,z)$ is independent of $z$, then $q(t,x,y):=\sum_{n=0}^{\infty}q_n(t,x,y)$ is
absolutely convergent on $(0,1]\times \R^d\times \R^d$, it solves equation \eqref{eq2-2} and satisfies that for any $0<\gamma<\theta<{\beta^*_0}/{\A_2}$ and
$\varepsilon>0$,
there exists a positive constant $c_1:=c_1(\A,\K,\G,\theta,\varepsilon)$ such that
for every $(t,x,x',y)\in (0,1]\times \R^d\times \R^d \times \R^d$,
\begin{equation}\label{p3-1-1}
|q(t,x,y)|\le c_1\left(\rho^{y,0}_{{(\B_0^*/\A_2)}-\gamma}+\rho^{y,\B_0^*}_{-\gamma}\right)(t,x-y)
\end{equation} and if, moreover, $\B_0^*<\A(x)$, then
\begin{equation}\label{p3-1-1a}
\begin{split}
 |q(t,x,y)-q(t,x',y)|\le & c_1\left(|x-x'|^{{\A_1(\B^*_0-\A_2\theta)}/{\A(x)}}\wedge 1\right)\\
   &\times\Big[\Big(\tilde \rho^{y,0}_{\theta-\gamma+(\B_0^{*}/\A_2)-(\B_0^*/\A_1)}
(t,x-y)+\tilde \rho^{y,0}_{\theta-\gamma+
(\B_0^{*}/\A_2)-(\B_0^*/\A_1)}(t,x'-y)\Big].
\end{split}
\end{equation}
Suppose, in addition, that there are  $y_0\in \R^d$ and $r\in(0,\infty]$ such that  $\alpha(z)=\alpha(y_0)$ for any $z \in B(y_0,r_0)$. Then there is a constant
$c_3:=c_3(\A,\K,r_0)>0$ such that for every $t\in (0,1]$ and $x \in \R^d$,
\begin{equation}\label{p3-1-2b}
|q(t,x,y_0)|\le c_3\left(\rho^{y_0,0}_{{\B_0^*/\A_2}}+\rho^{y_0,\B_0^*}_{0}\right)(t,x-y_0).
\end{equation}

\item[(2)] If $\kappa(x,z)$ depends on $z$ and $({\A_2}/{\A_1})-1<{\B^{**}_0}/{\A_2}$ holds true, then $q(t,x,y)=:\sum_{n=0}^{\infty}q_n(t,x,y)$ is
still absolutely convergent, solves the equation \eqref{eq2-2}, and satisfies that for every $\gamma,\theta>0$ such that $({\A_2}/{\A_1})-1+\gamma =:\G_2<\theta<{\B_0^{**}/\A_2}$ and
$\varepsilon>0$, there exists a positive constant $c_2:=c_2(\A,\K, \G,\theta,\varepsilon)$ such that for every $(t,x,x',y)\in (0,1]\times \R^d\times \R^d \times \R^d$,
\begin{equation}\label{p3-1-2}
|q(t,x,y)|\le c_2\left(\rho^{y,0}_{{(\B_0^{**}/\A_2)}-\G_2 }+\rho^{y,\B_0^{**}}_{-\G_2 }\right)(t,x-y)
\end{equation} and if, moreover, $\B_0^*<\A(x)$, then
\begin{equation}\label{p3-1-2a}
\begin{split}
 |q(t,x,y)-q(t,x',y)|
&\le c_2\left(|x-x'|^{{\A_1(\B_0^*-\A_2\theta)}/{\A(x)}}\wedge 1\right)\\
&\quad\times
\Big[\tilde \rho^{y,0}_{\theta-\gamma_2+(\B_0^{*}/\A_2)-(\B_0^*/\A_1)}(t,x-y)
+\tilde \rho^{y,0}_{\theta-\gamma_2
+(\B_0^{*}/\A_2)-(\B_0^*/\A_1)}(t,x'-y)\Big].
\end{split}
\end{equation}
If additionally there are some $y_0\in \R^d$ and $r_0\in(0,\infty]$ such that $\alpha(z)=\alpha(y_0)$ for all $z \in B(y_0,r_0)$, then there is a constant $c_4:=c_4(\A,\K, r_0)>0$ such that for every $t\in (0,1]$ and  $x \in \R^d$,
\begin{equation}\label{p3-1-2c}
|q(t,x,y_0)|\le c_4\left(\rho^{y_0,0}_{1-({\A_2}/{\A_1})+{(\B_0^{**}/\A_2)}}+\rho^{y_0,\B_0^{**}}_{1-({\A_2}/{\A_1})}\right)(t,x-y_0).
\end{equation}
\end{itemize}
\end{proposition}
\begin{proof}
Without loss of generality, throughout the proof we will assume that $\B_0< \A_2$ and $\B_0^*=\B_0$; otherwise, we will replace $\B_0$ by $\B_0^*.$

(i) According to \eqref{l2-1-2}, if $\kappa(x,z)$ is independent of $z$, then
$$|q_0(t,x,y)|\p \rho_{-\gamma}^{y,\beta_0}(t,x-y),\quad t\in(0,1] \textrm{ and } x,y \in \R^d.$$
Therefore, by Corollary \ref{l3-3},
\begin{align*}|q_1(t,x,y)|\p&\int_0^t\int \rho_{-\gamma}^{z,\B_0}(t-s,x-z) \rho^{y,\B_0}_{-\gamma}(s,z-y)\,dz\,ds\\
\p& \Bb\big(({\wh\B_0}/{\A_2})-\gamma, ({\B_0}/{\A_2})-\gamma\big)
\big(\rho_{2(({\wh
\B_0}/{\A_2})-\gamma)}^{y,0}+\rho_{-\gamma+(({\B_0}/{\A_2})-\gamma)}^{y,\B_0}\big)(t,x-y),\end{align*}
where $\wh \B_0:=\B_0 \wedge (\A_2/2)$.

Suppose now that
$$|q_n(t,x,y)|\le \omega_n\Big(\rho_{(n+1)(({\wh \B_0}/{\A_2})-\gamma)}^{y,0}+\rho_{-\gamma+n(({\B_0}/{\A_2})-\gamma)}^{y,\B_0} \Big)(t,x-y),\quad n\ge0,$$
where the constant $\omega_n$ is to be determined later.
Hence, according to Corollary \ref{l3-3} again,
\begin{align*}|q_{n+1}(t,x,y)|\le& C\int_0^t\int \rho_{-\gamma}^{z,\B_0}(t-s,x-z) \omega_n \Big(\rho_{(n+1)(({\wh\B_0}/{\A_2})-\gamma)}^{y,0}+\rho_{-\gamma+n(({\B_0}/{\A_2})-\gamma)}^{y,\B_0} \Big)(s,z-y)\,dz\,ds\\
\le&C \omega_n \Bb(({\beta_0}/{\A_2})-\gamma, (n+1)(({\wh \beta_0}/{\A_2})-\gamma)) \Big(\rho_{(n+2)(({\wh \B_0}/{\A_2})-\gamma)}^{y,0}+\rho_{-\gamma+(n+1)(({\B_0}/{\A_2})-\gamma)}^{y,\B_0} \Big)(t,x-y)\\
=&: \omega_{n+1}\Big(\rho_{(n+2)(({\wh \B_0}/{\A_2})-\gamma)}^{y,0}+\rho_{-\gamma+(n+1)(({\B_0}/{\A_2})-\gamma)}^{y,\B_0} \Big)(t,x-y), \end{align*} where $$ \omega_{n+1}= C \Bb\big(({\beta_0}/{\A_2})-\gamma, (n+1)(({\wh \beta_0}/{\A_2})-\gamma)\big)\omega_n ,$$
and $C$ is a constant independent of $t,n,x$ and $y$.
Note that $\omega_0\le C$ and $\Bb(\gamma, \B)=\frac{\Gamma(\gamma)\Gamma(\B)}{\Gamma(\gamma+\B)}$, where $\Gamma$ is the standard Gamma function. By iteration procedure we have
\begin{align*}\omega_n\le &C^{n+1}\Bb\big(({\beta_0}/{\A_2})-\gamma, n(({\wh \beta_0}/{\A_2})-\gamma)\big)\\
&\times\Bb\big(({\beta_0}/{\A_2})-\gamma, (n-1)(({\wh \beta_0}/{\A_2})-\gamma)\big)\cdots\Bb\big(({\beta_0}/{\A_2})-\gamma, ({\wh \beta_0}/{\A_2})-\gamma\big)\\
=&\frac{\big[C \Gamma(({\beta_0}/{\A_2})-\gamma )\big]^{n+1}}{\Gamma\big((n+1)(({\wh \beta_0}/{\A_2})-\gamma)\big)}.\end{align*}
Therefore,
\begin{equation}\label{p3-1-6a}\begin{split}
|q_n(t,x,y)|\le &\frac{\big[C \Gamma(({\beta_0}/{\A_2})-\gamma )\big]^{n+1}}{\Gamma\big((n+1)(({\wh \beta_0}/{\A_2})-\gamma)\big)} \Big(\rho_{(n+1)(({\wh \B_0}/{\A_2})-\gamma)}^{y,0}
+\rho_{-\gamma+n(({\B_0}/{\A_2})-\gamma)}^{y,\B_0} \Big)(t,x-y),\end{split}
\end{equation}
and so combing all estimates together
$$\sum_{n=0}^\infty |q_n(t,x,y)|\le \left(\sum_{n=0}^\infty \frac{\big[C \Gamma(({\beta_0}/{\A_2})-\gamma )\big]^{n+1}}{\Gamma\big((n+1)(({\wh\beta_0}/{\A_2})-\gamma)\big)}\right)\Big(\rho_{2(({\wh \B_0}/{\A_2})-\gamma)}^{y,0}+\rho_{-\gamma}^{y,\B_0} \Big)(t,x-y),$$ which means that $q(t,x,y):=\sum_{n=0}^\infty q_n(t,x,y)$ is absolutely convergent, and
$$|q(t,x,y)|\le C_1  \Big(\rho_{2(({\wh\B_0}/{\A_2})-\gamma)}^{y,0}+\rho_{-\gamma}^{y,\B_0} \Big)(t,x-y).$$
Thus, \eqref{p3-1-1} is proved by using the fact that $\B_0\le 2\wh \B_0$ and changing the constant $\G$ properly. From \eqref{eq3-1}
and the fact $q(t,x,y)=\sum_{n=0}^{\infty}q_n(t,x,y)$, it is easy to see that $q(t,x,y)$ solves the equation \eqref{eq2-2}.

(ii) Suppose that there are some $y_0\in \R^d$ and $r\in(0,\infty]$ such that $\alpha(z)=\alpha(y_0)$ for all $z\in B(y_0,r_0)$.
According to \eqref{l2-1-33a}, it is easy to verify that for every $z \in \R^d$,
\begin{equation}\label{p3-1-7a}
\begin{split}
|q_0(s,z,y_0)|&\p \big(|z-y_0|^{\beta_0}\wedge 1\big)\cdot
\begin{cases}
\frac{1}{(s^{{1/\A(y_0)}}+|z-y_0|)^{d+\A(y_0)}},& |z-y_0|\le r_0,\\
\frac{1}{|z-y_0|^{d+\A_1}},&|z-y_0|> r_0
\end{cases}\\
&\p \rho^{y_0,\B_0}_0(s,z-y_0).
\end{split}
\end{equation}
This along with the fact that $|q_0(t-s,x,z)|\p \rho^{z,\B_0}_{-\G}(t-s,x-z)$ for every $x,z\in \R^d$ yields that for all $x\in \R^d$,
\begin{align*}
|q_1(t,x,y_0)|&\p \int_0^t\int_{\R^d}\rho^{z,\B_0}_{-\G}(t-s,x-z)\rho^{y_0,\B_0}_0(s,z-y_0)\,dz\,ds\\
&\le C\mathscr{B}\big(({\wh\B_0}/{\A_2})-\G,({ \B_0}/{\A_2})-\G\big)
\big(\rho^{y_0,0}_{({\wh \B_0}/{\A_2})+(({\wh \B_0}/{\A_2})-\G)}+\rho^{y_0,\B_0}_{({\B_0}/{\A_2})-\G}\big)(t,x-y_0).
\end{align*}
Following the same arguments in part (i), we find that for every $n \ge 0$,
\begin{align*} |q_n(t,x,y_0)|
&\le \frac{\big[C \Gamma(({\beta_0}/{\A_2})-\gamma )\big]^{n+1}}{\Gamma\big((n+1)(({\wh \beta_0}/{\A_2})-\gamma)\big)}
\Big(\rho_{({\wh \B_0}/{\A_2})+n(({\wh\B_0}/{\A_2})-\gamma)}^{y_0,0}+\rho_{n(({ \B_0}/{\A_2})-\gamma)}^{y_0,\B_0} \Big)(t,x-y_0),\end{align*}
which implies $q(t,x,y):=\sum_{n=1}^{\infty}q_n(t,x,y)$ is absolutely convergent and gives us \eqref{p3-1-2b} immediately.

(iii) It suffices to prove \eqref{p3-1-1a} for the case that $|x-x'|\le R_1$  holds with some $R_1>0$, since the case
that $|x-x'|>1$ follows from \eqref{p3-1-1} immediately.
For any (fixed) $\e>0$ small enough, let $\tilde \rho^{y,\B}_{\G}$ be defined by \eqref{p3-1-7}.  By \eqref{p3-1-1a-1}, we have
\begin{equation}\label{p3-1-8a}
\begin{split}
&|q_0(t,x,y)-q_0(t,x',y)|\\
&\p (|x-x'|^{\tilde \T}\wedge 1)\times \bigg[\left(\tilde \rho^{y,0}_{\theta-\gamma}+\tilde \rho^{y,\B_0}_{\theta-\gamma-({\B_0}/{\A(x)})}\right)(t,x-y)  + \left(\tilde \rho^{y,0}_{\theta-\gamma}+\tilde \rho^{y,\B_0}_{\theta-\gamma-({\B_0}/{\A(x')})}\right)(t,x'-y)\bigg],
\end{split}
\end{equation}
where $\tilde \T:=\A_1(\B_0-\A_2\theta)/{\A(x)}$. Since $\B_0^*=\B_0<\A(x)$, by  \eqref{e1-3} we can find a constant
$R_1>0$ small enough such that $\B_0< \A(x)\wedge\A(x')$ for any $x'\in \R^d$ with $|x-x'|\le R_1$.
Then, according to \eqref{eq2-2}, it holds that
\begin{align*}
&|q(t,x,y)-q(t,x',y)|\\
&\le |q_0(t,x,y)-q_0(t,x',y)|+\int_0^t\int_{\R^d}
|q_0(t-s,x,z)-q_0(t-s,x',z)||q(s,z,y)|\,dz\,ds\\
&\p (|x-x'|^{\tilde \T}\wedge 1)\\
&\quad\times\bigg\{ \Big[\big(\tilde \rho^{y,0}_{\theta-\gamma}+\tilde \rho^{y,\B_0}_{\theta-\gamma-({\B_0}/{\A(x)})}\big)(t,x-y)
+ \left(\tilde \rho^{y,0}_{\theta-\gamma}+\tilde \rho^{y,\B_0}_{\theta-\gamma-({\B_0}/{\A(x')})}\right)(t,x'-y)\Big]\\
&\quad \quad+\int_0^t\int_{\R^d} \Big[\big(\tilde \rho^{z,0}_{\theta-\gamma}+\tilde \rho^{z,\B_0}_{\theta-\gamma-({\B_0}/{\A(x)})}\big)(t-s,x-z)\\
&\qquad\qquad\quad\qquad+ \big(\tilde \rho^{z,0}_{\theta-\gamma}+\tilde \rho^{z,\B_0}_{\theta-\gamma-({\B_0}/{\A(x')})}
\big)(t-s,x'-z)\Big]
\Big[\big(
\tilde \rho^{y,0}_{({\B_0}/{\A_2})-\G}+\tilde \rho^{y,\B_0}_{-\G}\big)(s,z-y)\Big]\,dz\,ds \bigg\}\\
&\p (|x-x'|^{\tilde \T}\wedge 1)\cdot\Big[\tilde \rho^{y,0}_{\theta-2\gamma
+(\B_0/\A_2)-(\B_0/\A_1)}(t,x-y)
+ \tilde \rho^{y,0}_{\theta-2\gamma+(\B_0/\A_2)-(\B_0/\A_1)}(t,x'-y)\Big].
\end{align*}
Here the second inequality follows from \eqref{p3-1-1} and \eqref{p3-1-8a}, thanks to the fact that
$\B_0< \A(x)\wedge \A(x')$, and the last inequality is due to Remark \ref{r3-3}(2) (which indicates that the convolution inequality \eqref{l3-3-0} holds for $\tilde \rho$) and the fact that $\gamma<\theta$.

(iv) Now we are going to prove the case that $\kappa(x,z)$ depends on $z$.
Due to the assumption $({\A_2}/{\A_1})-1<{\B_0^{**}}/{\A_2}$, we can choose $\gamma>0$ small enough such that
$\gamma+({\A_2}/{\A_1})-1=:\G_2 <{\B_0^{**}}/{\A_2}$. Using \eqref{l2-1-1} and following the same arguments in part (i) with $\G_2 $ instead of $\G$, we can obtain that for
every $n\ge 1$ and $(t,x,y)\in (0,1]\times \R^d \times \R^d$,
\begin{equation}\label{p3-1-9a}
\begin{split}
|q_n(t,x,y)|\le &\frac{\big[C \Gamma(({\beta_0^{**}}/{\A_2})-\G_2  )\big]^{n+1}}{\Gamma\big((n+1)(({\beta_0^{**}}/{\A_2})-\G_2 )\big)} \bigg(\rho_{(n+1)(({\B_0^{**}}/{\A_2})-\G_2 )}^{y,0}+\rho_{-\G_2 +n({\B_0^{**}}/{\A_2}
-\G_2 )}^{y,\B_0} \bigg)(t,x-y).
\end{split}
\end{equation}
Using this and \eqref{l2-1-1}, we find immediately that $\sum_{n=0}^{\infty}|q_n(t,x,y)|<\infty$ and so
\eqref{p3-1-2} is true.

Note that, according to Proposition \ref{p3-1-1a-p}(2), in this case \eqref{p3-1-8a} holds with
$\gamma_2$ in place of  $\G $. (Note that here $\B_0$ is not replaced by $\B_0^{**}$.) This along with \eqref{p3-1-2} and the same argument in (iii) gives us \eqref{p3-1-2a}.

(v) Suppose that the assumptions in part (2) of the Proposition hold, and that there are some $y_0\in \R^d$ and $r_0\in(0,\infty]$ such that $\alpha(z)=\alpha(y_0)$ for all $z \in B(y_0,r_0)$. Then, using \eqref{l2-1-33} and repeating the arguments in part (ii) above,
we can prove \eqref{p3-1-2c}.
\end{proof}

Furthermore, according to Proposition \ref{p2-1}, we also have the following estimates for $q(t,x,y)$.
\begin{proposition}\label{P-new}\begin{itemize}
\item[(1)] Suppose that $\K(x,z)$ is independent of $z$. Then there exists a constant $c_5:=c_5(\A,\K)>0$ such that
for every $t \in (0,1]$ and $x,y\in \R^d$,
\begin{equation}\label{p3-1-2e}
|q(t,x,y)|\le
\begin{cases}
c_5t^{-1-({d}/{\A(y)})},\quad& |x-y|\le t^{{1}/{\A(y)}},\\
\frac{c_5}{|x-y|^{d+\A_2}},\quad& t^{{1}/{\A(y)}}< |x-y| \le 1,\\
\frac{c_5}{|x-y|^{d+\A_1}},\quad& |x-y|>1.
\end{cases}
\end{equation}

\item[(2)] Let $\B_0^{**}\in(0,\B_0]\cap(0,{\alpha_2}/{2}).$ Suppose that $\K(x,z)$ depends on $z$, and
$\B_0^{**}/{\A_2}>({\A_2}/{\A_1})-1$. Then exists a constant $c_6:=c_6(\A,\K)>0$ such that
for every $t \in (0,1]$ and $x,y\in\R^d$,
\begin{equation}\label{p3-1-2f}
|q(t,x,y)|\le
\begin{cases}
c_6t^{-1-({d}/{\A(y)})},\quad& |x-y|\le t^{{1}/{\A(y)}},\\
\frac{c_6t^{1-({\A_2}/{\A_1})}}{|x-y|^{d+\A_2}},\quad& t^{{1}/{\A(y)}}< |x-y| \le 1,\\
\frac{c_6t^{1-({\A_2}/{\A_1})}}{|x-y|^{d+\A_1}},\quad& |x-y|>1.
\end{cases}
\end{equation}
\end{itemize}
\end{proposition}

\begin{proof}
Without loss of generality, throughout the proof we still assume that $\B_0< \A_2/2$ and $\B_0^{**}=\B_0$.

We first suppose that $\K(x,z)$ is independent of $z$.
As in the proof of Proposition \ref{p3-1}(1), we know that
$q(t,x,y)=\sum_{n=0}^{\infty}q_n(t,x,y)$
is absolutely convergent. In particular, \begin{equation}\label{p3-1-9}
|q(t,x,y)|\le |q_0(t,x,y)|+\sum_{n=1}^{\infty}|q_n(t,x,y)|.
\end{equation}
According to \eqref{p3-1-6a}, we have
\begin{align*}
 \sum_{n=1}^{\infty}|q_n(t,x,y)|
& \le
 \left(\sum_{n=1}^\infty \frac{\big[C \Gamma(({\beta_0}/{\A_2})-\gamma)\big]^{n+1}}{\Gamma\big((n+1)(({\beta_0}/{\A_2})-\gamma)\big)}
 \right)\Big(\rho_{3(({\B_0}/{\A_2})-\gamma)}^{y,0}+\rho_{({\B_0}/{\A_2})-2\gamma}^{y,\B_0} \Big)(t,x-y) \p \rho^{y,0}_0(t,x-y),
\end{align*}
where in the last inequality we choose $\G>0$ small enough such that
$\G<{\B_0}/({2\A_2})$.
Combining this estimate with \eqref{p2-1-1} and  \eqref{p3-1-9},  we obtain \eqref{p3-1-2e} immediately.

If $\K(x,z)$ depends on $z$ and ${\B_0}/{\A_2}>({\A_2}/{\A_1})-1$, then one can use \eqref{p2-1-1a} and follow
the same procedure above to prove \eqref{p3-1-2f}.
\end{proof}

\medskip

\section{Estimates for $p(t,x,y)$}\label{section4}
\subsection{Existence and upper bounds for $p(t,x,y)$}

By making full use of the estimates for $q(t,x,y)$ in Proposition
\ref{p3-1}, we now can prove that $p(t,x,y)$ is well defined by
\eqref{eq2-1}.

\begin{proposition}\label{p3-2} \begin{itemize}
\item[(1)] If $\K(x,z)$ is independent of $z$, then $p(t,x,y)$ is well defined by
\eqref{eq2-1}, and for every $\gamma,c_0>0$, there exist positive
constants $c_1:=c_1(\A,\K,\G,c_0)$ and $R_0:=R_0(\A,\K,\G,c_0)$ such that for any $t\in (0,1]$ and $x,y\in \R^d$,
\begin{equation}\label{p3-2-1}
p(t,x,y)\le
\begin{cases}
c_1t^{-{d/\A(x)}}, & |y-x|\le c_0t^{{1/\A(x)}},\\
\frac{c_1t^{1-\G}}{|x-y|^{d+\A(x)}},& c_0t^{{1/\A(x)}}\le |y-x|\le R_0,\\
\frac{c_1t^{1-\gamma}}{|x-y|^{d+\A_1}},&|y-x|> R_0.
\end{cases}
\end{equation} Suppose additionally that there are some $x_0 \in \R^d$ and $r_0\in(0,\infty]$ such that
$\alpha(z)=\alpha(x_0)$ for every $z \in B(x_0,r_0)$. Then we can
find a constant $c_2:=c_2(\A,\K,\G,r_0)>0$ such that for every $t\in (0,1]$ and $y \in
\R^d$,
\begin{equation}\label{p3-2-2}
p(t,x_0,y)\le
\begin{cases}
\frac{c_2 t }{(t^{{1/\A(x_0)}}+|x_0-y|)^{d+\A(x_0)}},& |x_0-y|\le {r_0}/{2},\\
\frac{c_2 t^{1-\G} }{|x_0-y|^{d+\A_1}},&|x_0-y|>{r_0}/{2}.
\end{cases}
\end{equation}

\item[(2)] If $\K(x,z)$ depends on $z$ and $({\A_2}/{\A_1})-1<\min\{{\B_0}/{\A_2},1/2\}$, then $p(t,x,y)$ is well defined by
\eqref{eq2-1}, and for every $\gamma, c_0>0$, there exist constants
$c_1:=c_1(\A,\K,\G,c_0)$ and $R_0:=R_0(\A,\K,\G,c_0)$ such that for any $t\in (0,1]$ and $x,y\in \R^d$,
\begin{equation}\label{p3-2-1a}
p(t,x,y)\le
\begin{cases}c_1t^{-{d/\A(x)}}, & |y-x|\le c_0 t^{{1/\A(x)}},\\
\frac{c_1t^{1-\tilde\G}}{|x-y|^{d+\A(x)}},& c_0 t^{{1/\A(x)}}\le |y-x|\le R_0,\\
\frac{c_1t^{1-\tilde\G }}{|x-y|^{d+\A_1}},&|y-x|> R_0,
\end{cases}
\end{equation}
where $\tilde\G:=({\A_2}/{\A_1})-1+\G$. Suppose additionally there exist some
$x_0 \in \R^d$ and $r_0\in(0,\infty]$ such that $\alpha(z)=\alpha(x_0)$ for
every $z \in B(x_0,r_0)$. Then, for any $c_0>0$, we can find a constant
$c_2:=c_2(\A,\K,\G,r_0,c_0)>0$ such that for every $t\in (0,1]$ and $y \in \R^d$,
\begin{equation}\label{p3-2-2a}
p(t,x_0,y)\le
\begin{cases}
{c_2}{t^{-{d/\A (x_0)}}},&\quad |y-x_0|< c_0t^{{1/\A(x_0)}},\\
\frac{c_2 t^{2-{(\A_2/\A_1)}} }{|y-x_0|^{d+\A(x_0)}},& \quad
c_0t^{{1/\A(x_0)}}\le |y-x_0|< {r_0}/{2},\\
\frac{c_2 t^{1-\tilde\G} }{|x-y_0|^{d+\A_1}},&\quad |y-x_0|\ge
{r_0}/{2}.
\end{cases}
\end{equation}
\end{itemize}
\end{proposition}

\begin{proof} Without loss of generality, throughout this proof we will assume that $2\B_0<\A_2$ and $\B_0^{**}=\B_0$; otherwise, we will replace $\B_0$ by $\B_0^{**}\in (0,\B_0]\cap (0,{\A_2}/{2}).$ For simplicity, we only verify the case that $c_0=1$.

(i) We first consider the case that $\K(x,z)$ is independent of $z$.
Note that $$p^z(t,x-z)\p \rho_1^{z,0}(t,x-z),\quad t\in (0,1],x,z\in \R^d.$$ This, along with
\eqref{p3-1-1} and \eqref{l3-3-1}, yields that for every $(t,x,y)\in
(0,1]\times \R^d \times \R^d$ and {any constant
$0<\G<{\B_0}/{\A_2},$}
\begin{equation*}
\int_0^t \int_{\R^d} p^z(t-s,x-
z)q(s,z,y)\,dz\,ds\p
\big(\rho^{y,\B_0}_{1-\gamma}+
\rho^{y,0}_{1+({\B_0}/{\A_2})-\gamma}\big)\big(t,x-y\big).
\end{equation*}
Therefore, $p(t,x,y)$ is well defined by \eqref{eq2-1}, and
\begin{equation}\label{p3-2-3}
p(t,x,y)\p \Big(\rho^{y,0}_1+\rho^{y,\B_0}_{1-\gamma}+
\rho^{y,0}_{1+({\B_0}/{\A_2})-\gamma}\Big)\big(t,x-y\big)\p
\big(\rho^{y,0}_1+\rho^{y,\B_0}_{1-\gamma}\big)\big(t,x-y\big).
\end{equation}

When $|x-y|\le t^{{1/\A(x)}}\le t^{{1/\A_2}}$,
$$\rho^{y,\B_0}_{1-\gamma}\big(t,x-y\big)\p t^{-\G}|x-y|^{\B_0}\rho^{y,0}_1\big(t,x-y\big)
\p t^{({\B_0}/{\A_2})-\G}\rho^{y,0}_1\big(t,x-y\big).$$
Hence, due to $\G<{\B_0}/{\A_2}$, we get
\begin{equation}\label{p3-2-3a}
\begin{split}
p(t,x,y)\p \rho^{y,0}_1\big(t,x-y\big)\p t^{-{d/\A(y)}}\p
t^{-{d}/{\overline{\A}(y;t^{{1/\A_2}})}}\p t^{-{d/\A(x)}},
\end{split}
\end{equation}
where the third inequality follows from \eqref{l3-2-8}.

Next, we will verify the other two cases in the upper bound \eqref{p3-2-1}. Observe that
for every $t \in (0,1]$ and $x,y,z\in \R^d$ such that
$|x-y|\le 1$ and $|z|\le 1$,
\begin{align*}
 \frac{1}{(t^{{1/\A(y)}}+|z|)^{d+\A(y)}}
&=
\frac{1}{(t^{{1/\A(x)}}t^{({1/\A(y)})-({1/\A(x)})}+|z|)^{d+\A(y)}}
 \p \frac{1}{(t^{{1/\A(x)}}t^{{|\A(x)-\A(y)|}/{\A_1^2}}+|z|)^{d+\A(y)}}\\
&\p \frac{1}{(t^{{1/\A(x)}}t^{{|\A(x)-\A(y)|}/{\A_1^2}}+|z|t^{{|\A(x)-\A(y)|}/{\A_1^2}})^{d+\A(y)}}\\
&=t^{-{|\A(x)-\A(y)|(d+\A(y))}/{\A_1^2}}\cdot\frac{1}{(t^{{1/\A(x)}}+|z|)^{\A(y)-\A(x)}}\cdot
\frac{1}{(t^{{1/\A(x)}}+|z|)^{d+\A(x)}}\\
&\p t^{-{(d+\A_2)|\A(x)-\A(y)|}/{\A_1^2}}\cdot
t^{-{|\A(x)-\A(y)|}/{\A_1}}\cdot
\frac{1}{(t^{{1/\A(x)}}+|z|)^{d+\A(x)}} \p
\frac{t^{-C_1R_0^{\B_0}}}{(t^{{1/\A(x)}}+|z|)^{d+\A(x)}},
\end{align*}
where the second inequality we used
$t^{{|\A(x)-\A(y)|}/{\A_1^2}}\le 1$, and the third inequality follows from the fact that for all $t\in (0,1]$ and $x,y,z\in \R^d$ with $|z|\le 1$,
\begin{align*}
\frac{1}{(t^{{1/\A(x)}}+|z|)^{\A(y)-\A(x)}}\p t^{-{|\A(y)-\A(x)|}/{\A(x)}}.
\end{align*}
Thus, choosing $R_0:=R_0(\A,\G)$ small enough  such that
$C_1R_0^{\B_0}<\G$ and using the definition of $\rho^{x,0}_0$, we can
get that for every $t \in (0,1]$ and $x,y,z\in \R^d$ with
$|x-y|\le R_0$ and $|z|\le 1$,
\begin{equation}\label{l4-1-5}
\rho^{y,0}_0(t,z)\p t^{-\G}\rho^{x,0}_0(t,z).
\end{equation}
By the definition of \eqref{e3-0}, we know immediately that \eqref{l4-1-5} still holds for
every $|x-y|\le R_0$ and $|z|> 1$. This is, \eqref{l4-1-5} holds for all $x,y,z\in \R^d$ with $|x-y|\le R_0$.

Hence, combining \eqref{p3-2-3}, \eqref{p3-2-3a} with \eqref{l4-1-5} and changing the constant $\G$ properly, we find that
\begin{align*}
p(t,x,y)&\p
 \big(\rho^{y,0}_{1}+\rho^{y,\B_0}_{1-\gamma}
\big)\big(t,x-y\big)\p \begin{cases}t^{-{d/\A(x)}}, & |y-x|\le t^{{1/\A(x)}},\\
\frac{t^{1-\G}}{|x-y|^{d+\A(x)}},& t^{{1/\A(x)}}\le |y-x|\le R_0,\\
\frac{t^{1-\gamma}}{|x-y|^{d+\A_1}},&|y-x|> R_0.
\end{cases}
\end{align*}

(ii) If there are some $x_0 \in \R^d$ and $r_0\in(0,\infty]$ such
that $\alpha(z)=\alpha(x_0)$ for all $z \in B(x_0,r_0)$, then for
every $y \in B(x_0,{r_0}/{2})$ and $z \in
B(y,{r_0}/{2})$,
$
\A(z)=\A(y)=\A(x_0).
$
Hence, by \eqref{p3-1-2b}, we have
\begin{equation*}
|q(t,x,y)|\p
\Big(\rho^{y,0}_{{\B_0}/{\A_2}}+\rho^{y,\B_0}_0\Big)(t,x-y)\quad\textrm{ for all }
x\in \R^d\textrm{ and }y \in B\big(x_0,{r_0}/{2}\big).
\end{equation*}
Based on the inequality above, the computation in part (i) is valid
with $\G=0$ for $x=x_0$ and $y \in B(x_0,{r_0}/{2})$, which
proves \eqref{p3-2-2} for the case that $y \in
B(x_0,{r_0}/{2})$. The upper bound for the case
$|y-x_0|>{r_0}/{2}$ is just the same as that of \eqref{p3-2-1}.

(iii) If $\K(x,z)$ depends on $z$ and
${(\A_2/\A_1)}-1<{\B_0}/{\A_2}$, then, according to
\eqref{p3-1-2}, we know that the computation in part (i) holds with
$\G$ replaced by $\tilde\G$. Thus, following the argument in part
(i), we can obtain \eqref{p3-2-1a}. Similarly as in (ii),
\eqref{p3-2-2a} could be verified by using \eqref{p3-1-2c}.
\end{proof}

\begin{remark}
The estimate \eqref{p3-2-2} indicates that if $\A(x)$ is a
constant $\A\in(0,2)$ locally, then we can get a upper bound for
$p(t,x,y)$ which is also locally comparable with that for the heat
kernel of rotationally symmetric $\A$-stable process.  In particular, when
$\A(x)\equiv\A$ for all $x \in \R^d$, \eqref{p3-2-2a} coincides with
the upper bound given in \cite[Therem 1.1]{CZ}.
\end{remark}

Besides Proposition \ref{p3-2}, we also can obtain the following upper bound for $p(t,x,y)$, which is based on Proposition \ref{P-new}.
\begin{proposition}\label{p3-3} \begin{itemize}
\item [(1)] If $\K(x,z)$ is independent of $z$, then for any $c_0>0$, there exists a positive
constant $c_1:=c_1(\A,\K,c_0)$ such that for every $t \in (0,1]$ and $x,y\in \R^d$,
\begin{equation}\label{p3-3-1}
p(t,x,y)\le
\begin{cases}
 c_1t^{-{d}/{\A(x)}},\quad& |x-y|\le c_0t^{{1}/{\A(x)}},\\
\frac{c_1t}{|x-y|^{d+\A_2}},\quad& c_0 t^{{1}/{\A(x)}}< |x-y| \le 1,\\
\frac{c_1t}{|x-y|^{d+\A_1}},\quad& |x-y|>1.
\end{cases}
\end{equation}

\item [(2)] Let $\B_0^{**}\in(0,\B_0]\cap(0,{\alpha_2}/{2}).$ Suppose that $\K(x,z)$ depends on $z$ and ${\B_0^{**}}/{\A_2}>({\A_2}/{\A_1})-1$. Then, for any $c_0>0$,  there exists a positive
constant $c_2:=c_2(\A,\K,c_0)$ such that for every $t \in (0,1]$ and $x,y\in \R^d$,
\begin{equation}\label{p3-3-1a}
p(t,x,y)\le
\begin{cases}
c_2t^{-{d}/{\A(x)}},\quad& |x-y|\le c_0t^{{1}/{\A(x)}},\\
\frac{c_2t^{2-({\A_2}/{\A_1})}}{|x-y|^{d+\A_2}},\quad& c_0t^{{1}/{\A(x)}}< |x-y| \le 1,\\
\frac{c_2t^{2-({\A_2}/{\A_1})}}{|x-y|^{d+\A_1}},\quad& |x-y|>1.
\end{cases}
\end{equation}
\end{itemize}
\end{proposition}
\begin{proof} Throughout the proof, we assume that $\B_0<\A_2/2$ and that $\B_0=\B_0^{**}.$ We only verify the case that $c_0=1$.

(1) We first suppose that $\K(x,z)$ is independent of $z$.

{\bf Case (a): $|x-y|\le t^{1/\A(x)}$.}

According to \eqref{p3-2-1}, we can easily see that \eqref{p3-3-1} holds.

{\bf Case (b): $|x-y|>2$.}

We have
\begin{equation}\label{p3-3-2}
\begin{split}
 \left|\int_{\R^d} p^z(t-s,x-z)q(s,z,y)\,dz\right|
&\le \int_{\{|z-y|>{|x-y|}/{2}\}} p^z(t-s,x-z)|q(s,z,y)|\,dz\\
&\quad
+\int_{\{|z-y|\le {|x-y|}/{2}\}} p^z(t-s,x-z)|q(s,z,y)|\,dz\\
&=:J_1+J_2.
\end{split}
\end{equation}When $|z-y|\ge {|x-y|}/{2}\ge 1$, it follows from \eqref{p3-1-2e} that for any $s\in (0,1]$ and $y,z\in \R^d$,
\begin{equation*}
|q(s,z,y)|\p \frac{1}{|z-y|^{d+\A_1}}\p \frac{1}{|x-y|^{d+\A_1}}.
\end{equation*}
Therefore, by \eqref{p1-1-1-1-0} and \eqref{l3-1-1},
\begin{align*}
J_1&\p \frac{1}{|x-y|^{d+\A_1}}\int_{\R^d}\rho^{z,0}_1(t-s,x-z)\, dz\p \frac{1}{|x-y|^{d+\A_1}}.
\end{align*}
Meanwhile, it is not difficult to see that condition $|z-y|\le {|x-y|}/{2}$ implies
$|z-x|\ge {|x-y|}/{2}\ge 1$, and so
\begin{equation*}
p^z(t-s,x-z)\p \frac{t-s}{|x-z|^{d+\A_1}}\p \frac{t}{|x-z|^{d+\A_1}}.
\end{equation*}
This, along with \eqref{p3-1-1} and \eqref{l3-1-1-00}, yields that
\begin{align*}
J_2&\p \frac{t}{|x-y|^{d+\A_1}}\int_{\R^d}\left(
\rho^{y,0}_{({\B_0}/{\A_2})-\G}+\rho^{y, \B_0}_{-\G}\right)(s,z-y)\, dz\p \frac{t}{|x-y|^{d+\A_1}}\cdot s^{-1+({\B_0}/{\A_2})-\G}.
\end{align*}

Combining all the estimates above together and choosing $0<\G<{\B_0}/{\A_2}$, we find that
\begin{align*}
 \left|\int_0^t \int_{\R^d}p^z(t-s,x-z)q(s,z,y)\,dz\,ds\right|
&\p \int_0^t (J_1+J_2)\,ds\p \frac{t}{|x-y|^{d+\A_1}}\cdot \big(1+t^{({\B_0}/{\A_2})-\G}\big)\!\p \frac{t}{|x-y|^{d+\A_1}}.
\end{align*}
Then, the desired assertion \eqref{p3-3-1}  immediately follows from the estimate above and \eqref{eq2-1}.

{\bf Case (c): $t^{{1}/ \A(x)}\le |x-y|\le 2$.}

We still define $J_1$ and $J_2$ by those in \eqref{p3-3-2}.
If $|z-y|\ge {|x-y|}/{2}$,
then, by  \eqref{p3-1-2e}, we have
\begin{align*}
|q(s,z,y)|&\p
\begin{cases}
s^{-d/\A(y)},& |z-y|\le s^{1/\A(y)},\\
|z-y|^{-d-\A_2},& s^{{1}/{\A(y)}}\le|z-y|\le 1,\\
 |z-y|^{-d-\A_1},& |z-y|>1
\end{cases}
\\
&\p
\begin{cases}
|x-y|^{-d-\A(y)},& |z-y|\le s^{{1}/{\A(y)}},\\
|x-y|^{-d-\A_2},& s^{{1}/{\A(y)}}\le|z-y|\le 1,\\
|x-y|^{-d-\A_1},& |z-y|>1
\end{cases}
\\
&\p |x-y|^{-d-\A_2},
\end{align*}
where in the second inequality we have used that fact that if
$|z-y|\le s^{1/\A(y)}$ and $|x-y|\p |z-y|$, then
$$s^{-d/\A(y)}\p |z-y|^{-d-\A(y)}\p |x-y|^{-d-\A(y)},$$
the last inequality follows from the fact that $|x-y|^{-d-\A_1}\p |x-y|^{-d-\A_2}$, thanks to $|x-y|\le 2$.
Hence,
\begin{align*}
J_1&\p \frac{1}{|x-y|^{d+\A_2}} \int_{\R^d}\rho^{z,0}_1(t-s,x-z)\,dz \p \frac{1}{|x-y|^{d+\A_2}}.
\end{align*}
At the same time, if $|z-y|\le {|x-y|}/{2}$, then
\begin{align*}
p^z(t-s,x-z)&\p \frac{t-s}{|x-z|^{d+\A(z)}}\p \frac{t}{|x-y|^{d+\A(z)}}\p \frac{t}{|x-y|^{d+\A_2}},
\end{align*}
where the last inequality follows from $|x-y|^{-d-\A(z)}\p |x-y|^{-d-\A_2}$ since $|x-y|\le 2$.
Combining this estimate with \eqref{p3-1-1},  we arrive at
$$
J_2 \p \frac{t}{|x-y|^{d+\A_2}}\int_{\R^d}\left(
\rho^{y,0}_{({\B_0}/{\A_2})-\G}+\rho^{y,\B_0}_{-\G}\right)(s,z-y)\,dz \p \frac{t}{|x-y|^{d+\A_2}}\cdot s^{-1+(\B_0/{\A_2})-\G}.
$$

Using all the estimates above and choosing $0<\G<{\B_0}/{\A_2}$, we get
\begin{align*}
 \left|\int_0^t \int_{\R^d}p^z(t-s,x-z)q(s,z,y)\,dz\,ds\right|
&\p \int_0^t (J_1+J_2)\,ds\p \frac{t}{|x-y|^{d+\A_2}}\cdot \left(1+t^{({\B_0}/{\A_2})-\G}\right)\!\p \frac{t}{|x-y|^{d+\A_2}}.
\end{align*}
This, along with \eqref{eq2-1} immediately yields \eqref{p3-3-1}.

(2) If $\K(x,z)$ depends on $z$ and ${\B_0}/{\A_2}>({\A_2}/{\A_1})-1$, then, applying \eqref{p3-1-2f} and following the same arguments as above, we can prove
\eqref{p3-1-1a}. The details are omitted here.
\end{proof}

\subsection{H\"{o}lder regularity and gradient estimates of $p(t,x,y)$}
In this part, we consider the H\"{o}lder regularity and gradient estimates of $p(t,\cdot,y)$.

\begin{lemma}\label{l4-4}
There exists a constant $c_1:=c_1(\A,\K)>0$ such that for all $x,x',y\in \R^d$ and $t \in (0,1]$,
\begin{equation}\label{l4-4-1}
\begin{split}
\left|\nabla p^y(t,x)-\nabla p^y(t,x')\right|\le &c_1
\left[\left(t^{-1/\A(y)}|x-x'|\right)\wedge 1\right]\left(
\rho^{y,0}_{1-(1/\A(y))}(t,x)+\rho^{y,0}_{1-(1/\A(y))}(t,x')\right).
\end{split}
\end{equation}
\end{lemma}
\begin{proof}
We write  $ \LL^y=\LL_{{\K_1}/{2}}^y+\LL_{\wh \K}^y$ with
\begin{align*}
\LL_{{\K_1}/{2}}^y f(x)&=\frac{1}{2}\int_{\R^d}\delta_{f}(x;z)\frac{\K_1}{2}\frac{1}{|z|^{d+\A(y)}}\,dz,\quad f\in C_c^2(\R^d),\\
\LL_{\wh \K}^y f(x)&=\frac{1}{2}\int_{\R^d}\delta_{f}(x;z)\frac{\wh \K(y,z)}{|z|^{d+\A(y)}}\,dz,\quad f\in C_c^2(\R^d),
\end{align*}
where $\K_1>0$ is the constant in \eqref{e1-2}, $\wh \K(y,z):=\K(y,z)-\K_1/2$ and
$$\delta_f(x;z)=f(x+z)+f(x-z)-2f(x).$$ Then, we have
\begin{equation}\label{l4-4-2}
p^y(t,x)=\int_{\R^d}p^y_{{\K_1}/{2}}(t,x-z)p^y_{\wh \K}(t,z)\,dz,\ \ x\in \R^d, t\in (0,1],
\end{equation}
where $p^y_{{\K_1}/{2}}$ and $p^y_{\wh \K}$ denote the fundamental solutions (i.e.\ heat kernel) associated with
the operators $\LL_{{\K_1}/{2}}^y$ and $\LL_{\wh \K}^y$, respectively.

We first show that \eqref{l4-4-1} holds for $p^y_{{\K_1}/{2}}$. Indeed, for every $x,x'\in \R^d$ such that $|x-x'|\le 1$,
\begin{align*}
\big|\nabla p^y_{{\K_1}/{2}}(1,x)-\nabla p^y_{{\K_1}/{2}}(1,x')\big|&=
\left|\int_0^1 \frac{d}{d \theta}\left(\nabla p^y_{{\K_1}/{2}}
\left(1,x+\theta(x'-x)\right)\right)\,d\theta\right|\\
&\le |x-x'|\cdot\int_0^1 \left|\nabla^2 p^y_{{\K_1}/{2}}
\left(1,x+\theta(x'-x)\right)\right|\,d\theta\\
&\p |x-x'|\cdot \int_0^1 \left(1+\left|x+\theta(x'-x)\right|\right)^{-d-\A(y)-2}\,d\theta\\
&\p |x-x'|\cdot \left(1+|x|\right)^{-d-\A(y)-2},
\end{align*}
where in the third inequality we have used \eqref{p1-1-1}, and the last inequality is due to
the fact that for all $\theta\in [0,1]$ and $x,x'\in \R^d$ with $|x-x'|\le 1$
\begin{equation*}
\left(1+\left|x+\theta(x'-x)\right|\right)^{-1}\p \left(1+|x|\right)^{-1}.
\end{equation*}
Also by \eqref{p1-1-1}, we have that for every
$x,x'\in \R^d$ with $|x-x'|>1$,
\begin{align*}
\big|\nabla p^y_{{\K_1}/{2}}(1,x)-\nabla p^y_{{\K_1}/{2}}(1,x')\big|&
\le \big|\nabla p^y_{{\K_1}/{2}}(1,x)\big|+
\big|\nabla p^y_{{\K_1}/{2}}(1,x')\big|\\
&\p (1+|x|)^{-d-\A(y)-1}+
(1+|x'|)^{-d-\A(y)-1}.
\end{align*}
Combining both estimates above yields that for all $x,x'\in \R^d$,
\begin{equation}\label{l4-4-3}
\begin{split}
\big|\nabla p^y_{{\K_1}/{2}}(1,x)-\nabla p^y_{{\K_1}/{2}}(1,x')\big|
\p &\left(|x-x'|\wedge 1\right)\left(
(1+|x|)^{-d-\A(y)}+
(1+|x'|)^{-d-\A(y)}\right).
\end{split}
\end{equation}
Since the Markov process $(X_t^{y,{\K_1}/{2}})_{t\ge0}$ associated with
$ \LL^y_{{\K_1}/{2}}$ is a constant time-change of standard rotationally invariant $\A(y)$-stable process, by
the scaling property, for all $t>0$ and $x \in \R^d$,
$$
 p^y_{{\K_1}/{2}}(t,x)= t^{-d/\A(y)} p^y_{{\K_1}/{2}}(1,t^{-1/\A(y)}x).
$$
This along with \eqref{l4-4-3} yields that for all $x,x'\in \R^d$ and $t\in (0,1]$,
\begin{equation}\label{l4-4-4}
\begin{split}
\big|\nabla p^y_{{\K_1}/{2}}(t,x)-\nabla p^y_{{\K_1}/{2}}(t,x')\big|
\p &\left[\left(t^{-1/\A(y)}|x-x'|\right)\wedge 1\right]\left(
\rho_{1-(1/\A(y))}^{y,0}(t,x)+\rho_{1-(1/\A(y))}^{y,0}(t,x')\right).
\end{split}
\end{equation}
Therefore,  for every $x,x'\in \R^d$ and $t\in (0,1]$,
\begin{align*}
&\left|\nabla p^y(t,x)-\nabla p^y(t,x')\right|\\
&=\left|\int_{\R^d}\big(\nabla p^y_{{\K_1}/{2}}(t,x-z)-\nabla p^y_{{\K_1}/{2}}(t,x'-z)\big)
\cdot p^y_{\wh \K}(t,z)\,dz\right|\\
&\le \int_{\R^d}\big|\nabla p^y_{{\K_1}/{2}}(t,x-z)-\nabla p^y_{{\K_1}/{2}}(t,x'-z)\big|
\cdot p^y_{\wh \K}(t,z)\,dz\\
&\p \left[\left(t^{-1/\A(y)}|x-x'|\right)\wedge 1\right]\cdot
\Bigg(\int_{\R^d}\rho^{y,0}_{1-(1/\A(y))}(t,x-z)\rho^{y,0}_1(t,z)\,dz+\int_{\R^d}\rho^{y,0}_{1-(1/\A(y))}(t,x'-z)\rho^{y,0}_1(t,z)\,dz\Bigg)\\
&\p  \left[\left(t^{-1/\A(y)}|x-x'|\right)\wedge 1\right]\cdot
\left(\rho^{y,0}_{1-(1/\A(y))}(t,x)
+\rho^{y,0}_{1-(1/\A(y))}(t,x')\right),
\end{align*}
where the equality above is due to \eqref{l4-4-2}, the second inequality follows from
\eqref{l4-4-4} and \eqref{p1-1-1-1-0}, and in the last inequality we have used
\eqref{l3-2-1-}. By now we have finished the proof.
\end{proof}

\begin{lemma}\label{l4-1}
For every $0<\varepsilon, \theta<\A_1$ and $0<\gamma<({\theta}/{\A_2})\wedge ((1-({\theta}/{\A_1}))/2)$,
there exist $R_1:=R_1(\A,\K,\G,\T)\in(0,1]$ and $c_1:=c_1(\A,\K,\G,\T,\e)>0$ such that
for all $t \in (0,1]$ and $x,y,z,w\in \R^d$ with $|x-y|\le R_1$,
\begin{equation}\label{l4-1-1}
|p^x(t,z)-p^y(t,z)|\le c_1|x-y|^{\B_0}\tilde \rho^{x,0}_{1-2\G-({\T}/{\A_1})}(t,z),\end{equation}
\begin{equation}\label{l4-1-1-1}\left|\nabla p^x(t,z)-\nabla p^y(t,z)\right|\le c_1|x-y|^{\B_0}\tilde \rho^{x,0}_{1-(1/\A(x))-2\G-
({\T}/{\A_1})}(t,z)
\end{equation}
and
\begin{equation}\label{l4-1-1a}
\left| \LL^w p^x(t,z)- \LL^w p^y(t,z)\right|\le c_1|x-y|^{\B_0}\tilde \rho^{x,0}_{-2\G-({\T}/{\A_1})}(t,z),
\end{equation}
where $\tilde \rho$ is defined by \eqref{p3-1-7}.
\end{lemma}
\begin{proof}
(i) Note that for all $t\in (0,1]$ and $x,y,z \in \R^d$,
\begin{equation}\label{l4-1-2}
\begin{split}
 p^x(t,z)-p^y(t,z)
&=\int_0^t \frac{d}{ds}\bigg(\int_{\R^d}p^x(s,w)p^y(t-s,z-w)\,dw\bigg)\,ds\\
&=\int_0^t \bigg(\int_{\R^d}\big( \LL^x p^x(s,w)\,p^y(t-s,z-w) -p^x(s,w)\, \LL^y
p^y(t-s,z-w)\big)\,dw\bigg)\,ds\\
&=\int_0^{{t}/{2}}\int_{\R^d}\big( \LL^x- \LL^y\big) p^x(s,w)\,p^y(t-s,z-w)\,dw\,ds\\
&\quad +\int_{{t}/{2}}^t\int_{\R^d} p^x(s,w)\big( \LL^x- \LL^y\big)
p^y(t-s,z-w)\,dw\,ds\\
&=\int_0^{{t}/{2}}\bigg(\int_{\R^d}\big( \LL^x- \LL^y\big) p^x(s,w)
\big(p^y(t-s,z-w)-p^y(t-s,z)\big)\,dw\bigg)\,ds\\
&\quad +\int_{{t}/{2}}^t\bigg(\int_{\R^d} \big(p^x(s,w)-p^x(s,z)\big) \big( \LL^x- \LL^y\big)
p^y(t-s,z-w)\,dw\bigg)\,ds\\
&=:J_1+J_2.
\end{split}
\end{equation}
Here, the first equality is due to the following estimate which we will verify later
\begin{equation*}
\int_0^t \left|\frac{d}{ds}\bigg(\int_{\R^d}p^x(s,w)p^y(t-s,z-w)\,dw\bigg)\,\right|ds<\infty.\end{equation*}
By estimates for $p^x(\cdot, \cdot)$ and $p^y(\cdot, \cdot)$ in Lemma \ref{p1-1},
we can change the order of derivatives and integrals in the second equality. The third and the fourth equalities above follow from the
following facts respectively
\begin{align*}
\int_{\R^d}  \LL^x f(z_1) g(z_1) \,dz_1=\int_{\R^d} f(z_1)  \LL^x g(z_1) \,dz_1,\quad f,g \in C_c^{\infty}(\R^d),
\end{align*}
\begin{align*}
\int_{\R^d}  \LL^y f(z_1) g(z_1)\, dz_1=\int_{\R^d} f(z_1)  \LL^y g(z_1) \, dz_1, \quad f,g \in C_c^{\infty}(\R^d)
\end{align*} and
\begin{align*}
\int_{\R^d}  \LL^x f(z_1)\,dz_1=\int_{\R^d}  \LL^y f(z_1)\,dz_1=0,\quad f \in C_c^{\infty}(\R^d).
\end{align*}
Note that, also due to estimates for $p^x(\cdot, \cdot)$ and $p^y(\cdot, \cdot)$,  the equalities above are still true for $p^x(t,\cdot)$ and $p^y(t,\cdot)$.

By Lemma \ref{l1-2}, for any $\G>0$ there is a constant $R_0:=R_0(\A,\K,\G)\in(0,1)$ such that for all $t\in(0,1]$ and $x,y,w\in \R^d$ with $|x-y|\le R_0$,
\begin{equation}\label{l4-1-4}
\big|\big( \LL^x- \LL^y\big)p^x(t,w)\big|\p (|x-y|^{\B_0}\wedge 1)\cdot\tilde \rho^{x,0}_{-\G}(t,w).
\end{equation}
On the other hand, for any $\T \in (0,1)$ and $\G\in (0,1)$, there exists a
constant $R_1:=R_1(\A,\K,\G)\in(0,R_0)$ such that for all
 $x,y,z\in\R^d$ with $|x-y|\le R_1$, $0<s\le t\le1$,
\begin{equation}\label{l4-1-6}
\begin{split}
 \big|p^y(t-s,z-w)-p^y(t-s,z)\big|
&\p \big[\big((t-s)^{-{1/\A(y)}}|w|\big)\wedge 1\big]\cdot
\big[\rho^{y,0}_1(t-s,z-w)+\rho^{y,0}_1(t-s,z)\big]\\
&\p \big[\big((t-s)^{-{1/\A(y)}}|w|\big)^{\T}\wedge 1\big]\cdot
\big[\rho^{y,0}_1(t-s,z-w)+\rho^{y,0}_1(t-s,z)\big]\\
&\p \big[\big((t-s)^{-{1}/{\A_1}}|w|\big)^{\T}\wedge 1\big]\cdot
\big[\rho^{x,0}_{1-\G}(t-s,z-w)+\rho^{x,0}_{1-\G}(t-s,z)\big]\\
&\p \big(|w|^{\T}\wedge 1\big)\cdot
\big[\rho^{x,0}_{1-\G-({\theta}/{\A_1})}(t-s,z-w)+\rho^{x,0}_{1-\G-({\theta}/{\A_1})}(t-s,z)\big],
\end{split}
\end{equation}
where the first inequality follows from \eqref{p1-1-3} and in the third inequality we have used \eqref{l4-1-5}.

Next, we choose $\theta\in(0,1\wedge \A_1)$ and $0<\gamma<({\theta}/{\A_2})\wedge ((1-({\theta}/{\A_1}))/2)$. Then, according to
\eqref{l4-1-4} and \eqref{l4-1-6}, for any $t\in(0,1]$ and $x,y,z\in \R^d$ with $|x-y|\le R_1$,
\begin{align*}
J_1&\p |x-y|^{\B_0}\cdot\bigg(\int_0^{{t}/{2}}\int_{\R^d}\tilde \rho^{x,\T}_{-\G}(s,w)\tilde \rho^{x,0}_{1-\G-({\T}/{\A_1})}(t-s,z-w)\,dw\,ds\\
&\qquad\qquad \qquad +\int_0^{{t}/{2}}\int_{\R^d}\tilde \rho^{x,\T}_{-\G}(s,w)
\tilde \rho^{x,0}_{1-\G-({\T}/{\A_1})}(t-s,z)\,dw\,ds\bigg)\\
&=:|x-y|^{\B_0}\cdot \big(J_{11}+J_{12}\big).
\end{align*}
As mentioned in Remark \ref{r3-3}(2), \eqref{r3-2-1} holds for $\tilde \rho$, from which we can obtain that
\begin{equation*}
J_{11}\p \tilde \rho^{x,0}_{1-2\G-({\T}/{\A_1})}(t,z).
\end{equation*}
At the same time, observe that for every $0<s<{t}/{2}$,
\begin{equation*}
\tilde \rho^{x,0}_{1-\G-({\T}/{\A_1})}(t-s,z)\p \tilde
\rho^{x,0}_{1-\G-({\T}/{\A_1})}(t,z),
\end{equation*}
then we have
\begin{equation}\label{l4-1-8}
\begin{split}
J_{12}&\p \tilde \rho^{x,0}_{1-\G-({\T}/{\A_1})}(t,z)
\int_0^{{t}/{2}}\int_{\R^d}\tilde \rho^{x,\T}_{-\G}(s,w)\,dw\,ds \\
&\p \tilde \rho^{x,0}_{1-\G-({\T}/{\A_1})}(t,z) \int_0^{{t}/{2}}s^{-1-\G+({\T}/{\A_2})}\,ds \p \tilde \rho^{x,0}_{1-2\G-({\T}/{\A_1})}(t,z),
\end{split}
\end{equation}
where in the second inequality we used the fact that \eqref{l3-1-1-00} is true for $\tilde \rho$, see again
Remark \ref{r3-3}(2).
By both estimates for $J_{11}$ and $J_{12}$, we get that for any $t\in(0,1]$ and $x,y,z\in \R^d$ with $|x-y|\le R_1$,
\begin{equation*}
J_1 \p |x-y|^{\B_0}\cdot\tilde \rho^{x,0}_{1-2\G-({\T}/{\A_1})}(t,z).
\end{equation*}
 For $J_2$, we need to handle the singularity near $s=t$. As the same way as before, we can obtain that for all $t\in (0,1]$ and
$x,y,z,w\in \R^d$ with $|x-y|\le R_1$,
\begin{align*}
\big|\big( \LL^x- \LL^y\big)p^y(t-s,z-w)\big|
&\p \big(|x-y|^{\B_0}\wedge 1\big)\cdot\tilde \rho^{y,0}_{-\G}(t-s,z-w)\p \big(|x-y|^{\B_0}\wedge 1\big)\cdot \tilde \rho^{x,0}_{-2\G}(t-s,z-w)\end{align*} and
\begin{align*}
\big|p^x(s,w)-p^x(s,z)\big|\p \big(|z-w|^{\T}\wedge 1\big) \cdot\big(\rho^{x,0}_{1-({\T}/{\A_1})}(s,z)+\rho^{x,0}_{1-({\T}/{\A_1})}(s,w)\big).
\end{align*}
Using both estimates above and following the same argument as that for $J_1$, we can get that  for any $t\in(0,1]$ and $x,y,z\in \R^d$ with $|x-y|\le R_1$, $$J_2 \p
|x-y|^{\B_0}\tilde \rho^{x,0}_{1-2\G-({\T}/{\A_1})}(t,z).$$ Then, \eqref{l4-1-1} is proved.

(ii) Following the  argument of  \eqref{l4-1-2}, we can verify that
for all $t\in (0,1]$ and $x,y,z \in \R^d$,
\begin{align*}
&\nabla p^x(t,z)-\nabla p^y(t,z)\\
&=\int_0^t \frac{d}{ds}\left(\int_{\R^d}
\nabla p^x(s,w)p^y(t-s,z-w)\,dw\right)\,ds\\
&=\int_0^t\bigg(\int_{\R^d}\Big( \LL^x p^x(s,w)\nabla p^y(t-s,z-w) -\nabla p^x(s,w)\, \LL^y
p^y(t-s,z-w)\Big)\,dw\bigg)\,ds\\
&=\int_0^{{t}/{2}}\int_{\R^d}\big( \LL^x- \LL^y\big) p^x(s,w)\,\nabla p^y(t-s,z-w)\,dw\,ds\\
&\quad +\int_{{t}/{2}}^t\int_{\R^d} \nabla p^x(s,w)\big( \LL^x- \LL^y\big)
p^y(t-s,z-w)\,dw\,ds\\
&=\int_0^{{t}/{2}}\bigg(\int_{\R^d}\big( \LL^x- \LL^y\big) p^x(s,w)
\big(\nabla p^y(t-s,z-w)-\nabla p^y(t-s,z)\big)\,dw\bigg)\,ds\\
&\quad +\int_{{t}/{2}}^t\bigg(\int_{\R^d} \big(\nabla p^x(s,w)-\nabla p^x(s,z)\big)  \big( \LL^x- \LL^y\big)
p^y(t-s,z-w)\,dw\bigg)\,ds\\
&=:I_1+I_2,
\end{align*} where in the first and the second equalities we used the integration by part formula
and  in the second equality we also used the fact that $\nabla  \LL^x p^y(t,w)= \LL^x \nabla p^y(t,w)$.

Observe that
$$
t^{-d/\A(y)}\le t^{-d/\A(x)}t^{-\big|\frac{d}{\A(y)}-\frac{d}{\A(x)}\big|}\le
t^{-d/\A(x)}t^{-\frac{|\A(x)-\A(z)|}{\A_1^2}}\le t^{-d/\A(x)}t^{-C|x-z|^{\B_0}}.
$$
Then, according to \eqref{l4-4-1} and the proof of \eqref{l4-1-6},
for any $\T \in (0,1)$, $\G\in (0,1)$, there exists a
constant $R_1:=R_1(\A,\K,\G)\in(0,R_0)$ such that for all $0<t \le 1$ and
 $x,y\in\R^d$ with $|x-y|\le R_1$,
\begin{equation}\label{l4-1-7}
\begin{split}
&\left|\nabla p^y(t,z-w)-\nabla p^y(t,z)\right|\\
&\p \big(|w|^{\T}\wedge 1\big)\cdot
\big[\rho^{x,0}_{1-(1/\A(x))-\G-({\theta}/{\A_1})}(t,z-w)+\rho^{x,0}_{1-(1/\A(x))-\G-({\theta}/{\A_1})}(t,z)\big].
\end{split}
\end{equation}
Choosing $\theta\in(0,1\wedge \A_1)$ and $0<\gamma<({\theta}/{\A_2})\wedge ((1-({\theta}/{\A_1}))/2)$, and using
\eqref{l4-1-4} and \eqref{l4-1-7}, we arrive at that for any $t\in(0,1]$ and $x,y,z\in \R^d$ with $|x-y|\le R_1$,
\begin{align*}
I_1&\p |x-y|^{\B_0}\cdot\bigg(\int_0^{{t}/{2}}\int_{\R^d}\tilde \rho^{x,\T}_{-\G}(s,w)\tilde \rho^{x,0}_{1-(1/\A(x))-\G-({\T}/{\A_1})}(t-s,z-w)\,dw\,ds\\
&\qquad\qquad \qquad\quad +\int_0^{{t}/{2}}\int_{\R^d}\tilde \rho^{x,\T}_{-\G}(s,w)
\tilde \rho^{x,0}_{1-(1/\A(x))-\G-({\T}/{\A_1})}(t-s,z)\,dw\,ds\bigg)\\
&=:|x-y|^{\B_0}\cdot \big(I_{11}+I_{12}\big).
\end{align*}
Noticing that, by Remark \ref{r3-3}(2),  \eqref{l3-2-1-} still holds for $\tilde \rho$, we have
\begin{align*}
I_{11}&\p \tilde \rho^{x,0}_0(t,z)\int_0^{t/2}
\Big(s^{-\G}(t-s)^{-(1/\A(x))-\G-(\theta/\A_1)}+s^{-1+\theta/\A_2-\G}(t-s)^{1-(1/\A(x))-\G-(\theta/\A_1)}\Big)\,ds\\
&\p \tilde \rho^{x,0}_0(t,z) \cdot t^{1-(1/\A(x))-2\G-(\theta/\A_1)}\p
 \tilde \rho^{x,0}_{1-(1/\A(x))-2\G-(\theta/\A_1)}(t,z).
\end{align*}
On the other hand, following the argument of  \eqref{l4-1-8}, we can obtain that
\begin{align*}
I_{12}&\p \tilde \rho^{x,0}_{1-(1/\A(x))-\G-(\theta/\A_1)}(t,z)\cdot\int_0^{t/2}
\int_{\R^d}\tilde \rho^{x,\theta}_{-\G}(s,w)\,dw\,ds\p \tilde \rho^{x,0}_{1-(1/\A(x))-2\G-(\theta/\A_1)}(t,z).
\end{align*}

By the same argument as that for $I_1$, we can also obtain that for every $t \in (0,1]$ and
$x,y,z\in \R^d$ with $|x-y|\le R_1$,
\begin{equation*}
I_2\p |x-y|^{\B_0}\tilde \rho^{x,0}_{1-(1/\A(x))-2\G-(\theta/\A_1)}(t,z).
\end{equation*}
Combining all estimates together, we have shown
\eqref{l4-1-1-1}.

(iii) Following the same procedure as these of \eqref{l4-1-1} and \eqref{l4-1-1-1}, we can also
verify \eqref{l4-1-1a}.
\end{proof}

\begin{proposition}\label{t4-2}
\begin{itemize}
\item[(1)] Suppose that $\K(x,z)$ is independent of $z$. Then, for every $\G>0$ small enough, there exist positive constants
$c_1:=c_1(\A,\K,\G)$ and $R_1:=R_1(\A,\K,\G)$ such that for all $t \in (0,1]$ and $x,x',y \in \R^d$ with $|x-x'|\le R_1$
\begin{equation}\label{t4-2-1}
\begin{split}
\big|p(t,x,y)-p(t,x',y)\big|
&\le c_1
|x-x'|^{(\A(x)-\G)_+\wedge 1}
\big(\rho^{y,0}_{{\G/(2\A_2)}}(t,x-y)+\rho^{y,0}_{{\G/(2\A_2)}}(t,x'-y)\big).
\end{split}
\end{equation}

\item[(2)]  Let $\B_0^{**}\in(0,\B_0]\cap(0,{\alpha_2}/{2}).$ If $\K(x,z)$ depends on $z$ and
${(\A_2/\A_1)}-1<\B_0^{**}/{\A_2}$, then for every $\G>0$, there exist positive constants
$c_1:=c_1(\A,\K,\G)$ and $R_1:=R_1(\A,\K,\G)$ such that for all $t \in (0,1]$ and $x,x',y \in \R^d$ with $|x-x'|\le R_1$
\begin{equation}\label{t4-2-1aa}
\begin{split}
\big|p(t,x,y)-p(t,x',y)\big|\le& c_1|x-x'|^{(\A_1-\G)_+\wedge 1}\big(\rho^{y,0}_{\G_1 }(t,x-y)+
\rho^{y,0}_{\G_1 }(t,x'-y)\big),
\end{split}
\end{equation}
where $\G_1:=1-({\A_2}/{\A_1})+({\G}/({2\A_2}))$.
\end{itemize}
\end{proposition}
\begin{proof} For simplicity, we assume that $\B_0< {\A_2}/{2}$ and $\B_0^{**}=\B_0$ as before.

(1) We first suppose that $\K(x,z)$ is independent of $z$. According
to  \eqref{p1-1-3}, for any $\G>0$,
\begin{equation}\label{t4-2-1a}
\begin{split}
&\big|p^z(t-s,x-z)-p^z(t-s,x'-z)\big|\\
&\p \big[\big((t-s)^{-{1}/{\A(z)}}|x-x'|\big)\wedge 1\big]\cdot
\big[\rho^{z,0}_1(t-s,x-z)+\rho^{z,0}_1(t-s,x'-z)\big]\\
&\p
\big[\big((t-s)^{-{1}/{\A(z)}}|x-x'|\big)^{(\A(x)-\G)_+\wedge
1}\wedge 1\big]
\big[\rho^{z,0}_1(t-s,x-z)+\rho^{z,0}_1(t-s,x'-z)\big]\\
&\p |x-x'|^{(\A(x)-\G)_+\wedge
1}(t-s)^{-{(\A(x)-\G)_+}/{\A(z)}}
\times\big[\rho^{z,0}_1(t-s,x-z)+\rho^{z,0}_1(t-s,x'-z)\big].
\end{split}
\end{equation}

On the one hand, observing that
$$\frac{\A(x)-\G}{\A(z)}\le 1+\frac{|\A(x)-\A(z)|}{\A_1}-\frac{\G}{\A_2}\le
1+\frac{C|x-z|^{\B_0}}{\A_1}-\frac{\G}{\A_2},$$ we can find a constant
$R_2:=R_2(\A,\G)>0$ such that for every $x,x',z\in \R^d$ satisfying
$|x-x'|\le R_2$ and $|x-z|\le 2R_2$ (which imply that $|x'-z|\le
3R_2$), it holds
\begin{equation*}
\frac{\A(x)-\G}{\A(z)}\le 1-\frac{2\G}{3\A_2},
\end{equation*}
which in turn yields that for every $0<s\le t\le 1$ and $x,x',z\in\R^d$ with $|x-x'|\le
R_2$ and $|x-z|\le 2R_2$
\begin{align*}
\big|p^z(t-s,x-z)-p^z(t-s,x'-z)\big|
&\p
|x-x'|^{(\A(x)-\G)_+\wedge 1}\left(\rho^{z,0}_{{2\G}/({3\A_2})}(t-s,x-z)+\rho^{z,0}_{{2\G}/({3\A_2})}(t-s,x'-z)\right).
\end{align*}

On the other hand, when $|x-x'|\le R_2$ and $|z-x|>2R_2$ (which
imply that $|z-x'|>R_2$), we obtain from \eqref{p1-1-1} and the mean
value theorem that
\begin{align*}
|p^z(t-s,x-z)-p^z(t-s,x'-z)|
&\le |x-x'|\cdot|\nabla p^z(t-s,x-z+\tilde \T_{x,x',z}(x'-x))|\\
&\p |x-x'|\cdot \frac{t-s}{((t-s)^{{1}/{\A(z)}}+|x-z|)^{d+\A(z)+1}}\\
&\p |x-x'|\cdot \frac{t-s}{|x-z|^{d+\A_1}}\p |x-x'|\cdot \rho^{z,0}_1(t-s,x-z).
\end{align*}
where in the first inequality $|\tilde \T_{x,x'z}|\le 1$ is a
constant (which may depend on $x$, $x'$ and $z$), in the second inequality
we have used the fact that
$$|x-z+\tilde \T_{x,x',z}(x'-x)|\ge |x-z|-|x'-x|\ge {|x-z|}/{2},$$
and the fourth inequality follows from
$|z-x|>R_2>C(t-s)^{{1}/{\A(z)}}$.

As a result, we obtain that for all $x,x',z\in\R^d$ with $|x-x'|\le R_2$,
\begin{equation}\label{pppxy}
\begin{split}
\big|p^z(t,x-z)-p^z(t,x'-z)\big|
&\p
|x-x'|^{(\A(x)-\G)_+\wedge 1}\Big(\rho^{z,0}_{{2\G}/({3\A_2})}(t,x-z)
+\rho^{z,0}_{{2\G}/({3\A_2})}(t,x'-z)\Big).
\end{split}
\end{equation}
Then, using \eqref{p3-1-1} and \eqref{l3-3-1} and changing the constant $\G$ properly, we arrive at for every
$|x-x'|\le R_2$,
\begin{align*}
&\int_0^t\int_{\R^d}|p^z(t-s,x-z)-p^z(t-s,x'-z)||q(s,z,y)|\,dz\,ds\\
&\p |x-x'|^{(\A(x)-\G)_+\wedge 1}\!\cdot\Big[\Big(\rho^{y,0}_{({\B_0}/{\A_2})\!+({\G}/({2\A_2}))}+\rho^{y,\B_0}_{{\G/(2\A_2)}}\Big)(t,x-y)\!+
\Big(\rho^{y,0}_{({\B_0}/{\A_2})\!+({\G}/({2\A_2}))}+\rho^{y,\B_0}_{{\G/(2\A_2)}}\Big)(t,x'-y)\Big].
\end{align*}

Combining all the estimates together with \eqref{eq2-1} and using again \eqref{pppxy}, we can prove \eqref{t4-2-1} immediately.

(2) Now we assume that $\K(x,z)$ depends on $z$ and ${(\A_2/\A_1)}-1<{\B_0}/{\A_2}$.
Replacing $\A(x)$ by $\A(z)$ in \eqref{t4-2-1a}, we find that for every
$x,x'\in \R^d$,
\begin{align*}
&\big|p^z(t-s,x-z)-p^z(t-s,x'-z)\big|\\
&\p (|x-x'|^{(\A_1-\G)_+\wedge 1}\wedge 1)(t-s)^{-{(\A(z)-\G)_+}/{\A(z)}} \big(\rho^{z,0}_1(t-s,x-z)+\rho^{z,0}_1(t-s,x'-z)\big)\\
&\p (|x-x'|^{(\A_1-\G)_+\wedge 1}\wedge 1)
\big[\rho^{z,0}_{{\G}/{\A_2}}(t-s,x-z)+\rho^{z,0}_{{\G}/{\A_2}}(t-s,x'-z)\big].
\end{align*}
Combining this with \eqref{eq2-1}, \eqref{p3-1-2} and \eqref{l3-3-1}, and following the
same arguments above,  we arrive at for all $x,x'\in\R^d$ with $|x-x'|\le R_2$,
\begin{align*}
 \big|p(t,x,y)-p(t,x',y)\big|
&\p |x-x'|^{(\A_1-\G)_+\wedge 1}\\
&\quad\times
\Big[\big(\rho^{y,0}_{{\G}/{\A_2}}(t,x-y)+
\rho^{y,0}_{{\G}/{\A_2}}(t,x'-y)\big)\\
&\qquad\,\,+\int_0^t\int_{\R^d}
\big(\rho^{z,0}_{{\G}/{\A_2}}(t-s,x-z)+\rho^{z,0}_{{\G}/{\A_2}}(t-s,x'-z)\big)\\
&\qquad \qquad  \times \big(\rho^{y,0}_{({\B_0}/{\A_2})-[{(\A_2/\A_1)}-1+{\G}/({2\A_2})]}+
\rho^{y,\B_0}_{-[{(\A_2/\A_1)}-1+{\G}/({2\A_2})]}\Big)(s,z-y)\,dz\, ds \Big]\\
&\p |x-x'|^{(\A_1-\G)_+\wedge 1}
\big[\big(\rho^{y,0}_{\B_0/\A_2+\G_1}+\rho^{y,\B_0}_{\G_1}\big)(t,x-y)+
\big(\rho^{y,0}_{\B_0/\A_2+\G_1}+\rho^{y,\B_0}_{\G_1}\big)(t,x'-y)\big],
\end{align*}
where $\G_1:=1-\A_2/\A_1+\G/(2\A_2)$. By now we have verified \eqref{t4-2-1aa}. The proof is complete.
\end{proof}

Furthermore, we have the following gradient estimates for $p(t,x,y)$.

\begin{proposition}\label{t5-1}
\begin{itemize}
\item[(1)]
Suppose that $\K(x,z)$ is independent of $z$. Let $\B_0^*:=\B_0\wedge\A_2$.
If $\tilde \B_0(x_0):=({\A_1\B_0}/{\A(x_0)})\wedge\A_1>1-\A(x_0)$ for some $x_0 \in \R^d$, then for every fixed $t \in (0,1]$ and $y \in \R^d$,
$p(t,\cdot,y)$ is differentiable at $x=x_0$. Moreover,
for every $\G>0$,
\begin{equation}\label{t5-1-1}
|\nabla p(t,\cdot,y)(x_0)|\le c_1\rho^{y,0}_{1-({1/\A(x_0)})+(\B_0^*/\A_2)-
(\B_0^*/\A_1)-\G}(t,x_0-y)
\end{equation}
holds for some $c_1:=c_1(\A,\K,\G,x_0)>0$.

\item[(2)]  Suppose that $\K(x,z)$ depends on $z$ and ${(\A_2/\A_1)}-1<\B_0^{**}/\A_2$, where $\B_0^{**}:=\B_0\wedge (\A_2/2)$.
If $\tilde \B_0(x_0)
-({\A_1\A_2}/{\A(x_0)})(({\A_2}/{\A_1})-1)>1-\A_1$ for some
$x_0\in \R^d$ with $\tilde \B_0(x_0)$ defined in $(i)$,
then for every fixed $t \in (0,1]$ and $y \in \R^d$, $p(t,\cdot,y)$ is differentiable at
$x_0\in \R^d$. Furthermore, for every $\G>0$, there exists a constant
$c_2:=c_2(\A,\K,\G,x_0)>0$ such that
\begin{equation}\label{t5-1-2}
|\nabla p(t,\cdot,y)(x_0)|\le c_2\rho^{y,0}_{1-({1}/{\A_1})+(\B_0^*/\A_2)-(\B_0^*/\A_1)-\G_2 }(t,x_0-y),
\end{equation}
where $\G_2 :={(\A_2/\A_1)}-1+\G$.
\end{itemize}
\end{proposition}
\begin{proof}

(1) We first suppose that $\K(x,z)$ is independent of $z$. For simplicity, we assume
that $\B_0< \A(x_0)$, and so $\B_0^*=\B_0$ and
$\tilde \B_0(x_0)={\A_1\B_0}/{\A(x_0)}$. We will show that we can take the gradient with respect to variable $x$ in the equation \eqref{eq2-1}.
Note that
\begin{align*}
&\int_0^t \left|\int_{\R^d}\nabla p^z(t-s,x_0-z)q(s,z,y)\,dz\right|\,ds\\
&=\int_0^{t/2}\left|\int_{\R^d}\nabla p^z(t-s,x_0-z)q(s,z,y)\,dz\right|\,ds\\
&\quad+\int_{t/2}^t\left|\int_{\R^d}\nabla p^z(t-s,x_0-z)\left(q(s,z,y)-q(s,x_0,y)\right)\,dz\right|\,ds\\
&\quad+\int_{t/2}^t\left|\int_{\R^d}\left(\nabla p^z(t-s,x_0-z)-\nabla p^{x_0}(t-s,x_0-z)\right)
q(s,x_0,y)\,dz\right|\,ds\\
&=:\int_0^{t/2}\left|\int_{\R^d}J_1(s,z)\,dz\right|\,ds+\int_{t/2}^t
\left|\int_{\R^d}J_2(s,z)\,dz\right|\,ds+\int_{t/2}^t
\left|\int_{\R^d}J_3(s,z)\,dz\right|\,ds.
\end{align*}

According to \eqref{p1-1-1},
\begin{equation}\label{t5-1-3}
\begin{split}
|\nabla p^z(t,x-z)|&\p \frac{t}{(t^{{1}/{\A(z)}}+|x-z|)^{d+\A(z)+1}}.
\end{split}
\end{equation}
Observing that
\begin{equation*}
\Big|\frac{1}{\A(z)}-\frac{1}{\A(x)}\Big|\le
\frac{|\A(x)-\A(z)|}{\A_1^2}\le C|x-z|^{\B_0},
\end{equation*}
we find by \eqref{t5-1-3} that for every (small enough) $\G>0$, there is a constant $R_1:=R_1(\A, \K,\G)>0$ such that
\begin{equation}\label{t5-1-4}\begin{split}
|\nabla p^z(t,x-z)|&\p
\begin{cases}
\rho^{z,0}_{1-({1}/{\A(z)})}(t,x-z),\ &|x-z|\le R_1,\\
\frac{t}{|x-z|^{d+\A_1+1}}
,\ & |x-z|>R_1
\end{cases}\\
&\p \rho^{z,0}_{1-({1/\A(x)})-\G}(t,x-z). \end{split}
\end{equation}

Thus, choosing $\G>0$ small enough such that $\G<\B_0/\A_2$, and using \eqref{l3-2-0}, \eqref{p3-1-1} and
\eqref{t5-1-4}, we can find that
\begin{align*}
&\int_0^{t/2}\int_{\R^d}|J_1(s,z)|\,dz\,ds\\
&\p
\int_0^{t/2}\int_{\R^d}\rho^{z,0}_{1-({1/\A(x_0)})-\G}(t-s,x_0-z)
\big(
\rho^{y,0}_{({\B_0}/{\A_2})-\G}(s,z-y)+
\rho^{y,\B_0}_{-\G}(s,z-y)\big)\,dz\,ds\\
&\p \rho^{y,0}_0(t,x_0-y)\!\bigg(\!\int_0^{t/2}\!(t\!-\!s)^{-(1/\A(x_0))-\G}s^{(\B_0/\A_2)-\G}\! +\!
(t\!-\!s)^{1-(1/\A(x_0))-\G}s^{-1+(\B_0/\A_2)-\G}\,ds \bigg)\\
&\p \rho^{y,0}_{1-(1/\A(x_0))-2\G}(t,x-y).
\end{align*}

On the other hand, according to \eqref{p3-1-1a}, \eqref{t5-1-4} and the fact that $\B_0<\A(x_0)$, we arrive at
\begin{align*}
 \int_{t/2}^t\int_{\R^d}|J_2(s,z)|\,dz\,ds
&\p \int_{t/2}^t\int_{\R^d}
\bigg(\rho^{z,\wh \B_0}_{1-({1/\A(x_0)})-\G}(t-s,x_0-z)\tilde\rho^{y,0}_{\theta-\gamma+(\B_0/\A_2)-{(\B_0/\A_1)}}
(s,z-y)
\\
&\qquad\qquad
+\rho^{z,\wh \B_0}_{1-({1/\A(x_0)})-\G}(t-s,x_0-z)\tilde\rho^{y,0}_{\theta-\gamma+(\B_0/\A_2)-{(\B_0/\A_1)}}(s,x_0-y)\bigg)\,dz\,ds\\
&=:\int_{t/2}^t \int_{\R^d}(J_{21}(s,z)+J_{22}(s,z))\,dz\,ds,
\end{align*}
where $\wh \B_0=\wh \B_0(\theta):={\A_1(\B_0-\A_2\theta)}/{\A(x_0)}$.

Since $\tilde \B_0(x_0)>1-\A(x_0)$, we can choose
$\theta, \G>0$ small enough such that $1+(\wh \B_0/\A(x_0))-(1/\A(x_0))-\G>0$, which along with
\eqref{l3-2-0} yields
\begin{align*}
\int_{t/2}^t \int_{\R^d}J_{21}(s,z)\,dz\,ds
&\p \tilde \rho^{y,0}_0(t,x_0-y)\cdot\Big(
\int_{t/2}^t (t-s)^{1-(1/\A(x_0))-\G}s^{-1+\T-\G+(\B_0/\A_2)-(\B_0/\A_1)}\\
&\quad \quad +(t-s)^{(\wh \B_0/\A(x_0))-
(1/\A(x_0))-\G}s^{\T-\G+(\B_0/\A_2)-(\B_0/\A_1)}\,ds\Big)\\
&\p\tilde \rho^{y,0}_{1-(1/\A(x_0))+(\B_0/\A_2)-(\B_0/\A_1)-2\G}(t,x_0-y).
\end{align*}
On the other hand, noting that for $t/2<s\le t$, it holds that
\begin{equation}\label{t5-1-5}
\tilde \rho^{y,\B_0}_{\theta-\gamma+(\B_0/\A_2)-{(\B_0/\A_1)}}(s,x_0-y)\p
\tilde \rho^{y,\B_0}_{-\gamma+(\B_0/\A_2)-{(\B_0/\A_1)}}(t,x_0-y).
\end{equation}
Since $\tilde \B_0>1-\A(x_0)$, by \eqref{l3-1-1} we have
\begin{align*}
 \int_{t/2}^t \int_{\R^d}J_{22}(s,z)\,dz\,ds
&\p
\tilde \rho^{y,\B_0}_{-\gamma+(\B_0/\A_2)-(\B_0/\A_1)}(t,x_0-y)
\int_{t/2}^t\int_{\R^d}\rho^{z,\wh \B_0}_{1-({1/\A(x_0)})-\G}(t-s,x_0-z)\,dz\,ds\\
&\p \tilde \rho^{y,0}_{1+(\wh \B_0/\A(x_0))-(1/\A(x_0))+
(\B_0/\A_2)-(\B_0/\A_1)-2\G}(t,x_0-y).
\end{align*}

According to \eqref{l4-1-1-1} for the case that $|x_0-z|\le R_1$ and \eqref{t5-1-4} for the case
that $|x_0-z|>R_1$), we arrive at that (by changing the constant $\G$ properly) for any $t>0$ and $x_0,z\in \R^d$,
\begin{align*}
\left|\nabla p^{x_0}(t,x_0-z)-\nabla p^z(t,x_0-z)\right|\p
\tilde \rho^{z,\B_0}_{1-1/\A(x_0)-\G}(t,x_0-z).
\end{align*}
This, along with \eqref{p3-1-1} yields
\begin{align*}
\int_{t/2}^t\int_{\R^d}|J_3(s,z)|\,dz\,ds
&\p \tilde \rho^{y,0}_{-\G}(t,x_0-y)\int_{t/2}^t\int_{\R^d} \rho^{z,\B_0}_{1-(1/\A(x_0))-\G}(t-s,x_0-z)\,dz\,ds\\
&\p \tilde \rho^{y,0}_{1+(\B_0/\A(x_0))-1/\A(x_0)-2\G}(t,x_0-y),
\end{align*} where in the last inequality we used the fact that $\B_0\ge\tilde \B_0(x_0)>1-\A(x_0)$.

Combining all estimates together, we obtain that
\begin{equation*}
\int_0^t \left|\int_{\R^d}\nabla p^z(t-s,x_0-z)q(s,z,y)\,dz\right|\,ds
\p  \tilde \rho^{y,0}_{1-(1/\A(x_0))+(\B_0/\A_2)-(\B_0/\A_1)-2\G}(t,x_0-y).
\end{equation*}
According to the estimate above and the dominated convergence theorem,  we know that
$\nabla p(t,\cdot)(x_0)$ exists, and \eqref{t5-1-1} holds by changing $\T$ and $\G$ properly.

(2) Suppose that $\K(x,z)$ depends on $z$ and
${(\A_2/\A_1)}-1<{\B_0^{**}}/{\A_2}$. If for some $x_0\in \R^d$,
$\tilde \B_0(x_0)-({\A_1\A_2}/{\A(x_0)})
\big(({\A_2}/{\A_1})-1\big)>1-\A_1$, then, by \eqref{l4-1-1-1} and
\eqref{t5-1-3},
$$
|\nabla p^z(t,x-z)|\p t^{-{1}/{\A(z)}}\rho^{z,0}_1(t,x-z)\p
\rho^{z,0}_{1-(1/\A_1)}(t,x-z)
$$
and $$
|\nabla p^x(t,x-z)-\nabla p^z(t,x-z)|\p
\tilde \rho^{z,\B_0}_{1-(1/\A_1)-\G}(t,x-z).
$$
In the following, we choose $\theta,\gamma$ small enough such that
$$\wh \B_0:={\A_1}\big(\B_0-\A_2\theta\big)/{\A(x_0)}>-1+(1/\A_1)+\G.$$
Using estimates and following the same argument in part (1), we can obtain that
for every fixed $t \in (0,1]$ and $y \in \R^d$, $p(t,\cdot,y)$ is differentiable at $x_0\in \R^d$ and
\eqref{t5-1-2} holds.
\end{proof}
\begin{remark}
Propositions \ref{t4-2} and \ref{t5-1} show that,
the regularity of $p(t,\cdot,y)$ at
$x=x_0$ depends on the index $\A(x_0)$.
\end{remark}

\begin{proposition}\label{p4-5}
Suppose that $\K(x,z)$ is independent of $z$, or that $\K(x,z)$ depends on $z$ and ${(\A_2/\A_1)}-1<\B_0^{**}/\A_2$, where $\B_0^{**}:=\B_0\wedge (\A_2/2)$.
Then,
$p:(0,1]\times\R^d \times \R^d \rightarrow \R_+$ is continuous.
\end{proposition}
\begin{proof}
We only consider the case that $\K(x,z)$ is independent of $z$, and the other case can be verified similarly.

According to \eqref{eq2-1}, \eqref{p3-1-1} and \eqref{p3-2-1}, we know
immediately the continuity of $p$ with respect to time variable.
It remains to show the continuity of $p$ with respect to space variables.

For $x,y,z\in \R^d$ and any $\varepsilon>0$
\begin{align*}
|p^{y+\varepsilon z}(t,x-y-\varepsilon)-
p^{y}(t,x-y)|
&\le  |p^{y+\varepsilon z}(t,x-y-\varepsilon)-
p^{y+\varepsilon z}(t,x-y)|+
|p^{y+\varepsilon z}(t,x-y)-p^{y}(t,x-y)|.
\end{align*}
According to \eqref{p1-1-3} and \eqref{l4-1-1}, we can show that
$(x,y)\mapsto p^y(t,x-y)$ is continuous. On the other hand,  \eqref{p1-1-4}, \eqref{p3-1-1a-1} and \eqref{l4-1-1a} imply that
$(x,y)\mapsto q_0(t,x,y)$ is continuous. By the iteration estimates in Proposition \ref{p3-1} and the
dominated convergence theorem, we also can verify that $(x,y)\mapsto q(t,x,y)$ is continuous.
Due to the expression \eqref{eq2-1} and again the dominated convergence theorem,  we know that
$(x,y)\mapsto p(t,x,y)$ is continuous. The proof is finished.
\end{proof}

\medskip

\section{Existence and Uniqueness of the Solution to \eqref{t1-1-1}}\label{section5}
\subsection{Existence}
The purpose of this subsection is to prove rigorously that $p(t,x,y)$ defined by \eqref{eq2-1}
satisfies \eqref{t1-1-1}. First, as a direct consequence of Lemma \ref{l4-1}, we have the following statement.

\begin{lemma}\label{p4-1}  Let $\B_0^*\in(0,\B_0]\cap(0,{\alpha_2}).$  Then, for any $0<\theta<\B_0^*\A_1/\A_2$ and $0<\G<(\theta/\A_2)\wedge ((\B_0^*/\A_2-\theta/\A_1)/2)$, $t\in (0,1]$ and $x\in\R^d$,
\begin{equation}\label{p4-1-1-e}
\bigg|\int_{\R^d}p^y(t,x-y)\,dy-1\bigg|\p t+t^{(\B_0^*/{\A_2})-2\G-({\T}/{\A_1})}.
\end{equation}
In particular,
\begin{equation}\label{p4-1-1}
\lim_{t \downarrow 0}\sup_{x \in \R^d}\bigg|\int_{\R^d}p^y(t,x-y)\,dy-1\bigg|=0.
\end{equation}
\end{lemma}
\begin{proof} Throughout the proof, we will assume that $\B_0<\alpha_2$ and $\B_0^*=\B_0$ for simplicity.
Noting that $\int_{\R^d}p^x(t,x-y)\,dy=1$ for all $x\in \R^d$ and $t\in (0,1]$, we have
\begin{align*}
\bigg|\int_{\R^d}p^y(t,x-y)\,dy-1\bigg|&=\bigg|
\int_{\R^d}p^y(t,x-y)\,dy-\int_{\R^d}p^x(t,x-y)\,dy\bigg|\\
&\le \int_{\R^d}|p^x(t,x-y)-p^y(t,x-y)|\,dy\\
&\le \int_{\{|x-y|\le R_1\}}|p^x(t,x-y)-p^y(t,x-y)|\,dy\\
&\quad +\int_{\{|x-y|>R_1\}}|p^x(t,x-y)-p^y(t,x-y)|\,dy\\
&=:J_1+J_2,
\end{align*}
where $R_1$ is the constant in Lemma \ref{l4-1}.

On the one hand, \eqref{l4-1-1} yields that
\begin{align*}
J_1& \p \int_{\R^d} \tilde \rho^{x,\B_0}_{1-2\G-({\T}/{\A_1})}(t,x-y)\,dy\p t^{(\B_0/{\A_2})-2\G-({\T}/{\A_1})},
\end{align*}
where the last inequality follows from the fact that \eqref{l3-1-1-00} holds for $\tilde \rho$.

On the other hand,
\begin{align*}
J_2& \p \int_{\{|x-y|>R_1\}}\big(\rho^{x,0}_1(t,x-y)+\rho^{y,0}_1(t,x-y)\big)\,dy\\
&\p \int_{\{|x-y|>R_1\}}\bigg(\frac{t}{|x-y|^{d+\A(x)}}+\frac{t}{|x-y|^{d+\A(y)}}\bigg)\,dy\p t \int_{\{|x-y|>R_1\}}\frac{1}{|x-y|^{d+\A_1}}\,dy \p t.
\end{align*}
Combining with both estimates for $J_1$ and $J_2$, we prove \eqref{p4-1-1-e}.

Choosing $\G,\T>0$ small enough such that $2\G+({\T}/{\A_1})<{\B_0}/{\A_2}$, we immediately get \eqref{p4-1-1}.
\end{proof}
\begin{lemma}\label{l4-2}
For every $0<\varepsilon, \theta<\A_1$ and $0<\gamma<({\theta}/{\A_2})\wedge ((1-({\theta}/{\A_1}))/2)$, there are constants $R_1:=R_1(\A,\K,\G,\T)\in (0,1)$ and $c_1:=c_1(\A,\K,\G,\T,\e)>0$ such that
for every $t \in (0,1]$ and $x,y\in \R^d$ with $|x-y|\le R_1$,
\begin{equation}\label{l4-2-1}
\begin{split}
&\int_{\R^d}\big|\delta_{p^x}(t,x-y;z)-\delta_{p^y}(t,x-y;z)\big|\cdot\frac{1}{|z|^{d+\A(x)}}\,dz
\le c_1\tilde \rho^{x,\B_0}_{-2\G-({\T}/{\A_1})}(t,x-y),
\end{split}
\end{equation}
where $\tilde \rho$ is defined by \eqref{p3-1-7}.
\end{lemma}
\begin{proof}
By using \eqref{l4-1-2}, we can verify that
\begin{align*}
& \delta_{p^x}(t,x-y;z)-\delta_{p^y}(t,x-y;z)\\
&=\int_0^{{t}/{2}}\bigg(\int_{\R^d}\big( \LL^x- \LL^y\big)p^x(s,w)
\big(\delta_{p^y}(t-s,x-y-w;z)-\delta_{p^y}(t-s,x-y;z)\big)\,dw\bigg)\,ds\\
&\quad+\int_{{t}/{2}}^t\!\!\bigg(\int_{\R^d}\big( \LL^x- \LL^y\big)p^y(t-s,x-y-w)
\big(\delta_{p^x}(s,w;z)-\delta_{p^x}(s,x-y;z)\big)\,dw\bigg)\,ds.
\end{align*}
Therefore,
\begin{align*}
& \int_{\R^d}\big|\delta_{p^x}(t,x-y;z)-\delta_{p^y}(t,x-y;z)\big|\cdot\frac{1}{|z|^{d+\A(x)}}\,dz\\
& \le \int_0^{{t}/{2}}\bigg[\int_{\R^d}\big|\big( \LL^x- \LL^y\big)p^x(s,w)\big|
\bigg(\int_{\R^d}\big|\delta_{p^y}(t-s,x-y-w;z)-\delta_{p^y}(t-s,x-y;z)\big|\cdot \frac{1}{|z|^{d+\A(x)}}\,dz\bigg)\,dw\bigg]\,ds\\
&+\int_{{t}/{2}}^t\bigg[\int_{\R^d}\big|\big( \LL^x- \LL^y\big)p^y(t-s,x-y-w)\big|
\bigg(\int_{\R^d}\big|\delta_{p^x}(s,w;z)-\delta_{p^x}(s,x-y;z)\big|\cdot \frac{1}{|z|^{d+\A(x)}}\,dz\bigg)\,dw\bigg]\,ds\\
&=:J_1+J_2.
\end{align*}

According to \eqref{p1-1-4} and \eqref{l4-1-5}, for every $\T \in (0,1)$ and $\G>0$, there exists a constant
$R_1:=R_1(\A,\K,\G)>0$ such that for every $0<s<t<1$ and $x,y\in \R^d$ with $|x-y|\le R_1$,
\begin{equation}\label{l4-2-2}
\begin{split}
&\big|\delta_{p^y}(t-s,x-y-w;z)-\delta_{p^y}(t-s,x-y;z)\big|\\
&\p \big[\left((t-s)^{-{1/\A(y)}}|w|\right)^{\T}\wedge 1\big]\cdot
\left[\left((t-s)^{-{2}/{\A(y)}}|z|^2\right)\wedge 1\right]
\\
&\quad\times\big[
\rho_1^{y,0}(t-s,x-y)+\rho_1^{y,0}(t-s,x-y-w)
+\rho_1^{y,0}(t-s,x-y\pm z)+\rho_1^{y,0}(t-s,x-y-w\pm z)\big]\\
&\p \left(|w|^{\T}\wedge 1\right)\cdot
\left[\left((t-s)^{-{2}/{\A(x)}}|z|^2\right)\wedge 1\right]
 \cdot\big[
\rho_{1-\G-({\T}/{\A_1})}^{x,0}(t-s,x-y)+\rho_{1-\G-({\T}/{\A_1})}^{x,0}(t-s,x-y-w)\\
&\qquad\qquad\qquad\quad
+\rho_{1-\G-({\T}/{\A_1})}^{x,0}(t-s,x-y\pm z)+\rho_{1-\G-({\T}/{\A_1})}^{x,0}(t-s,x-y-w\pm z)\big].
\end{split}
\end{equation}
Furthermore, following the argument of Lemma \ref{l1-1}, we can derive that for every $t\in (0,1]$ and $x,y \in \R^d$
\begin{equation}\label{l4-2-3}
\begin{split}
&\int_{\R^d} \left[\left(t^{-{2}/{\A(x)}}|z|^2\right)\wedge 1\right]
\cdot \rho_{1}^{x,0}(t,y \pm z)\cdot |z|^{-d-\A(x)}\,dz
\p \rho^{x,0}_{0}(t,y),\\
&\int_{\R^d} \left[\left(t^{-{2}/{\A(x)}}|z|^2\right)\wedge 1\right]
\cdot \rho_{1}^{x,0}(t,y)\cdot |z|^{-d-\A(x)}\,dz
\p \rho^{x,0}_{0}(t,y).
\end{split}
\end{equation}
Combining \eqref{l4-1-4} with \eqref{l4-2-2} and \eqref{l4-2-3}, we find that for  all $t\in(0,1]$ and  $x,y\in\R^d$ with $|x-y|\le R_1$,
\begin{align*}
J_1 &\p |x-y|^{\B_0}\cdot \bigg[
\int_0^{{t}/{2}}\int_{\R^d}\tilde \rho^{x,\T}_{-\G}(s,w)\cdot
\tilde \rho^{x,0}_{-\G-({\T}/{\A_1})}(t-s,x-y-w)\,dw\,ds\\
&\qquad\qquad\qquad+\int_0^{{t}/{2}}\int_{\R^d}\tilde \rho^{x,\T}_{-\G}(s,w)\cdot
\tilde \rho^{x,0}_{-\G-({\T}/{\A_1})}(t-s,x-y)\,dw\,ds\bigg]\\
&=:|x-y|^{\B_0}\cdot(J_{11}+J_{12}).
\end{align*}
Since \eqref{l3-2-1-} holds true for $\tilde \rho$ as mentioned in Remark \ref{r3-3}(2),
for every $0<\theta<\A_1$ and $0<\gamma<({\theta}/{\A_2})\wedge ((1-({\theta}/{\A_1}))/2)$, we have
\begin{equation}\label{l4-2-3a}
\begin{split}
J_{11}&\p \tilde \rho^{x,0}_0(t,x-y) \int_0^{{t}/{2}}\big(s^{-1-\G+({\T}/{\A_2})}(t-s)^{-\G-({\T}/{\A_1})}
+s^{-\G}(t-s)^{-1-\G-({\T}/{\A_1})}\big)\,ds\\
&\p \tilde \rho^{x,0}_0(t,x-y) \bigg(t^{-\G-({\T}/{\A_1})}\int_0^{{t}/{2}}s^{-1-\G+({\T}/{\A_2})}\,ds+
t^{-1-\G-({\T}/{\A_1})}\int_0^{{t}/{2}} s^{-\G}\,ds\bigg)\\
& \p  \tilde \rho^{x,0}_{-2\G-({\T}/{\A_1})}(t,x-y).
\end{split}
\end{equation}
On the other hand, note that for every $0<s<{t}/{2}$, $$\tilde \rho^{x,0}_{-\G-({\T}/{\A_1})}(t-s,x-y)\p \tilde \rho^{x,0}_{-\G-({\T}/{\A_1})}(t,x-y).$$
Then, it holds that
\begin{align*}
J_{12}&\p  \tilde \rho^{x,0}_{-\G-({\T}/{\A_1})}(t,x-y)\cdot
\int_0^{{t}/{2}}\int_{\R^d}\tilde \rho^{x,\T}_{-\G}(s,w)\,dw\,ds\p \tilde \rho^{x,0}_{-2\G-({\T}/{\A_1})}(t,x-y),
\end{align*} where \eqref{l3-1-1-00} was used in the last inequality.
Hence, we find that for every $t\in (0,1]$ and $x,y \in \R^d$ with $|x-y|\le R_1$,
\begin{equation*}
J_1 \p \tilde \rho^{x,\B_0}_{-2\G-({\T}/{\A_1})}(t,x-y).
\end{equation*}

By estimates for
 the terms
$$\big|\big( \LL^x- \LL^y\big)p^y(t-s,x-y-w)\big|$$ and $$\big|
\delta_{p^x}(s,w;z)-\delta_{p^x}(s,x-y;z)\big|,$$ we can deal with the singularity
as $s$ near $t$ by the same arguments above and obtain that for all $t\in (0,1]$ and $x,y\in \R^d$ with $|x-y|\le R_1$,
\begin{equation*}
J_2 \p \tilde \rho^{x,\B_0}_{-2\G-({\T}/{\A_1})}(t,x-y).
\end{equation*}

Combining both of the estimates for $J_1$ and $J_2$, we finally obtain that \eqref{l4-2-1} holds true.
\end{proof}
The following result is a consequence of Lemma \ref{l4-2}.
\begin{lemma}\label{p4-2} Let $\B_0^*\in(0,\B_0]\cap(0,{\alpha_2}).$ Then, we have the following two statements.
\begin{itemize}
\item[(1)] If $\K(x,z)$ is independent of $z$, then for every positive
constants $\gamma,\theta \in (0,1)$ such that $\gamma<{\theta}/{\A_2}$ and
$2\gamma+({\theta}/{\A_1})<\B_0^*/\A_2$, there exists a constant $c_1:=c_1(\A,\K,\G,\T)>0$ such that for all $t \in (0,1]$ and $x\in \R^d$,
\begin{equation}\label{p4-2-1}
\Big|\int_{\R^d} \LL^xp^y(t,x-y)\,dy\Big|\le c_1t^{-1+(\B_0^*/\A_2)-2\G-({\theta}/{\A_1})}.
\end{equation}

\item[(2)] If $\K(x,z)$ depends on $z$, then for every positive
constants $\gamma,\theta \in (0,1)$ such that $\gamma<{\theta}/{\A_2}$ and
$2\gamma+({\theta}/{\A_1})<\B_0^*/\A_2$, the estimate
\eqref{p4-2-1} will be replaced by
\begin{equation}\label{p4-2-1a}
\Big|\int_{\R^d} \LL^xp^y(t,x-y)\,dy\Big|\le c_1\big(
t^{-1+(\B_0^*/\A_2)-2\G-({\theta}/{\A_1})}+t^{1-{(\A_2/\A_1)}}\big).
\end{equation}
\end{itemize}
\end{lemma}
\begin{proof} We assume that $\B_0<\alpha_2$ and $\B_0^*=\B_0$ for simplicity.
Observe that
\begin{align*}
 \int_{\R^d}  \LL^xp^y(t,x-y)\,dy
&=\int_{\R^d}\big( \LL^xp^y(t,x-y)- \LL^xp^x(t,x-y)\big)\,dy\\
&=\int_{\R^d}
\int_{\R^d}\big(\delta_{p^y}(t,x-y;z)
-\delta_{p^x}(t,x-y;z)\big)\cdot\frac{\K(x,z)}{|z|^{d+\A(x)}}\,dz\,dy.
\end{align*}
Therefore, for all $t\in (0,1]$ and $x\in \R^d$,
\begin{align*}
 \Big|\int_{\R^d}  \LL^xp^y(t,x-y)\,dy\Big|
&\p \int_{\{|y-x|\le R_1\}}\int_{\R^d}
|\delta_{p^y}(t,x-y;z)
-\delta_{p^x}(t,x-y;z)|\cdot \frac{1}{|z|^{d+\A(x)}}\,dz\,dy\\
&\quad+\int_{\{|y-x|>R_1\}}\int_{\R^d}|\delta_{p^y}(t,x-y;z)
-\delta_{p^x}(t,x-y;z)|\cdot \frac{1}{|z|^{d+\A(x)}}\,dz\,dy\\
&=:J_1+J_2,
\end{align*}
where $R_1$ is the constant in Lemma \ref{l4-2}.
According to \eqref{l4-2-1}, we have for all $t\in (0,1]$ and $x\in \R^d$,
\begin{align*}
J_1& \p \int_{\{|y-x|\le R_1\}}\tilde \rho^{x,\B_0}_{-2\G-({\T}/{\A_1})}(t,x-y)\,dy
\p t^{-1+(\B_0^*/\A_2)-2\G-({\theta}/{\A_1})},
\end{align*}
where we used the fact that \eqref{l3-1-1-00} holds for $\tilde \rho$.

(1) Suppose that $\K(x,z)$ is independent of $z$. Then, according to the argument for $J_1$ in  \eqref{l2-1-3a}, for every $t \in (0,1]$ and $x,y\in \R^d$ with $|x-y|>R_1$,
\begin{equation}\label{p4-2-2} \begin{split}
&\int_{\R^d}|\delta_{p^y}(t,x-y;z)|\cdot \frac{1}{|z|^{d+\A(x)}}\,dz\p \frac{1}{|x-y|^{d+\A_1}},\\
& \int_{\R^d}|\delta_{p^x}(t,x-y;z)|\cdot \frac{1}{|z|^{d+\A(x)}}\,dz\p \frac{1}{|x-y|^{d+\A_1}},\end{split}
\end{equation}
which imply that for all $t\in (0,1]$ and $x\in \R^d$,
\begin{align*}
J_2&\p \int_{\{|y-x|>R_2\}}\int_{\R^d}
|\delta_{p^y}(t,x-y;z)|\cdot \frac{1}{|z|^{d+\A(x)}}\,dz\,dy +\int_{\{|y-x|>R_2\}}\int_{\R^d}
|\delta_{p^x}(t,x-y;z)|\cdot \frac{1}{|z|^{d+\A(x)}}\,dz\,dy\\
&\le C_1.
\end{align*}
Combining both the estimates for $J_1$ and $J_2$, we prove \eqref{p4-2-1}.

(2) When $\K(x,z)$ depends on $z$, by the argument of \eqref{l2-1-5}, the first inequality in \eqref{p4-2-2} will be changed into
\begin{equation*}
\int_{\R^d}|\delta_{p^y}(t,x-y;z)|\cdot \frac{1}{|z|^{d+\A(x)}}\,dz\p
\frac{t^{1-{(\A_2/\A_1)}}}{|x-y|^{d+\A_1}}.
\end{equation*}
Using this inequality and following the same line as above, we will obtain \eqref{p4-2-1a}.
\end{proof}

 The statement below is the main result in this subsection.

\begin{proposition}\label{t4-1}
The following two statements hold.
\begin{itemize}
\item [(1)] If $\K(x,z)$ is independent of $z$, then $p(t,x,y)$ defined by \eqref{eq2-1} satisfies the equation \eqref{t1-1-1}
pointwise.

\item[(2)] If $\K(x,z)$ depends on $z$ and ${(\A_2/\A_1)}-1<\min\{{\B_0}/{\A_2},{1}/{2}\}$, then
  $p(t,x,y)$ also satisfies the equation \eqref{t1-1-1}
	pointwise.
  \end{itemize}
\end{proposition}
\begin{proof}
For simplicity, we assume that $\B_0< {\A_2}/{2}$ and $\B_0^*=\B_0$.

{(1)} We first assume that $\K(x,z)$ is independent of $z$. The proof is split into four parts.

{(i)} For every $0<s<t<1$ and $x,y \in \R^d$, define
\begin{equation*}
\phi(t,s,x,y):=\int_{\R^d}p^z(t-s,x-z)q(s,z,y)\,dz,
\end{equation*}
where $q(t,x,y)$ is constructed in Proposition \ref{p3-1}.
By \eqref{eq2-1}, it holds that
\begin{equation}\label{t4-1-1}
p(t,x,y)=p^y(t,x-y)+\int_0^t \phi(t,s,x,y)\,ds.
\end{equation}

Note that for every $t \in (0,1]$, $x,y\in \R^d$ and $\e>0$,
\begin{align*}
 \frac{1}{\e}\int_t^{t+\e}\phi(t+\e,s,x,y)\,ds-q(t,x,y)
&= \frac{1}{\e}\int_t^{t+\e}\int_{\R^d}p^z(t+\e-s,x-z)q(s,z,y)\,dz\,ds-q(t,x,y)\\
&=\frac{1}{\e}\int_t^{t+\e}\int_{\R^d}p^z(t+\e-s,x-z)\big(
q(s,z,y)-q(s,x,y)\big)\,dz\,ds\\
&\quad +\frac{1}{\e}\int_t^{t+\e}\bigg(
\int_{\R^d} p^z(t+\e-s,x-z)\,dz\bigg)\cdot \big(q(s,x,y)-q(t,x,y)\big)\,ds\\
&\quad +q(t,x,y)\cdot\frac{1}{\e}\int_t^{t+\e}\bigg(\int_{\R^d} p^z(t+\e-s,x-z)\,dz-1\bigg)\,ds\\
&=:J_1(\e)+J_2(\e)+J_3(\e).
\end{align*}

First, by \eqref{p3-1-1a}, for every $t>0$, $x,y \in \R^d$ and $\sigma>0$, there exists
a constant $\varepsilon_1:=\varepsilon_1(t,x,y, \sigma)>0$ such that for all $z\in \R^d$ with $|z-x|\le \varepsilon_1$,
\begin{equation}\label{t4-1-2}
|q(s,z,y)-q(s,x,y)|\le \sigma
\end{equation} holds for all $t<s<t+\e$ and $0<\e<1$.
Let
\begin{align*}
|J_1(\e)|&\p \frac{1}{\e}\int_t^{t+\e}\int_{\{|z-x|\le \varepsilon_1\}}p^z(t+\e-s,x-z)\big|
q(s,z,y)-q(s,x,y)\big|\,dz\,ds\\
&\quad +\frac{1}{\e}\int_t^{t+\e}\int_{\{|z-x|>\varepsilon_1\}}p^z(t+\e-s,x-z)\big|
q(s,z,y)-q(s,x,y)\big|\,dz\,ds\\
&=:J_{11}(\e)+J_{12}(\e).
\end{align*}
According to \eqref{t4-1-2} and \eqref{l3-1-1}, we know immediately that
\begin{align*}
J_{11}(\e)&\p \sigma\cdot\frac{1}{\e}\int_t^{t+\e}\int_{\R^d}\rho^{z,0}_{1}
(t+\e-s,x-z)\,dz\,ds\p \frac{\sigma}{\e}\cdot \int_t^{t+\e}\,ds \le C_1 \sigma,
\end{align*}
where $C_1>0$ is a constant independent of $\e$ and $\sigma$.
At the same time, it holds that for any $\G<{\B_0}/{\A_2},$
\begin{align*}
|J_{12}(\e)|&\p \frac{1}{\e}\cdot\bigg(
\int_t^{t+\e}\big(t+\e-s\big) \int_{\{|z-x|>\varepsilon_1\}}
\frac{|q(s,z,y)|+|q(s,x,y)|}{|x-z|^{d+\A_1}}\,dz\,ds\bigg)\\
&\p \varepsilon_1^{-d-\A_1}\int_t^{t+\e}
\int_{\R^d}\big(\rho^{y,0}_{({\B_0}/{\A_2})-\G}+\rho^{y,\B_0}_{-\G}\big)(s,z-y)\,dz\,ds\\
&\quad +\int_t^{t+\e}\int_{\{|z-x|>\varepsilon_1\}}
\frac{\big(\rho^{y,0}_{({\B_0}/{\A_2})-\G}+\rho^{y,\B_0}_{-\G}\big)(s,x-y)}{|z-x|^{d+\A_1}}\,dz\,ds\\
&\p \varepsilon_1^{-d-\A_1}\bigg(
\int_t^{t+\e}s^{-1+({\B_0}/{\A_2})-\G}\,ds+\varepsilon_1^d t^{-1-({d}/{\A_1})-\G} \int_t^{t+\e}\,ds\bigg)\\
&\le C(\varepsilon_1)t^{-1-({d}/{\A_1})-\G}\e,
\end{align*}
where in the first inequality we used the fact that $p^z(t,x)\p \frac{t}{|x|^{d+\A_1}}$ for every
$t \in (0,1]$ and $x\in \R^d$ with $|x|>\varepsilon_1$, the second inequality follows from \eqref{p3-1-1}, and in the third inequality we used \eqref{l3-1-1-00} and the fact
that $\rho^{y,0}_{-\G}(s,x)\p t^{-1-({d}/{\A_1})-\G}$ for all $t<s<t+\e$ and $x\in \R^d$.

Second, note that for every fixed $x,y \in \R^d$,
$q_0(\cdot,x,y)$ is continuous in $(0,1]$. Then, by \eqref{eq2-2}, \eqref{p3-1-1} and the dominated convergence theorem,
we know that $q(\cdot,x,y)$ is continuous in $(0,1]$. Thus, for every $t \in (0,1]$, $x,y \in \R^d$ and $\sigma>0$, there exists a constant $\e_2:=\e_2(t,x,y,\sigma)>0$ such that for all $|s-t|<\e_2$,
\begin{equation*}
|q(s,x,y)-q(t,x,y)|\le \sigma,
\end{equation*}
from which we have that when $\e<\e_2$,
\begin{align*}
|J_2(\e)|&\p \frac{\sigma}{\e}\cdot
\Big(\int_t^{t+\e}\Big|\int_{\R^d}p^z(t+\e-s,x-z)\,dz\Big|\,ds\Big)\le C_2\sigma,
\end{align*}
where in the last inequality we have used Lemma \ref{p4-1}, and
$C_2>0$ is a constant independent of $\e$ and $\sigma$.

Third, according to Lemma \ref{p4-1}, we arrive at
\begin{align*}
\lim_{\e \downarrow 0}|J_3(\e)|&\le
|q(t,x,y)|\cdot \bigg(\lim_{\e \downarrow 0}\frac{1}{\e}
\cdot \int_t^{t+\e}\Big|\int_{\R^d}p^z(t+\e-s,x-z)\,dz-1\Big|\,ds\bigg)\\
&\le |q(t,x,y)|\cdot\bigg(\lim_{\e \downarrow 0}\sup_{s \in (0,\e)}
\Big|\int_{\R^d}p^z(s,x-z)\,dz-1\Big|\bigg)=0.
\end{align*}

Combining all the estimates together, we arrive at for every $\sigma>0$,
\begin{align*}
\lim_{\e \downarrow 0}\Big|\frac{1}{\e}\int_t^{t+\e}\phi(t+\e,s,x,y)\,ds-q(t,x,y)\Big|\le C_3\sigma,
\end{align*}
where $C_3$ is independent of $\e$ and $\sigma$. Since $\sigma$ is arbitrary, we finally obtain that
\begin{equation}\label{t4-1-3}
\begin{split}
\lim_{\e \downarrow 0}\Big|\frac{1}{\e}\int_t^{t+\e}\phi(t+\e,s,x,y)\,ds-q(t,x,y)\Big|=0.
\end{split}
\end{equation}

{(ii)} By the proof of \eqref{l4-2-3} (also see \cite[Theorem 2.4]{CZ}), it holds that
\begin{equation}\label{t4-1-4}
| \LL^zp^z(t-s,x-z)|\p \rho^{z,0}_0(t-s,x-z),
\end{equation} and so we can verify that for every $0<s<t<1$ and $x,y\in \R^d$,
$$\int_{\R^d}| \LL^z p^z(t-s,x-z)||q(s,z,y)|\,dz<\infty.$$
Thus, by the dominated convergence theorem, for every $0<s<t$ and $x,y \in \R^d$,
\begin{equation}\label{e:lkk}
\begin{split}
\partial_t \phi(t,s,x,y)&=\int_{\R^d} \LL^zp^z(t-s,x-z)q(s,z,y)\,dz\\
&=\int_{\R^d} \LL^zp^z(t-s,x-z)(q(s,z,y)-q(s,x,y))\,dz\\
&\quad +q(s,x,y)\int_{\R^d}( \LL^z- \LL^x)p^z(t-s,x-z)\,dz\\
&\quad +q(s,x,y)\int_{\R^d} \LL^x p^z(t-s,x-z)\,dz\\
&=:\tilde J_1+ \tilde J_2+ \tilde J_3.
\end{split}
\end{equation}

Combining \eqref{t4-1-4} with \eqref{p3-1-1a} yields that for any $0<\G<\T<{\B_0}/{\A_2}$,
\begin{align*}
|\tilde J_1| &\p \int_{\R^d}\tilde \rho^{z,\tilde \T}_0(t-s,x-z)\cdot
\tilde \rho^{y,0}_{\T-\G+({\B_0}/{\A_2})-({\B_0}/{\A_1})}(s,z-y)\,dz\\
&\quad +\int_{\R^d}\tilde \rho^{z,\tilde \T}_0(t-s,x-z)\cdot
\tilde \rho^{y,0}_{\T-\G+({\B_0}/{\A_2})-({\B_0}/{\A_1})}(s,x-y)\,dz\\
&=:\tilde J_{11}+ \tilde J_{12},
\end{align*}
where $\tilde \T:=\A_1(\B_0-\A_2\theta)/{\A_2}$.
According to \eqref{l3-2-1} (which holds for $\tilde \rho$), we obtain that
for every ${t}/{2}\le s \le t$,
\begin{align*}
\tilde J_{11}&\p \tilde \rho^{y,0}_0(t,x-y)\cdot\big[(t-s)^{-1+({\tilde \T}/{\A_2})}s^{\T-\G+(\B_0/\A_2)-({\B_0}/{\A_1})-1}\big]\p t^{-\G-({\B_0}/{\A_1})-({d}/{\A_1})-2}(t-s)^{-1+({\tilde \T}/{\A_2})},
\end{align*}
where in the last inequality we have used the fact that $\tilde\rho^{y,0}_0(t,x-y)\p t^{-1-d/\A_1}$ and
$s^{-1}\p t^{-1}$ for ${t}/{2}\le s \le t$. Note that $\tilde \rho^{y,\B_0}_{\T-\G-({\B_0}/{\A_1})}(s,x-y)\p t^{-\G-({\B_0}/{\A_1})-({d}/{\A_1})-1}$ for
every ${t}/{2}\le s \le t$,
and \eqref{l3-1-1} holds for $\tilde \rho$, then we have
\begin{align*}
\tilde J_{12}& \p t^{-\G-({\B_0}/{\A_1})-({d}/{\A_1})-1}\int_{\R^d}\tilde \rho^{z,\tilde \T}_{0}(t-s,x-z)\,dz\p t^{-\G-({\B_0}/{\A_1})-({d}/{\A_1})-1}(t-s)^{-1+({\tilde \T}/{\A_2})}.
\end{align*}

By \eqref{l2-1-2} and \eqref{p3-1-1}, we arrive at that for any $0<\G<\T<{\B_0}/{\A_2}$ and every
${t}/{2}\le s \le t$,
\begin{equation*}
|( \LL^z- \LL^x)p^z(t-s,x-z)|\p
|q_0(t-s,x, z)|\p \rho^{z,\B_0}_{-\G}(t-s,x-z)
\end{equation*} and \begin{equation*}
|q(s,x,y)|\p \big(\rho^{y,0}_{({\B_0}/{\A_2})-\G}+\rho^{y,\B_0}_{-\G}\big)(s,x-y)\p t^{-\G-({d}/{\A_1})-1}.
\end{equation*}
Hence, combining both estimates above with \eqref{l3-1-1},  we obtain
\begin{align*}
|\tilde J_{2}|&\p t^{-\G-({d}/{\A_1})-1}\int_{\R^d} \rho^{z,\B_0}_{-\G}(t-s,x-z)\,dz
 \p t^{-\G-({d}/{\A_1})-1}(t-s)^{-1+({\B_0}/{\A_2})-\G}.
\end{align*}
Furthermore, \eqref{p3-1-1} and \eqref{p4-2-1} yield that for every ${t}/{2}\le s \le t$ and
constants $\gamma,\theta \in (0,1)$ such that $\gamma<{\theta}/{\A_2}$ and
$2\gamma+({\theta}/{\A_1})<\B_0/\A_2$,
\begin{equation*}
|\tilde J_3| \p t^{-\G-({d}/{\A_1})-1}(t-s)^{-1+({\B_0}/{\A_2})-2\G-({\theta}/{\A_1})}.
\end{equation*}

On the other hand, when $0<s\le {t}/{2}$,  it follows from \eqref{t4-1-4}, \eqref{p3-1-1} and \eqref{l3-2-1} that
\begin{align*}
\big|\partial_t \phi(t,s,x,y)\big|
 &\p
\int_{\R^d}\big| \LL^z p^z(t-s,x-z)\big|\big|q(s,z,y)\big|\,dz\\
&\p \int_{\R^d} \rho^{z,0}_0(t-s,x-z)\cdot
\big(\rho^{y,0}_{({\B_0}/{\A_2})-\G}+\rho^{y,\B_0}_{-\G}\big)(s,z-y)\,dz\\
&\p \rho^{y,0}_0(t,x-y)\cdot\big[(t-s)^{-1}s^{({\B_0}/{\A_2})-\G}+
(t-s)^{-1+({\B_0}/{\A_2})}s^{-\G}+s^{-1+({\B_0}/{\A_2})-\G}\big]\\
&\quad +\rho^{y,\B_0}_{0}(t,x-y)\cdot(t-s)^{-1}s^{-\G}\\
&\p t^{-2-({d}/{\A_1})}s^{-1+({\B_0}/{\A_2})-\G},
\end{align*}
where in the last inequality we have used the facts that $\rho^{y,0}_0(t,x-y)\p t^{-1-({d}/{\A_1})}$ and $(t-s)^{-1}\p t^{-1}$ for every
$0<s\le {t}/{2}$.
Therefore, choosing $\gamma,\theta \in (0,1)$ such that $\gamma<{\theta}/{\A_2}$ and
$2\gamma+({\theta}/{\A_1})<\B_0/\A_2$, and combining all the estimates above,  we
know that \eqref{e:lkk} is well defined, and that for every $t \in (0,1)$,
$
\sup_{t_0 \in (t,t+\e_1)}|
\partial_t\phi(t_0,s,x,y)|\le \eta(t,s),
$
where $\eta$ is a non-negative measurable function such that
$\int_0^t \eta(t,s)\,ds<\infty$ and $\e_1>0$ is a constant small enough (which may depend on $t$).
Now, according to the dominated convergence theorem, for every $t \in (0,1]$ and $x,y\in \R^d$,
\begin{equation}\label{t4-1-5}
\begin{split}
\lim_{\e \downarrow 0}\frac{1}{\e}\cdot
\bigg(\int_0^t \big(\phi(t+\e,s,x,y)-\phi(t,s,x,y)\big)\,ds\bigg)=
\int_0^t \partial_t \phi(t,s,x,y)\,ds.
\end{split}
\end{equation}

{(iii)} We obtain by \eqref{t4-1-3} and \eqref{t4-1-5}  that for every $t \in (0,1)$ and $x,y \in \R^d$,
\begin{align*}
&\lim_{\e \downarrow 0}\Bigg|\frac{1}{\e}
\cdot \Big(\int_0^{t+\e}\phi(t+\e,s,x,y)\,ds-
\int_0^{t}\phi(t,s,x,y)\,ds\Big)-q(t,x,y)-\int_0^t \partial_t \phi(t,s,x,y)\,ds\Bigg|\\
&\le \lim_{\e \downarrow 0}\Bigg|\frac{1}{\e}\cdot\Big(\int_t^{t+\e}
\phi(t+\e,s,x,y)\,ds\Big)-q(t,x,y)\Bigg|\\
&\quad+\lim_{\e \downarrow 0}\Bigg|\frac{1}{\e}\cdot\Big(\int_0^{t}
(\phi(t+\e,s,x,y)-\phi(t,s,x,y))\,ds\Big)-\int_0^t \partial_t \phi(t,s,x,y)\,ds\Bigg|\\
&=0.
\end{align*}
By the same way, it is not difficult to verify
\begin{align*}
&\lim_{\e \uparrow 0}\Bigg|\frac{1}{\e}
\cdot \Big(\int_0^{t+\e}\phi(t+\e,s,x,y)\,ds-
\int_0^{t}\phi(t,s,x,y)\,ds\Big)-q(t,x,y)-\int_0^t \partial_t \phi(t,s,x,y)\,ds\Bigg|=0.
\end{align*}
Hence, we have for any $t_0\in (0,1]$ and $x,y\in \R^d$,
\begin{equation*}
\partial_t\left(\int_0^{\cdot}\phi(\cdot,s,x,y)\,ds\right)(t_0)=q(t_0,x,y)+\int_0^{t_0}\partial_t \phi(t_0,s,x,y)\,ds.
\end{equation*}

Combining all estimates above with \eqref{t4-1-1}, \eqref{eq2-2} and \eqref{e:lkk},
we have for every $t \in (0,1]$ and $x,y \in \R^d$,
\begin{equation}\label{t4-1-6}
\begin{split}
\frac{\partial p(t,x,y)}{\partial t}
&= \LL^y p^y(t,x-y)+q(t,x,y)+\int_0^t \partial_t \phi(t,s,x,y)\,ds\\
&= \LL^y p^y(t,x-y)+( \LL^x- \LL^y)p^y(t,x-y)\\
&\quad +\int_0^t
\int_{\R^d}( \LL^x- \LL^z)p^z(t-s,x-z)\cdot q(s,z,y)\,dz\,ds\\
&\quad +\int_0^t \int_{\R^d} \LL^z p^z(t-s,x-z)q(s,z,y)\,dz\,ds\\
&= \LL^x p^y(t,x-y)+\int_0^t\int_{\R^d} \LL^x p^z(t-s,x-z)q(s,z,y)\,dz\,ds.
\end{split}
\end{equation}

Furthermore, by the same arguments for estimates of $\tilde J_1$, $\tilde J_2$ and $\tilde J_3$ above, we have
\begin{equation}\label{t4-1-7a}
\begin{split}
& \int_0^t\int_{\R^d}|  \LL^x p^z(t-s,x-z)q(s,z,y)|\,dz\,ds\\
&\p \int_0^t\int_{\R^d}\big|  \LL^z p^z(t-s,x-z)q(s,z,y)\big|\,dz\,ds+\int_0^t \int_{\R^d}|( \LL^x- \LL^z) p^z(t-s,x-z)q(s,z,y)|\,dz\,ds\\
&<\infty
\end{split}
\end{equation} and so
$$\int_0^t \LL^x(\phi(t,s,\cdot,y))(x)\,ds= \int_0^t\int_{\R^d}  \LL^x p^z(t-s,x-z)q(s,z,y)\,dz\,ds,$$ which in turn implies that \eqref{t4-1-6} is well defined.

{(iv)} For any $\e,\sigma>0$ small enough,
\begin{align*}
&\int_{\{|z|>\sigma\}}\delta_{\int_0^t\phi(t,s,\cdot,y)\,ds}(x;z)\cdot\frac{\K(x,z)}{|z|^{d+\A(x)}}\,dz\\
&=\int_{\e}^{t-\e} \int_{\{|z|>\sigma\}}\delta_{\phi(t,s,\cdot,y)}(x;z)\cdot\frac{\K(x,z)}{|z|^{d+\A(x)}}\,dz
\,ds +\int_{\{(0,\e)\cup(t-\e,t)\}}
\int_{\{|z|>\sigma\}}\delta_{\phi(t,s,\cdot,y)}(x;z)\cdot\frac{\K(x,z)}{|z|^{d+\A(x)}}\,dz
\,ds\\
&=:\wh J_1(\e,\sigma)+\wh J_2(\e,\sigma).
\end{align*}
Note that for the integral with respect to time variable in $\wh J_1(\e,\sigma)$, there is not any
singularity. According to the dominated convergence theorem,  we have
\begin{align*}
\lim_{\sigma \downarrow 0}J_1(\e,\sigma)&=\int_{\e}^{t-\e}
 \LL^x(\phi(t,s,\cdot,y))(x)\,ds=\int_{\e}^{t-\e}\int_{\R^d} \LL^x p^z(t-s,x-z)q(s,z,y)\,dz\,ds,
\end{align*}
Meanwhile, according to the proof of \eqref{t4-1-7a} and Fatou's lemma we arrive at
\begin{equation*}
\limsup_{\sigma \downarrow 0}|\wh J_2(\e,\sigma)|\le
\int_{\{(0,\e)\cup(t-\e,t)\}}| \LL^x(\phi(t,s,\cdot,y)(x)|\,ds.
\end{equation*}
Hence, combining all the estimates above with \eqref{e2-3} yields
\begin{align*}
&\Big| \LL^x\Big(\int_0^t \phi(t,s,\cdot,y)\,ds\Big)(x)-
\int_0^t  \LL^x(\phi(t,s,\cdot,y))(x)\,ds\Big|\\
&\le \limsup_{\sigma \downarrow 0}\bigg|\int_{\{|z|>\sigma\}}\delta_{\int_0^t\phi(t,s,\cdot,y)\,ds}(x;z)\cdot\frac{\K(x,z)}{|z|^{d+\A(x)}}\,dz
-\int_0^t  \LL^x(\phi(t,s,\cdot,y))(x)\,ds \bigg|\\
&\le 2\int_{\{(0,\e)\cup(t-\e,t)\}}| \LL^x(\phi(t,s,\cdot,y))(x)|\,ds.
\end{align*}
Then, letting $\e \rightarrow 0$, we know that for every $t \in (0,1)$ and $x,y\in \R^d$,
\begin{align*}
 \LL^x\Big(\int_0^t \phi(t,s,\cdot,y)\,ds\Big)(x)
&=\int_0^t  \LL^x(\phi(t,s,\cdot,y))(x)\,ds=
\int_0^t\int_{\R^d} \LL^x p^z(t-s,x-z)q(s,z,y)\,dz\,ds,
\end{align*}
which along with \eqref{t4-1-6} yields \eqref{t1-1-1} immediately.

{(2)} Suppose that $\K(x,z)$ depends on $z$. Then, using \eqref{l2-1-1},
\eqref{p3-1-2} and \eqref{p3-1-2a}, and following the same arguments above,
we can also show \eqref{t1-1-1} holds true. The details are omitted here.
\end{proof}

\subsection{Maximum principle and uniqueness}

Adopting the approach of \cite[Theorem 4.1]{CZ}, we will prove the following maximum principle for non-local parabolic PDEs associated with the operator $\LL$, which is crucial for the uniqueness of solution to the
corresponding Cauchy problem.

\begin{theorem}\label{t4-3}
Let $u \in C_b([0,1]\times \R^d)$ be the solution of the following equation
\begin{equation}\label{t4-3-2a}
{\partial_t u(t,x)}= \LL u(t,x),\quad (t,x)\in (0,1]\times \R^d.
\end{equation}
Suppose \begin{equation}\label{t4-3-1}
\lim_{t \downarrow 0}\sup_{x \in \R^d}\big|u(t,x)-u(0,x)\big|=0,
\end{equation}
and for every $x \in \R^d$, $t \mapsto  \LL u(t,x)$ is continuous in $t \in (0,1]$.
Assume that there exists a function $\T(x)\in (0,1)$ such that
$\inf_{x \in \R^d}\big(\T(x)+1-\A(x)\big)>0$ and for every $\e \in (0,1)$,
\begin{equation}\label{t4-3-2}
\sup_{t \in (\e,1)}\big|u(t,x)-u(t,x')\big|\le c_1(\e)|x-x'|^{\T(x)},\quad x,x'\in \R^d.
\end{equation}
Then for every $t \in (0,1]$,
\begin{equation*}
\sup_{x \in \R^d}u(t,x)\le \sup_{x \in \R^d}u(0,x).
\end{equation*}
\end{theorem}

\begin{proof}
Throughout the proof, the constant $C$ denotes a positive constant that is independent of $R$
and $x$  whose exact value may change from line to line. Since \eqref{t4-3-1} holds, it suffices to prove that
for any $\e \in (0,1)$ and $t\in (\e,1]$,
\begin{equation}\label{t4-3-5}
\sup_{x \in \R^d}u(t,x)\le \sup_{x \in \R^d}u(\e,x).
\end{equation}

For every $R>1$, we can choose a smooth cut-off function $l_R:\R^d \rightarrow \R$ such that
\begin{equation*}
l_R(x)=
\begin{cases}
1,\ \ \   & |x|\le R,\\
\in [0,1],\ & R<|x|<2R,\\
0,\ \ \  &|x|\ge 2R,
\end{cases}
\end{equation*} and
\begin{equation}\label{t4-3-6}
|\nabla l_R(x)|^2+|\nabla^2 l_R(x)|\le \frac{C}{R^2},\quad x\in \R^d.
\end{equation}
For every $R,\delta>0$ and $\e\in (0,1)$, define $$u_{R}^{\delta}(t,x)=
u(t,x)l_R(x)-(t-\e)\delta,\quad x\in \R^d, t\in (\e,1).$$ Then, it follows from the fact $\partial_t u(t,x)= \LL u(t,x)$ that
\begin{equation}\label{t4-3-3}
{\partial_t u_R^{\delta}(t,x)}=
\LL u_R^{\delta}(t,x)+g_R^{\delta}(t,x),
\end{equation}
where $g_R^{\delta}(t,x):= \LL u(t,x)\cdot l_R(x)-\LL(u\cdot l_R)(t,x)-\delta$.

Observe that
\begin{align*}
& \LL(u\cdot l_R)(t,x)- \LL u(t,x)\cdot l_R(x)-\LL l_R(x)\cdot u(t,x)\\
&=\int_{\R^d}(u(t,x+z)-u(t,x))\cdot (l_R(x+z)-l_R(x))
\frac{\K(x,z)}{|z|^{d+\A(x)}}\,dz\\
&\p \int_{\{|z|\le 1\}}
(u(t,x+z)-u(t,x))\cdot (l_R(x+z)-l_R(x))
\frac{1}{|z|^{d+\A(x)}}\,dz\\
&\quad +\int_{\{|z|> 1\}}
(u(t,x+z)-u(t,x))\cdot (l_R(x+z)-l_R(x))
\frac{1}{|z|^{d+\A(x)}}\,dz\\
&=:J_1+J_2.
\end{align*}
On the one hand, by \eqref{t4-3-2} and \eqref{t4-3-6},
\begin{align*}
J_1&\le \frac{C(\e)}{R}\int_{\{|z|\le 1\}}
\frac{|z|^{1+\T(x)}}{|z|^{d+\A(x)}}\,dz\le \frac{C(\e)}{R},
\end{align*}
where in the last inequality we used $\inf_{x \in \R^d}\big(\T(x)+1-\A(x)\big)>0$.
On the other hand, for every $0<\G< \A_1\wedge 1$,
\begin{align*}
J_2&\le 2\|u\|_{\infty}\|l_R\|_{\infty}^{1-\G}
\int_{\{|z|>1\}}\frac{|l_R(x+z)-l_R(x)|^{\G}}{|z|^{d+\A(x)}}\,dz\le \frac{C\|u\|_{\infty}}{R^{\G}}\int_{\{|z|>1\}}
|z|^{-d-\A(x)+\G}\,dz\le C\|u\|_{\infty}R^{-\G}.
\end{align*}
According to the argument above, it is easy to verify that
$$
|\LL l_R(x)|\le CR^{-\G}.
$$
Combining all the estimates above yields
\begin{equation}\label{t4-3-6a}
\begin{split}
g_R^{\delta}(t,x)\le C(\e,u)R^{-\G}-\delta.
\end{split}
\end{equation}

Now we are going to verify that for every fixed $R$ large enough,
\begin{equation}\label{t4-3-7}
\sup_{x \in \R^d}u_R^{\delta}(t,x)\le \sup_{x \in \R^d}u_R^{\delta}(\e,x),\quad t\in (\e,1].
\end{equation}
Suppose \eqref{t4-3-7} does not hold. Then for every large enough $R$, there exists
$(t_0,x_0) \in (\e,1]\times \R^d$ (which may depend on $R$, $\e$ and $\delta>0$) such that
\begin{equation}\label{t4-3-7a}
\sup_{(t,x) \in (\e,1)\times\R^d}u_R^{\delta}(t,x)=u_R^{\delta}(t_0,x_0).
\end{equation} Note that the existence of $(t_0,x_0)$ follows from the fact
that $u_R^{\delta}(t,x)=0$ for every $t \in (\e,1]$ and $|x|>2R$.
Therefore, by \eqref{t4-3-3}, we have for every $h \in (0,t_0-\e)$,
\begin{align*}
0\le \frac{u_R^{\delta}(t_0,x_0)-u_R^{\delta}(t_0-h,x_0)}{h}=
\frac{1}{h}\int_{t_0-h}^{t_0}  \LL u_R^{\delta}(s,x_0)\,ds+
\frac{1}{h}\int_{t_0-h}^{t_0}g_R^{\delta}(s,x_0)\,ds.
\end{align*}
Letting $h \downarrow 0$, we arrive at
\begin{equation}\label{t4-3-8}
0\le  \LL u_R^{\delta}(t_0,x_0)+g_R^{\delta}(t_0,x_0),
\end{equation} thanks to the assumption that the function $t \mapsto  \LL u(t,x)$ is continuous on $(0,1]$.
Furthermore, from \eqref{t4-3-7a} it is easy to see
\begin{equation*}
 \LL u_R^{\delta}(t_0,x_0)=\int_{\R^d}\big(u_R^{\delta}(t_0,x_0+z)+
u_R^{\delta}(t_0,x_0-z)-2u_R^{\delta}(t_0,x_0)\big)\frac{\K(x_0,z)}{|z|^{d+\A(x_0)}}\,dz\le 0.
\end{equation*}
Combining this with \eqref{t4-3-6a}, we get that for every $R>\big(\frac{2C(\e,u)}{\delta}\big)^{{1}/{\G}}$,
\begin{align*}
 \LL u_R^{\delta}(t_0,x_0)+g_R^{\delta}(t_0,x_0)\le -{\delta}/{2},
\end{align*}
which contradicts with \eqref{t4-3-8}.
Hence, the assumption above fails and so \eqref{t4-3-7} holds. Letting $R \rightarrow \infty$ in \eqref{t4-3-7},
we obtain \eqref{t4-3-5} immediately.
\end{proof}

Now, we are in a position to prove Theorems \ref{t1-1} and \ref{t1-2}.
\begin{proof}[Proofs of Theorems $\ref{t1-1}$ and $\ref{t1-2}$]
(i) We first assume that $\K(x,z)$ is independent of $z$.
We will prove that $p(t,x,y)$ constructed by \eqref{eq2-1} satisfies all the assertions.

According to Proposition \ref{t4-1}, \eqref{t1-1-1} holds. By
Propositions \ref{p3-2} and \ref{p3-3}, \eqref{t1-1-4} and \eqref{t1-1-4a} hold true.
Propositions \ref{t4-2} and \ref{t5-1} imply
\eqref{t1-1-7} and \eqref{t1-1-7a}, respectively. By Proposition \ref{p4-5},  we know that
$p:(0,1]\times\R^d\times \R^d\rightarrow \R_+$ is continuous. The lower bounds
\eqref{t1-1-4b}  will be shown in Corollary \ref{c-lower}, which is a consequence of Propositions
\ref{p4-3} and \ref{t4-4} in the next subsection.

\smallskip

(ii) According to  the proof of Proposition \ref{t4-1}, we can
obtain that for every $\e \in (0,1)$, there exist a constant
$C(\e)>0$ and a measurable function $\eta:\R^d \to  \R_+$ such that
$\int_{\R^d}\eta(y)\,dy<\infty$,

\begin{equation*}
\sup_{t \in (\e,1)}\Big|\frac{\partial p(t,x,y)}{\partial t}\Big|\le C(\e)\eta(x-y)
\end{equation*} and
\begin{equation*}\begin{split}
&\sup_{t \in (\e,1),\sigma>0}\bigg|\int_{\{|z|>\sigma\}}\big(p(t,x+z,y)+p(t,x-z,y)-2p(t,x,y)\big)
\frac{\K(x,z)}{|z|^{d+\A(x)}}\,dz\bigg|\le C(\e)\eta(x-y).\end{split}
\end{equation*}
Thus, by the dominated convergence theorem, it is easy to verify
that $\LL u_f (t, \cdot) (x)$ exists for each $t>0$ and $x\in \R^d$, $t\mapsto \LL u_f (t, \cdot) (x)$
is continuous,  and that \eqref{t1-1-2} holds.

\smallskip

(iii) We have by \eqref{eq2-1} that
\begin{align*}
u_f(t,x)-f(x)&=\int_{\R^d}p^y(t,x-y)\big(f(y)-f(x)\big)\,dy+
f(x)\left(\int_{\R^d}p^y(t,x-y)\,dy-1\right)\\
&\quad +\int_{\R^d} f(y)\int_0^t\int_{\R^d}p^z(t-s,x-z)q(s,z,y)\,dz\,ds\,dy\\
&=:J_1+J_2+J_3.
\end{align*}

For any $f\in C_{b,u}(\R^d)$ and for every $\e>0$, there exists
a constant $\delta:=\delta(\e)>0$ such that $|f(x)-f(y)|<\e$ for all $x,y \in \R^d$ with
$|x-y|\le \delta$. Thus,
\begin{align*}
|J_1|&\le \int_{\{|y-x|\le \delta\}}p^y(t,x-y)|f(y)-f(x)|\,dy+
\int_{\{|y-x|>\delta\}}p^y(t,x-y)|f(y)-f(x)|\,dy\\
&\le \e \int_{\R^d}\rho^{y,0}_1(t,x-y)\,dy+
C(\e)\|f\|_{\infty}\int_{\{|y-x|>\delta\}}\frac{t}{|x-y|^{d+\A_1}}\,dy\le C\e+C(\e,\delta)t,
\end{align*}
where in the last inequality we have used \eqref{l3-1-1}.
According to \eqref{p4-1-1}, it holds that $\lim_{t \downarrow 0}\sup_{x \in \R^d}|J_2|=0$.
Furthermore, by \eqref{p3-1-1} and \eqref{l3-3-1}, we obtain that for $\G<\B_0^*/{\A_2}$ with $\B^*\in (0,\B]\cap(0,\A_2/2)$,
\begin{align*}
|J_3|&\p \|f\|_{\infty}\Bigg[\int_{\R^d}
\int_0^t\int_{\R^d}\rho^{z,0}_1(t-s,x-z)
\Big(\rho^{y,0}_{{(\B_0^*/\A_2)}-\G}(s,z-y)+ \rho^{y,\B_0^*}_{-\G}(s,z-y)\Big)\,dz\,ds\,dy\Bigg]\\
&\p \|f\|_{\infty}\int_{\R^d}
\Big(\rho^{y,0}_{1+{(\B_0^*/\A_2)}-\G}+\rho^{y,\B_0^*}_{1-\G}\Big)(t,x-y)\,dy\p t^{{(\B_0^*/\A_2)}-\G}\|f\|_{\infty}.
\end{align*}
Combining all the estimates above together, we arrive at
\begin{equation*}
\lim_{t \downarrow 0}\sup_{x \in \R^d}\big|u_f(t,x)-f(x)\big|\le \e.
\end{equation*}
Since $\e$ is arbitrary, we know that \eqref{t1-1-3} holds.

\smallskip

(iv) Denote by $C_{b}^{\varepsilon}(\R^d)$ the set of bounded H\"{o}lder continuous functions, and by $C_{b}^{2,\varepsilon}(\R^d)$ the set of bounded twice differentiable functions whose second derivatives are uniformly H\"{o}lder continuous.  We first suppose that $f \in C_{b}^{2,\varepsilon}(\R^d)$. Let
\begin{equation*}
\tilde u_f(t,x):=f(x)+\int_0^t\int_{\R^d}p(s,x,y)\LL f(y)\,dy\,ds=f(x)+\int_0^t
u_{\LL f}(s,x)\,ds.
\end{equation*}
Since $\LL f \in C_{b}^{\varepsilon}(\R^d)$, it is easy to see that
\begin{equation*}
\frac{\partial \tilde u_f(t,x)}{\partial t}=\int_{\R^d}p(t,x,y)\LL f(y)\,dy
=u_{\LL f}(t,x),\quad t\in (0,1]
\end{equation*} and
\begin{equation*}
\lim_{t \downarrow 0}\sup_{x\in\R^d}|\tilde u_f(t,x)-f(x)|=0.
\end{equation*}

On the other hand, due to $\LL f \in C_{b}^{\varepsilon}(\R^d)$ again, it follows from \eqref{t1-1-2} that
$ \LL u_{\LL f}(t,x)=\frac{\partial u_{\LL f}(t,x)}{\partial t}$. Furthermore, following the
proof of Proposition \ref{t4-1}, we will get
\begin{align*}
\int_0^t \Big|\int_{\R^d} \LL^x p^z(t-s,x-z)\LL f(z)\,dz\Big|\,ds
&\le \int_0^t \Big|\int_{\R^d} \LL^z p^z(t-s,x-z)\LL f(z)\,dz\Big|\,ds\\
&\quad +\int_0^t \Big|\int_{\R^d}( \LL^x- \LL^z) p^z(t-s,x-z)\LL f(z)\,dz\Big|\,ds<\infty
\end{align*} and so $$\int_0^t \LL u_{\LL f}(s,x)\,ds$$ is well defined.
Hence,
\begin{align*}
\LL\tilde u_f(t,x)&=\LL f(x)+\int_0^t
 \LL u_{\LL f}(s,x)\,ds=\LL f(x)+\int_0^t \frac{\partial u_{\LL f}(s,x)}{\partial s}\,ds\\
&=\LL f(x)+u_{\LL f}(t,x)-\LL f(x)=u_{\LL f}(t,x),
\end{align*}
where in the third equality we used
$\lim_{t \downarrow 0}\sup_{x\in\R^d}|u_{\LL f}(t,x)-\LL f(x)|=0$, thanks to \eqref{t1-1-3}.
Therefore, both $u_f$ and $\tilde u_f$ are solutions of the following PDE
\begin{equation}\label{t1-1-11}
\begin{cases}
\frac{\partial u(t,x)}{\partial t}= \LL u(t,x),\\
\lim_{t \downarrow 0}\sup_{x \in \R^d}|u(t,x)-f(x)|=0.
\end{cases}
\end{equation}

Let $w_f(t,x):=u_f(t,x)-\tilde u_f(t,x)$. Then,  \eqref{t4-3-2a} and \eqref{t4-3-1} hold for $w_f$ with $w_f(0,x)\equiv 0$. At the same time, it is easy to verify from \eqref{eq2-1}
that the function $t \mapsto \LL w_f(t,x)$ is continuous on $(0,1]$, and that
\eqref{t1-1-7} implies \eqref{t4-3-2} holds for $w_f$. Thus, by Theorem \ref{t4-3}, we have
\begin{equation*}
w_f(t,x)\le w_f(0,x)=0,\quad (t,x)\in (0,1]\times \R^d.
\end{equation*}
Furthermore, applying the argument above to $-w_f$, we finally get that \begin{equation*}
w_f(t,x)\equiv 0,\quad  (t,x)\in (0,1]\times \R^d.
\end{equation*}
Therefore, $u_f(t,x)=\tilde u_f(t,x)$ for any $t \in (0,1]$ and $x \in \R^d$, which further implies that
\begin{equation*}
\int_{\R^d} \LL p(t,\cdot,y)(x)f(y)\,dy=\frac{\partial u_f(t,x)}{\partial t}
=\frac{\partial \tilde u_{f}(t,x)}{\partial t}=\int_{\R^d}p(t,x,y)\LL f(y)\,dy.
\end{equation*}
Thus, \eqref{t1-1-5a} holds. Observe that
\begin{equation*}
u_f(t,x)=\tilde u_f(t,x)=f(x)+\int_0^t u_{\LL f}(s,x)\,ds.
\end{equation*}
According to \eqref{t1-1-3}, we arrive at \eqref{t1-1-5b} immediately.
By the standard approximation procedure, we know that \eqref{t1-1-5a} and
\eqref{t1-1-5b} still hold for every $f \in C_{b,u}^2(\R^d)$.

\smallskip

(v) Let $u(t,x)=:\int_{\R^d}p(t,x,y)\,dy$. Then, according to \eqref{t1-1-2} and \eqref{t1-1-3}, we know that
$u$ satisfies the following equation
\begin{equation}\label{t1-1-11}
\begin{cases}
\frac{\partial u(t,x)}{\partial t}= \LL u(t,x),\\
\lim_{t \downarrow 0}\sup_{x \in \R^d}\big|u(t,x)-1\big|=0.
\end{cases}
\end{equation}
At the same time, $v(t,x)\equiv 1$ satisfies the equation \eqref{t1-1-11} above. Note that \eqref{t4-3-2} and the time continuity condition hold for both
$u(t,x)$ and $v(t,x)$. Then, using the same argument as in (iv)  and applying Theorem \ref{t4-3}, we obtain
\begin{equation*}
\int_{\R^d}p(t,x,y)\,dy=u(t,x)=1,\quad (t,x)\in (0,1]\times \R^d,
\end{equation*}
which is \eqref{t1-1-6}.

For every fixed $s \in (0,1)$ and $y \in \R^d$, we define $$u_{s,y}(t,x):=\int_{\R^d}p(t,x,z)p(s,z,y)\,dz.$$
 Again by \eqref{t1-1-2} and \eqref{t1-1-3}, the following equation holds for $u_{s,y}$:
\begin{equation}\label{t1-1-12}
\begin{cases}
\frac{\partial u_{s,y}(t,x)}{\partial t}= \LL u_{s,y}(t,x),\\
\lim_{t \downarrow 0}\sup_{x \in \R^d}|u_{s,y}(t,x)-p(s,x,y)|=0.
\end{cases}
\end{equation}
On the other hand, it is easy to verify that $v_{s,y}(t,x):=p(t+s,x,y)$ satisfies
\eqref{t1-1-12}, and \eqref{t4-3-2} and the time continuity condition hold for both
$u_{s,y}(t,x)$ and $v_{s,y}(t,x)$.  Following the procedure above and applying Theorem \ref{t4-3}, we arrive at
$$\int_{\R^d}p(t,x,z)p(s,z,y)\,dz=u_{s,y}(t,x)=v_{s,y}(t,x)=p(t+s,x,y),$$
which is \eqref{t1-1-5}.

\smallskip

(vi)
Suppose that $\wh p(t,x,y)$ is another
jointly continuous function  on $(0, 1]\times \R^d\times \R^d$
that is bounded for each $t>0$ and
 satisfies  \eqref{t1-1-7} and \eqref{t1-1-2a}-\eqref{t1-1-3}.
For every $f \in C_{b,u}(\R^d)$, let
$$\wh u_f(t,x):=\int_{\R^d}\wh p(t,x,y)f(y)\,dy \quad \hbox{and} \quad
\wh w_f(t,x):=u_f(t,x)-\wh u_f(t,x).
$$
 Then both $u_f$ and $\wh u_f$ satisfy \eqref{t1-1-2}. By the same argument above, we have
$\wh w_f(t,x)=0$ for each $(t,x)\in (0,1]\times \R^d$. This implies
that for every $f \in C_{b,u}(\R^d)$,
\begin{equation*}
\int_{\R^d} p(t,x,y)f(y)\,dy=\int_{\R^d} \wh p(t,x,y)f(y)\,dy.
\end{equation*}
Consequently, $p(t,x,y)=\wh p(t,x,y)$ for a.e.\ $y \in \R^d$ and hence for every $y\in \R^d$.

\smallskip

(vii) Suppose that $\K(x,z)$ depends on $z$.  The desired upper and lower estimates and the regularity can be proved in a similar way as before. On the other hand, the condition ${(\A_2/\A_1)}-1<\B_*/\alpha_2<{1}/{2}$ implies that
$\A_2-\A_1<{\A_1}/{2}<1$, which ensures that
\eqref{t4-3-2} holds with $u_f(t,x):=\int_{\R^d}p(t,x,y)f(y)\,dy$ for every
$f \in C_{b,u}(\R^d)$. Therefore, following the arguments in steps (i)-(vi), we can verify that \eqref{t1-1-5}-\eqref{t1-1-5b} hold in such case.
\end{proof}

\begin{proof}[Sketch of the Proof for Remark $\ref{rrr-000}$$(1)$] In Theorem \ref{t1-1} we assume that $\K(x,z)$ is independent of $z$.
According to its proof, the reason why we need such condition is only due to that this implies the
gradient estimate \eqref{p1-1-1}. Thus, for the upper bound estimates for $|\nabla p^y(t,x)|$
and  $|\nabla^2 p^y(t,x)|$, the time singularity factor $t^{-{1/\A(y)}}$ will not appear when $|x|$ is large and $t$ is small. This point is
crucial for estimates \eqref{l2-1-2} and \eqref{p3-1-1a-1}, which yield Theorem \ref{t1-1}.

Now, we turn to these two assumptions in Remark \ref{rrr-000}(1). According to \cite[Theorem 1.5]{KR} and  \cite[Corollary 7 and Theorem 21]{BGR}, the following gradient estimate \begin{equation*}
|\nabla^k p^y(t,x)|\p \frac{t}{|x|^{d+\A(y)+k}},\quad k=1,2
\end{equation*} holds for all $t \in (0,1]$ and $x\in\R^d$ with $|x|$ is large,
which also ensures that
the required estimates \eqref{l2-1-2} and \eqref{p3-1-1a-1} hold true.\end{proof}

\subsection{Lower bound estimates}
In this subsection, we will establish lower bound estimates for $p(t,x,y)$. The idea of the arguments below is inspired
by that in \cite[Subsection 4.4]{CZ}. Throughout this part, we will always
suppose that either of two conditions below is satisfied:
\begin{itemize}
\item [(1)] $\K(x,z)$ is independent of $z$.
\item [(2)] $\K(x,z)$ depends on $z$, and $({\A_2}/{\A_1})-1<\B_0^{**}/{\A_2},$ where $\B_0^{**}\in (0,\beta_0]\cap(0,\alpha_2/2).$
\end{itemize}

\begin{proposition}{\bf (On diagonal lower bounds)}\label{p4-3}
There exists a positive constant $c_1:=c_1(\A,\K)$ such that for all $t \in (0,1]$ and $x,y \in \R^d$ with $|x-y|\le 5
(t^{{1}/{\A(x)}}\vee t^{{1}/{\A(y)}})$,
\begin{equation}\label{p4-3-1}
p(t,x,y) \ge c_1t^{-{d}/{\A(x)}}.
\end{equation}
\end{proposition}
\begin{proof}
For simplicity, we only prove the case that $\K(x,z)$ depends on $z$ and $({\A_2}/{\A_1})-1<\B_0^{**}/{\A_2}$, since the other case can be
tackled similarly and easily.

First, according to \eqref{l3-2-8}, for any $t\in(0,1]$ and $x,y\in\R^d$ with $|x-y|\le 5
(t^{{1}/{\A(x)}}\vee t^{{1}/{\A(y)}})\le 5t^{{1}/{\A_2}}$,
\begin{equation}\label{p4-3-2}
t^{{1}/{\A(x)}}\asymp t^{{1}/{\A(y)}}.
\end{equation}
It is well known that for any $t>0$ and $x,y\in \R^d$, $$p^y(t,x-y)\asymp \frac{t}{(t^{{1}/{\A(y)}}+|x-y|)^{d+\A(y)}}.$$
Thus, for all $t\in (0,1]$ and $x,y\in\R^d$,
\begin{equation}\label{p4-3-3}
p^y(t,x-y)\succeq t^{-{d}/{\A(y)}}
\quad\textrm{ when }|x-y|\le 5(t^{{1}/{\A(x)}}\vee t^{{1}/{\A(y)}}).
\end{equation}

Second, \eqref{p3-1-2} and \eqref{l3-3-1}  yield that for all $t\in (0,1]$ and $x,y\in\R^d$ with $|x-y|\le 5t^{{1}/{\A_2}}$,
\begin{equation}\label{p4-3-4}
\begin{split}
&\int_0^t \int_{\R^d}
p^z(t-s,x-z)|q(s,z,y)|\,dz\,ds\\
&\p\int_0^t \int_{\R^d}\rho^{z,0}_1(t-s,x-z)
\cdot\Big(\rho^{y,0}_{({\B_0^{**}}/{\A_2})-\G_2 }+\rho^{y,\B_0^{**}}_{-\G_2 }\Big)(s,z,y)\,dz\,ds\\
&\p \rho^{y,\B_0^{**}}_{1-\G_2 }(t,x-y)+\rho^{y,0}_{1+({\B_0^{**}}/{\A_2})-\G_2 }(t,x-y)\\
& \p
|x-y|^{\B_0^{**}}t^{-{d}/{\A(y)}-\G_2 }+t^{-({d}/{\A(y)})+({\B_0^{**}}/{\A_2})-\G_2 }\p t^{-({d}/{\A(y)})+({\B_0^{**}}/{\A_2})-\G_2 },
\end{split}
\end{equation}
where in the first inequality $\G_2: =({\A_2}/{\A_1})-1+\G$ and the fourth inequality
follows from the fact that $|x-y|\le 5t^{{1}/{\A_2}}$.

The assumption $({\A_2}/{\A_1})-1<{\B_0^{**}}/{\A_2}$ ensures that we can choose
$\G>0$ small enough such that $\G_2 <{\B_0^{**}}/{\A_2}$. Combining \eqref{p4-3-4}
with \eqref{p4-3-3}, \eqref{eq2-1} and \eqref{p4-3-2}, we arrive at that there is a constant $t_0\in(0,1]$ such that for all $t\in (0,t_0]$ and $x,y\in \R^d$ with $|x-y|\le 5(t^{{1}/{\A(x)}}\vee t^{{1}/{\A(y)}})$,
\begin{equation}\label{pnte}
p(t,x,y)\q t^{-{d}/{\A(y)}}\asymp t^{-{d}/{\A(x)}}.
\end{equation}
Note that, according to the argument above, \eqref{pnte} still holds for all $t\in (0,t_0]$ and $x,y\in \R^d$ with $|x-y|\p
(t^{{1}/{\A(x)}}\vee t^{{1}/{\A(y)}})$. Furthermore, due to \eqref{p4-3-2}, for any $y,z\in \R^d$ with $|z-y|\le 5t^{{1}/{\A(y)}}$, we have $|x-z|\le |z-y|+|x-y|\p t^{{1}/{\A(x)}}$ for any $x\in \R^d$ with $|x-y|\p (t^{1/\A(x)}\vee t^{1/\A(y)})$, and so $$p(t,x,z)\q t^{-{d}/{\A(z)}}\asymp t^{-{d}/{\A(x)}},\quad t\in (0,t_0],$$ thanks to \eqref{p4-3-2} again. Therefore, according to the Chapman-Kolmogorov equation
\eqref{t1-1-5} and \eqref{pnte}, for every $t\in [t_0,2t_0]$ and $x,y\in \R^d$ with $|x-y|\le
5t^{{1}/{\A(x)}}$,
\begin{align*}
p(t,x,y)&=\int_{\R^d}p(t_0,x,z)p(t-t_0,z,y)\,dz\q (t-t_0)^{-{d}/{\A(y)}}\int_{\{|z-y|\le 5(t-t_0)^{{1}/{\A(y)}}\}}
p(t_0,x,z)\,dz\\
&\q (t-t_0)^{-{d}/{\A(y)}}t_0^{-{d}/{\A(x)}}(t-t_0)^{{d}/{\A(y)}}\q
t^{-{d}/{\A(x)}}.
\end{align*}
Iterating the arguments above $[{1}/{t_0}]+1$ times, we can
obtain \eqref{p4-3-1}.
\end{proof}

To consider off-diagonal lower bounds for the heat kernel $p(t,x,y)$,  we will make use of a strong Markov process, in particular the corresponding L\'evy system, associated with the operator $\LL$.
Note that, from the Chapman-Kolmogorov equation \eqref{t1-1-5} and the properties \eqref{t1-1-2} and \eqref{t1-1-5a}, it is standard to prove the following result. Since the proof is almost the same as \cite[Theorem 4.5]{CZ},
we omit it here.

\begin{proposition}\label{p4-4}
$(1)$ There is a strong Markov process $X:=\big((X_t)_{t \ge 0}; (\Pp_x)_{x\in \R^d}\big)$ such that for every
$f \in C_b^2(\R^d)$,
\begin{equation}\label{p4-4-0}
M_t^f:=f(X_t)-f(X_0)-\int_0^t \LL f(X_s)\,ds,\quad t \in (0,1]
\end{equation}
is a martingale with respect to the natural filtration $\mathscr{F}_t:=\sigma\{X_s,0\le s\le t\}$ under probability
measure $\Pp_x$ for all $x \in \R^d$. Moreover, $X$ has the strong Feller property.

$(2)$ For every non-negative measurable function $g:(0,1]\times \R^d \times \R^d \rightarrow \R_+$  vanishing
on $\{(s,x,y)\in (0,1]\times \R^d \times \R^d: x=y\}$ and any stopping time $T$, we have
\begin{equation}\label{p4-4-1}
\Ee_x\left(\sum_{s \le T\wedge 1}g(s,X_{s-},X_s)\right)=\Ee_x \bigg(\int_0^{T\wedge 1}\int_{\R^d}
g(s,X_s,y)J(X_s,y)\,dy\,ds\bigg),
\end{equation} where $$J(x,y):=\frac{\K(x,y-x)}{|y-x|^{d+\A(x)}},\quad x,y \in \R^d.$$
\end{proposition}

For any subset $D\subseteq \R^d$, define
\begin{equation*}
\sigma_{D}:=\inf\{t \ge 0: X_t \in D\},\qquad \tau_D:=\inf\{t \ge 0: X_t \notin D\}.
\end{equation*}
\begin{lemma}\label{l4-3}
There exist constants $R_1,A_0\in (0,1)$ such that for every
$r \in (0,R_1)$,
\begin{equation}\label{l4-3-1}
\Pp_x\big(\tau_{B(x,A_0 r)}\le r^{\A(x)}\big)\le {1}/{2},\quad x\in \R^d.
\end{equation}
\end{lemma}
\begin{proof}
Choose a function $f \in C_b^2(\R^d)$ such that
$f(0)=0$ and $f(x)=1$ for every $|x|\ge 1$. For each $r>0$ and $x \in \R^d$, define
$f_{r,x}(y):=f\big(\frac{y-x}{r}\big)$.
Then, by \eqref{p4-4-0}, for every $0<r<1$, $A>0$ and $x \in \R^d$,
\begin{equation}\label{l4-3-2}
\begin{split}
\Pp_x \big(\tau_{B(x,Ar)}\le r^{\A(x)}\big)&\le \Ee_x f_{Ar,x}\big(X_{\tau_{B(x,Ar)}\wedge r^{\A(x)}}\big)=\Ee_x \Big(\int_0^{\tau_{B(x,Ar)}\wedge r^{\A(x)}}\LL f_{Ar,x}(X_s)\,ds\Big).
\end{split}
\end{equation}
Observe that for every $y \in B(x,Ar)$ and $\lambda>1$,
\begin{align*}
|\LL f_{Ar,x}(y)|&\p
\int_{\{|z|\le \lambda r\}}|\delta_{f_{Ar,x}}(y;z)|\cdot \frac{dz}{|z|^{d+\A(y)}}
+\int_{\{|z|>\lambda r\}}|\delta_{f_{Ar,x}}(y;z)|\cdot \frac{dz}{|z|^{d+\A(y)}}\\
&\p \frac{\|\nabla^2 f\|_{\infty}}{(Ar)^2}
\int_{\{|z|\le \lambda r\}}|z|^{2-d-\A(y)}\,dz+\|f\|_{\infty}
\int_{\{|z|> \lambda r\}}|z|^{-d-\A(y)}\,dz\\
&\p \bigg(\frac{\|\nabla^2 f\|_{\infty}\lambda^{2-\A(y)}}{A^2}+
\frac{\|f\|_{\infty}}{\lambda^{\A(y)}}\bigg)\cdot r^{-\A(y)}.
\end{align*}
Hence, first taking $\lambda$ large enough and then $A$ large enough, we can find a constant $A_0>0$ such that
\begin{equation*}
|\LL f_{A_0r,x}(y)|\le \frac{1}{4}r^{-\A(y)}\quad \textrm{for all } y \in B(x,A_0r).
\end{equation*}

Since for every $y \in B(x,A_0r)$ and $r\in (0,1)$,
\begin{align*}
r^{-\A(y)}&=r^{-\A(x)}r^{-(\A(y)-\A(x))}\le r^{-|\A(x)-\A(y)|}r^{-\A(x)}\le r^{-C_1(A_0r)^{\B_0}}r^{-\A(x)}\le \exp(C_2|\log r|r^{\B_0})r^{-\A(x)},
\end{align*}
there exists a constant $r_0>0$ small enough such that $r^{-\A(y)}\le 2 r^{-\A(x)}$ for all
$r \in (0,r_0)$ and $y \in B(x,A_0r)$. Hence, we have for every $r \in (0,r_0)$ and $y\in B(x,A_0r)$,
\begin{equation*}
|\LL f_{A_0r,x}(y)|\le \frac{1}{2}r^{-\A(x)}.
\end{equation*}
Therefore, putting this estimate into \eqref{l4-3-2}, we obtain \eqref{l4-3-1}.
\end{proof}

We now show the following off-diagonal lower bound estimates for $p(t,x,y)$.

\begin{proposition}{\bf (Off-diagonal lower bound estimates)} \label{t4-4}
There exists a constant $c_1:=c_1(\A,\K)>0$ such that for every $t\in (0,1]$ and $x,y\in\R^d$ with
$|x-y|>5\max\{t^{{1}/{\A(x)}},t^{{1}/{\A(y)}}\}$
\begin{equation}\label{t4-4-1}
p(t,x,y)\ge \frac{c_1t}{|x-y|^{d+\A(x)}}.
\end{equation}
\end{proposition}
\begin{proof}
For any $t\in (0,1]$ and $x,y\in\R^d$, it holds that
\begin{equation}\label{t4-4-2}
\begin{split}
&\Pp_x\big(X_{\lambda t}\in B(y,t^{{1}/{\A(y)}})\big)\\
&\ge \Pp_x\big(\sigma_{B(y,t^{{1}/{\A(y)}}/2)}\le \lambda t;
\sup_{s \in (\sigma_{B(y,t^{{1}/{\A(y)}}/2)},\lambda t)}\big|X_s-X_{\sigma_{B(y,t^{{1}/{\A(y)}}/2)}}\big|\le t^{{1}/{\A(y)}}/2\big)\\
&\ge\Pp_x\big(\sigma_{B(y,t^{{1}/{\A(y)}}/2)}\le \lambda t ;\Pp_{X_{\sigma_{B(y,
t^{{1}/{\A(y)}}/2)}}}\big(\sup_{0\le s \le \lambda t}|X_s-X_0|\le
 t^{{1}/{\A(y)}}/2\big)\big)\\
&\ge \Pp_x\big(\sigma_{B(y,t^{{1}/{\A(y)}}/2)}\le \lambda t\big)
\cdot \inf_{z \in B(y,t^{{1}/{\A(y)}}/2)}\Pp_z\big(
\tau_{B(z,t^{{1}/{\A(y)}}/2)}>\lambda t\big),
\end{split}
\end{equation}
where in the second inequality we have used the strong Markov property.

Following the proof of \eqref{l3-2-8}, we know that
$t^{{1}/{\A(y)}}\ge C_1 t^{{1}/{\A(z)}}$ for every $z \in B(y,t^{{1}/{\A(y)}}/2)$.
Hence, there exists a constant $0<\lambda_0<1$ such that for all $0<\lambda<\lambda_0$
and $z \in B(y,t^{{1}/{\A(y)}}/2)$,
\begin{equation}\label{t4-4-3}
\begin{split}
&\Pp_z\big(\tau_{B(z,t^{{1}/{\A(y)}}/2)}>\lambda t\big)
\ge \Pp_z\big(\tau_{B(z,{C_1t^{{1}/{\A(z)}}}/{2})}>\lambda t\big)\ge {1}/{2},
\end{split}
\end{equation}
where the last inequality  follows from \eqref{l4-3-1}.
On the other hand, by the L\'evy system \eqref{p4-4-1}, for $\lambda>0$ small enough and for any $x,y\in \R^d$ with $
|x-y|>5\max\{t^{{1}/{\A(x)}},t^{{1}/{\A(y)}}\}$,
\begin{equation}\label{t4-4-4}
\begin{split}
 \Pp_x\big(\sigma_{B(y,t^{{1}/{\A(y)}}/2)}\le \lambda t\big)
&\ge \Pp_x\big(X_{\lambda t \wedge \tau_{B(x,t^{{1}/{\A(x)}})}}\in B(y,t^{{1}/{\A(y)}}/2)\big)\\
&=\Ee_x\bigg(\int_0^{\lambda t \wedge \tau_{B(x,t^{{1}/{\A(x)}})}}
\int_{B(y,t^{{1}/{\A(y)}}/2)}\frac{du}{|X_s-u|^{d+\A(X_s)}}\,ds\bigg)\\
&\q \lambda t^{1+({d}/{\A(y)})}\cdot \inf_{z \in B(x,t^{{1}/{\A(x)}})}\frac{1}{|x-y|^{d+\A(z)}}
\cdot \Pp_x\big(\tau_{B(x,t^{{1}/{\A(x)}})}\ge \lambda t\big)\\
&\q \frac{t^{1+({d}/{\A(y)})}}{|x-y|^{d+\A(x)}},
\end{split}
\end{equation}
where in the first inequality we used the fact that
$|x-y|\ge 5\max\{t^{{1}/{\A(x)}},t^{{1}/{\A(y)}}\}$ implies $B(x,t^{{1}/{\A(x)}})\cap
B(y,t^{{1}/{\A(y)}}/2)=\emptyset$, the second inequality follows from the fact that
$|z-u|\p |x-y|$ for any $z \in B\big(x,t^{{1}/{\A(x)}}\big)$
and $u \in B\big(y,t^{{1}/{\A(y)}}/2\big)$, and in the last inequality we used
\eqref{l4-3-1} and the fact that for any $x,y,z\in \R^d$ with $|x-y|\ge 5t^{{1}/{\A(x)}}$ and $z \in
B(x,t^{{1}/{\A(x)}})$ and for every $0<t\le 1$,
\begin{align*}
|x-y|^{-d-\A(z)}&\q |x-y|^{-d-\A(x)+C_2|x-z|^{\B_0}}\q t^{C_2t^{{\B_0}/{\A_1}}}|x-y|^{-d-\A(x)}\q |x-y|^{-d-\A(x)}.
\end{align*}

According to \eqref{t4-4-2}, \eqref{t4-4-3} and \eqref{t4-4-4}, we obtain
that for any $|x-y|>5t^{{1}/{\A(x)}}$ and $\lambda$ small enough,
\begin{equation}\label{t4-4-5}
\Pp_x\big(X_{\lambda t}\in B(y,t^{{1}/{\A(y)}})\big)\q
\frac{t^{1+({d}/{\A(y)})}}{|x-y|^{d+\A(x)}}.
\end{equation}
Hence, we arrive at
\begin{align*}
p(t,x,y)&=\int_{\R^d}p(\lambda t,x,z)p((1-\lambda)t,z,y)\,dz\\
&\ge \int_{B(y,t^{{1}/{\A(y)}})}
p(\lambda t,x,z)p((1-\lambda)t,z,y)\,dz\\
&\ge \inf_{|z-y|\le t^{{1}/{\A(y)}}}p((1-\lambda)t,z,y)
\int_{B(y,t^{{1}/{\A(y)}})}p(\lambda t,x,z)\,dz\\
&\q \inf_{|z-y|\le C_3t^{{1}/{\A(z)}}}p((1-\lambda)t,z,y)
\int_{B(y,t^{{1}/{\A(y)}})}p(\lambda t,x,z)\,dz\\
&\q t^{-{d}/{\A(y)}} \Pp_x\big(X_{\lambda t}\in B(y,t^{{1}/{\A(y)}})\big)\q \frac{t}{|x-y|^{d+\A(x)}},
\end{align*}
where the third inequality follows from the fact that $t^{{1}/{\A(y)}}\asymp t^{{1}/{\A(z)}}$ for every
$|y-z|\le t^{{1}/{\A(y)}}$, in the forth inequality we used \eqref{p4-3-1}, and in the last inequality we used \eqref{t4-4-5}.
Therefore, \eqref{t4-4-1} has been proved.
\end{proof}

According to Propositions \ref{p4-3} and \ref{t4-4}, we immediately get the following

\begin{corollary}\label{c-lower}
There exists a constant
$c_0:=c_0(\A,\K)>0$ such that for every $t\in (0,1]$ and $x,y \in \R^d$,
\begin{equation*}\label{t4-4-1a}
p(t,x,y)\ge
\frac{c_0 t }{\big(t^{{1}/{\A(x)}}+|y-x|\big)^{d+\A(x)}}.
\end{equation*}
\end{corollary}

\bigskip

\noindent \textbf{Acknowledgements.} \small The
research of Xin Chen is supported by National Natural Science Foundation of China (No.\ 11501361 and No.\ 11871338) and \lq\lq Yang Fan Project\rq\rq \, of Science and Technology Commission of Shanghai Municipality (No.\ 15YF1405900).
The research of Zhen-Qing Chen is  partially supported by Simons Foundation grant 520542 and a Victor Klee Faculty Fellowship at UW.
The research of Jian Wang is supported by the National Natural Science Foundation of China (Nos.\ 11522106 and 11831014), the Fok Ying Tung Education Foundation (No.\ 151002), the Alexander von Humboldt foundation, the Program for Probability and Statistics: Theory and Application (No.\ IRTL1704) and the Program for Innovative Research Team in Science and Technology in Fujian Province University (IRTSTFJ).

\end{document}